\documentclass{amsart}
\usepackage{fullpage}
\usepackage[utf8]{inputenc}
\usepackage{amsmath, graphicx}
\usepackage{color}
\usepackage{amsthm}
\usepackage{amssymb}
\usepackage{amscd}
\usepackage{mathtools}
\usepackage{hyperref}
\usepackage{makecell}
\usepackage{subcaption}
\usepackage{tikz}
\usepackage{comment}
\usepackage{cleveref}
\usepackage{bm}
\usepackage{enumitem} 
\usetikzlibrary{patterns, calc}
\tikzset{
  norm/.style     = {shape=circle, draw},
  blue/.style     = {shape=circle, draw, fill=blue!25},
  high/.style     = {shape=circle, draw, color=red},
  bluehigh/.style = {shape=circle, draw, color=red, fill=blue!25},
  red/.style      = {shape=circle, draw, fill=red!25},
  both/.style     = {shape=circle, draw, fill=violet!35},
  root/.style     = {node, bottom color=red!30},
  env/.style      = {treenode, font=\ttfamily\normalsize},
  dummy/.style    = {circle}
}
\tikzstyle{standard}=[circle, draw=black, fill=white, very thick]
\tikzstyle{standard2}=[circle, draw=black, fill=white, very thick]
\tikzstyle{blue2}=[circle, draw=black, fill=blue!25, very thick, inner sep=2pt]
\tikzstyle{small}=[circle, draw=black, fill=black, very thick, minimum size=4mm]
\tikzstyle{small2}=[circle, draw=black, fill=white, very thick, minimum size=4mm] 
\tikzstyle{special}=[circle, draw=red!60, fill=red!5, very thick, inner sep=2pt]
\usetikzlibrary{patterns, calc, decorations.pathreplacing}
\newtheorem{theorem}{Theorem}[section]
\newtheorem{lemma}[theorem]{Lemma}
\newtheorem{cor}[theorem]{Corollary}
\newtheorem{prop}[theorem]{Proposition}
\theoremstyle{definition}
\newtheorem{df}[theorem]{Definition}
\newtheorem{rem}[theorem]{Remark}

\newtheorem{ex}[theorem]{Example}

\newtheorem{conj}[theorem]{Conjecture}

\DeclareMathOperator{\F}{\mathcal{F}}
\DeclareMathOperator{\susp}{susp}
\DeclareMathOperator{\cone}{cone}

\DeclareMathOperator{\lk}{lk}
\DeclareMathOperator{\st}{st}
\DeclareMathOperator{\del}{del}
\DeclareMathOperator{\rank}{rank}
\DeclareMathOperator{\sgn}{sgn} 
\newcommand{\bbS}{\mathbb{S}}

\newcommand{\bbZ}{\mathbb{Z}}
\renewcommand{\tilde}{\widetilde}
\renewcommand{\hat}{\widehat}

\title{Topology of Cut Complexes of Graphs}

\author[M. Bayer]{Margaret Bayer}
\address{Margaret Bayer: University of Kansas, Lawrence, Kansas, USA}
\email{bayer@ku.edu}

\author[M. Denker]{Mark Denker}
\address{Mark Denker: University of Kansas, Lawrence, Kansas, USA}
\email{mark.denker@ku.edu}

\author[M. Jeli\'c Milutinovi\'c]{Marija Jeli\'c Milutinovi\'c}
\address{Marija Jeli\'c Milutinovi\'c: University of Belgrade, Serbia}
\email{marija.jelic@matf.bg.ac.rs}

\author[R. Rowlands]{Rowan Rowlands}
\address{Rowan Rowlands: University of Washington, Seattle, Washington, USA}
\email{rowanr@uw.edu}

\author[S. Sundaram]{Sheila Sundaram}
\address{Sheila Sundaram: University of Minnesota, Minneapolis, USA}
\email{shsund@umn.edu}

\author[L. Xue]{Lei Xue}
\address{Lei Xue: University of Michigan, Ann Arbor, Michigan, USA}
\email{leixue@umich.edu}

\begin{document}    


\begin{abstract}
 We define the \emph{$k$-cut complex} of a graph $G$ with vertex set $V(G)$ to be the simplicial complex whose facets are the complements of  sets of size $k$ in  $V(G)$ inducing disconnected subgraphs of $G$. This generalizes the Alexander dual of a graph complex studied by Fr\"oberg (1990), and Eagon and Reiner (1998).  We describe the effect of various graph operations on the cut complex, and study its shellability, homotopy type and homology for various families of graphs, including  trees, cycles, complete multipartite graphs, and the prism $K_n \times K_2$, using techniques from algebraic topology, discrete Morse theory and equivariant poset topology.
\end{abstract}
  
\maketitle

\section{Introduction}\label{sec:Intro}

This paper, a companion to \cite{BDJRSX-TOTAL2024},  deals with a class of graph complexes. 
In recent years there has been much interest in the topology of simplicial
complexes associated with graphs.  
A major contribution to this subject is the book \cite{JonssonBook2008} by Jonsson, who considers simplicial complexes defined on edge sets of graphs.  Simplicial complexes defined on vertex sets of graphs include clique complexes (see, e.g., \cite{Ivashchenko}), independence complexes (see, e.g., \cite{Meshulam}), neighborhood complexes (see, e.g., \cite{Lovasz}), and dominance complexes (see, e.g., \cite{Matsushita}). 
In this paper we introduce a new family of  simplicial  complexes associated to the vertex set of a graph, 
which we call \emph{cut complexes}. We consider only simple graphs. Our work is motivated by a famous theorem of Ralf Fr\"oberg \cite{Froberg1990} connecting commutative algebra and graph theory through topology.  We investigate  our new complexes in the spirit of Fr\"oberg's theorem, relating topological properties of the cut complex  to structural properties of the graph. 

For a field $\mathbb{K}$ and a finite simplicial complex $\Delta$ with vertex set $[n]=\{1,2,\ldots,n\}$, the Stanley--Reisner ideal of $\Delta$ is the ideal $I_\Delta$  of the polynomial ring $\mathbb{K}[x_1,\ldots , x_n]$ generated by the monomials $x_{i_1}\cdots x_{i_k}$ running over the inclusion-minimal subsets 
$\{i_1, \ldots , i_k\}$ of $[n]$ that are NOT faces of $\Delta$.  The Stanley--Reisner ring $\mathbb{K}[\Delta]$ is the quotient of the polynomial ring $\mathbb{K}[x_1,\ldots , x_n]$ by the ideal $I_\Delta$. 

For a graph $G$, 
the \emph{clique complex} $\Delta(G)$  is  the simplicial complex whose simplices are subsets of vertices of $G$, in which every pair of vertices is connected by an edge of $G$.  Fr\"oberg \cite{Froberg1990} characterized monomial ideals that have a \emph{$2$-linear resolution}, by first reducing to the case of square-free monomial ideals. The ideal $I_\Delta$ is generated by quadratic square-free monomials precisely when the simplicial complex $\Delta$ is $\Delta(G)$ for some graph $G$ (see \cite[Proposition~8]{EagonReiner1998}).  Hence Fr\"oberg's theorem can be stated as follows:

\begin{theorem}[Fr\"oberg {\cite{Froberg1990}, \cite[p. 274]{EagonReiner1998}}] \label{thm:Fr}
A Stanley--Reisner ideal $I_\Delta$ generated by quadratic square-free monomials has a $2$-linear resolution if and only if $\Delta$ is the clique complex $\Delta(G)$ of a chordal  graph $G$.
\end{theorem} 
Define the combinatorial \emph{Alexander dual} of a simplicial complex $\Delta$ \cite[p.188]{BrunsHerzog1997} on $n$ vertices to be 
\[\Delta^{\vee} \coloneqq \{F\subset [n]: [n]\setminus F\notin \Delta\}.\]  
The $i$th homology of $\Delta$ and the $(n-i-3)$th cohomology of $\Delta^{\vee}$ are isomorphic by Alexander duality in the sphere $\bbS^{n-2}$.

For a graph $G$, write $\Delta_2(G)$ for the  Alexander dual $\Delta(G)^\vee$ of the clique complex $\Delta(G)$.
The facets of  $\Delta_2(G)$  are the complements of independent sets of size 2 in $G$.

Eagon and Reiner's  reformulation \cite[Proposition 8]{EagonReiner1998} of Fr\"oberg's theorem includes the following equivalences.

\begin{theorem}\label{thm:Froberg} 
The graph $G$ is chordal  $\!\iff\!$ $\Delta_2(G)$ is shellable $\!\iff\!$ $\Delta_2(G)$ is vertex decomposable.
\end{theorem} 

Fr\"oberg's Theorem deals with ideals generated by monomials of degree two.  Consideration of higher degree monomials leads us to 
the following generalization of the simplicial complex $\Delta_2(G)$.  Let $k\ge 1$.  Define a complex whose facets are complements of sets $F$ of size $k$ in $G$ such that the induced subgraph of $G$ on  the vertex set $F$ is disconnected; we call this the $k$-cut complex of $G$, and denote it by $\Delta_k(G)$. This complex was first introduced in \cite{DenkerMastersReport}. A different generalization, the total $k$-cut complex $\Delta_k^t(G)$,  is treated in \cite{BDJRSX-TOTAL2024}.  The two notions coincide for $k=2$.

The paper is organized as follows.  In Sections~\ref{sec:Definitions} and~\ref{sec:Topology}, we collect basic definitions and background facts about simplicial complexes and posets. We begin Section~\ref{sec:Constructions} with a construction in Theorem~\ref{thm:anycomplexiscutcomplex},  showing that any pure simplicial complex can be realized as the $k$-cut complex of some graph $G$, which is in fact chordal. We examine the topology of the cut complex $\Delta_k(G)$ in Sections~\ref{sec:Constructions} and~\ref{sec:Mark:k=3-cut-complex}, and consider the effect of  properties of the graph $G$, and graph constructions such as join and disjoint union. For example, in analogy with Fr\"oberg's theorem,  Corollary~\ref{cor:chordal-Delta3-shell} asserts that graph chordality implies  shellability of the 3-cut complex $\Delta_3$. 
This result is the best possible: we show that for any $k\ge 4$, there is a chordal graph $G$ whose $k$-cut complex is not shellable, and is minimal with respect to this property. Section~\ref{sec:Euler-char} describes the face lattice of the cut complex, from which we can deduce information about its homology  and compute the Euler characteristic.  We also determine completely the homotopy type of $\Delta_2(G)$ for connected triangle-free graphs $G$; see Theorem~\ref{thm:Delta2nonchordalSMark}. In Section~\ref{sec:Families} we show that for many common families of graphs, the homotopy type of $\Delta_k(G)$ is a wedge of spheres in a single dimension.
Often there is a simplicial group action on the cut complex $\Delta_k(G)$, which in turn acts on the rational homology. This homology representation is particularly interesting in the case of complete multipartite graphs; see Section \ref{sec:new-multipartite-2022April23S}.  

Table~\ref{tab:Summary-table} summarizes our results for various families of graphs.

\begin{table}[htbp]
\caption{$k$-cut complexes for different graphs}
\label{tab:Summary-table}
\begin{center}

\scalebox{0.7}{
\begin{tabular}{|c|c|c|c|} \hline
Graph & Shellable? & Homotopy type and Betti numbers & Equivariant homology \\ \hline \hline
Complete bipartite, $K_{m,n}$ & \makecell{Theorem~\ref{thm:bipartite}: \\ Yes if and only if $m < k$} & Theorem~\ref{thm:BipartiteEquiv:k<=m<=n} & Theorem~\ref{thm:BipartiteEquiv:k<=m<=n} \\ \hline
Complete multipartite, $K_{m_1, \dotsc, m_r}$ & \makecell{Theorem~\ref{thm:MultipartiteShellable}: \\  
Yes if and only if $m_{r-1}<k$}
& Theorems~\ref{thm:MultipartiteCase4} \& \ref{thm:MultipartiteCase6} & Theorems~\ref{thm:MultipartiteCase4} \& \ref{thm:MultipartiteCase6} \\ \hline
Cycle, $C_n$ & \makecell{Theorem~\ref{thm:DaneCycleShellability}: \\ Yes if $k \geq 3$} & \makecell{Proposition~\ref{prop:MarijaDMT-Delta2cycle} ($k = 2$), \\ Proposition~\ref{prop:BettiNumberCycles} ($k \geq 3$)} & Theorem~\ref{thm:Cyclicgpaction-k-cut-complex-n-cycle} \\ \hline
Squared cycle, $W_n$ & \makecell{
Proposition~\ref{prop:wreath-graph-k+4-homotopy}: \\ No if $k = n-4$} & \makecell{Proposition~\ref{prop:MarijaDMT-Delta2-wreath-graph} ($k = 2$), \\ Proposition~\ref{prop:wreath-graph-k+4-homotopy} ($k = n-4$)} & \makecell{Proposition~\ref{prop:cyclic-gp-actionS-Delta_k(W_{k+4})}\\ ($k=2$ and $k=n-4$)}\\ \hline
Prism over clique, $K_n \times K_2$ & 
\makecell{Theorem~\ref{thm:conj-prismClique2022June21S-MarijaDMT}: \\ 
Yes if and only if $k > n$ (void complex)} & 
\makecell{
Theorem~\ref{thm:conj-prismClique2022June21S-MarijaDMT} ($k \le n$)} & \\ \hline
Tree & \makecell{Corollary~\ref{cor:Mark-trees}: \\ Yes for all $k \geq 2$} & Proposition~\ref{prop:cut-complex-trees} & \\ \hline
Threshold graph & \makecell{Corollary~\ref{cor:Mark-threshold-graph}: \\ Yes for all $k \geq 2$} & & \\ \hline
Connected \& triangle-free & \makecell{Fr\"oberg's theorem (Theorem~\ref{thm:Froberg}) \\ No for $k=2$, except trees}  & Theorem~\ref{thm:Delta2nonchordalSMark} ($k=2$)& \\ \hline
\end{tabular}
}
\end{center}
\end{table}

\noindent
\textbf{Acknowledgments.}  We thank the organizers of the 2021 Graduate Research Workshop in Combinatorics, where this work originated. We also thank Natalie Behague, Dane Miyata, and George Nasr for their early contributions to our project. Marija Jeli\'c Milutinovi\'c has been supported by the Project No.\ 7744592 MEGIC ``Integrability and Extremal Problems in Mechanics, Geometry and Combinatorics'' of the Science Fund of Serbia, and by the Faculty of Mathematics University of Belgrade through the grant (No.\ 451-03-47/2023-01/200104) by the Ministry of Education, Science, and Technological Development of the Republic of Serbia. Rowan Rowlands was partially supported by a graduate fellowship from NSF grant DMS-1953815.

We are also very grateful to the anonymous referees for their careful reading of the paper, for the expert advice and for the many valuable suggestions which we have implemented, including  the suggestion to investigate the group action on $\Delta_k(W_n)$ in Section~\ref{sec:Wreath-MarijaDMT-Rowan-Dane-Mark}.

\section{Definitions}\label{sec:Definitions}

General references for simplicial complexes, shellability and related concepts are \cite{BjTopMeth1995}, \cite{Koz2008},  \cite[Chapter II, Chapter III, Section 2]{RPSCCA1996} and \cite{WachsPosetTop2007}, and \cite{WestGraphTheory1996}  for graph theory. All graphs in this paper are simple (no loops and no multiple edges) and finite.

\begin{df}\label{def:simplicial-complex}  A \emph{simplicial complex} $\Delta$  is a collection of subsets such that 
\[\sigma\in \Delta \text{ and } \tau\subseteq \sigma \Rightarrow \tau \in \Delta.\] 
The elements of $\Delta$ are called its \emph{faces} or \emph{simplices}. If the collection of subsets is empty, i.e., $\Delta$ has no faces,  we call $\Delta$ the \emph{void complex}.  Otherwise $\Delta$ always contains the empty set as a face. 

The \emph{dimension of a face} $\sigma$, $\dim(\sigma)$, is one less than its cardinality; thus the dimension of the empty face is $(-1)$, and the 0-dimensional faces are the \emph{vertices} of $\Delta$. A \emph{$d$-face} or \emph{$d$-simplex} is a face of dimension $d$.  The maximal faces of $\Delta$ are called its \emph{facets}, and the maximum dimension of a facet is the \emph{dimension} $\dim(\Delta)$ of the simplicial complex $\Delta$. The dimension of the void complex is defined to be $-\infty$ (see \cite{Koz2008}).  

A (nonvoid) simplicial complex is \emph{pure} if all its facets have the same dimension, which is then the dimension of the complex.
We write $\Delta=\langle\mathcal{F}\rangle$ to denote the simplicial complex $\Delta$ whose set of facets is $\mathcal{F}$. In this paper all simplicial complexes will be \emph{finite}, i.e., the vertex set is finite.
\end{df}

We will be using the following constructions.
\begin{df}[{\cite{Koz2008}}] \label{defn:star-del-linkS}
Let $\Delta$ be a simplicial complex and $\sigma$ a face of $\Delta$.
\begin{itemize}
\item The \emph{link} of $\sigma$ in $\Delta$ is $\lk_{\Delta} \sigma \coloneqq \{\tau \in \Delta : \text{$\sigma\cap \tau = \emptyset$, and $\sigma\cup \tau \in \Delta$}\}$.
\item The (closed) \emph{star} of $\sigma$ in $\Delta$ is $\st_{\Delta} \sigma \coloneqq \{\tau \in \Delta : \sigma \cup \tau \in \Delta\}$.
\item The \emph{deletion} of $\sigma$ in $\Delta$ is $\del_{\Delta} \sigma \coloneqq \{\tau \in \Delta : \sigma \not\subseteq \tau\}$.
\end{itemize}
(Note that in the deletion, we are not removing proper faces of $\sigma$.)

Thus $\lk_\Delta(\emptyset)=\st_\Delta(\emptyset)=\Delta$, and $\del_\Delta(\emptyset)$ is the void complex.

For $v$ a vertex of $\Delta$, we also have the following useful facts (see \cite{Koz2008}):
\begin{equation}\label{eqn:st-del-lk}
\Delta=\st_\Delta(v)\cup \del_\Delta(v) \quad \text{and} \quad \lk_\Delta(v)=\st_\Delta(v)\cap \del_\Delta(v).
\end{equation}
\end{df}

\begin{df}[{\cite[Chapter III, Section 2]{RPSCCA1996}, \cite[Section~11.2]{BjTopMeth1995}}] \label{def:shelling}
An ordering $F_1,F_2,\dots,F_t$ of the facets of a simplicial complex $\Delta$ is a \emph{shelling} if, for every $j$ with $1<j\leq t$,
$$\left( \bigcup_{i=1}^{j-1}\langle F_i\rangle\right)\cap \langle F_j\rangle$$
is a simplicial complex whose facets all have cardinality $|F_j|-1$, where $\langle F_i\rangle$ is the simplex generated by the face $F_i$.

Equivalently, an ordering $F_1,F_2,\dots,F_t$ of the facets of $\Delta$ is a shelling if and only if for all $i,j$ such that $1\leq i<j\leq t$, there exists $k<j$ such that
$$F_i\cap F_j\subset F_k\cap F_j \quad \text{and} \quad |F_k\cap F_j| = |F_j|-1.$$
If the simplicial complex  $\Delta$ has a shelling, it is called \emph{shellable}.
\end{df}

In combinatorial topology, shellability is an important tool for determining the homotopy type of simplicial complexes, thanks to the following theorem of Bj\"orner.

\begin{theorem}[{\cite[Theorem~1.3]{Bj-AIM1984}}] \label{thm:shell-implies-homotopytype}
A pure shellable simplicial complex of dimension $d$ has the homotopy type of a wedge of spheres, all of dimension $d$. The complex is contractible if there are no spheres in the wedge.
\end{theorem}

Figure~\ref{fig:Delta2C5} depicts a triangulation of the M\"obius strip, a 2-dimensional surface which is homotopy equivalent to a one-dimensional circle. Thus the triangulation, a pure 2-dimensional complex, is not shellable.

\begin{rem} \label{rem:empty-complex-0-dim-complex}
By convention, the void complex is shellable. The complex whose only face is the empty set is vacuously shellable. The complex with a unique nonempty facet (i.e., a simplex) is shellable, and contractible. 
\end{rem}

\begin{df}\label{def:graph-disc-sep}
Let $G=(V,E)$ be a graph on $|V|=n$ vertices.  If $S$ is a subset of the vertex set $V$, we write $G[S]$ to denote the \emph{induced subgraph} of $G$ whose vertex set is $S$. 

For $k\ge 2$ define $D_k(G)\coloneqq\{S\subseteq V : \text{$G[S]$ is disconnected and $|S|=k$}\}$. We call the elements of $D_k(G)$ \emph{disconnected $k$-sets} of $G$, and the complements of elements of $D_k(G)$ \emph{separating $(n-k)$-sets} of $G$. When $k=1$, we define $D_1(G)$ to be the empty set.  

See Figure~\ref{fig:sepset}.
\end{df}

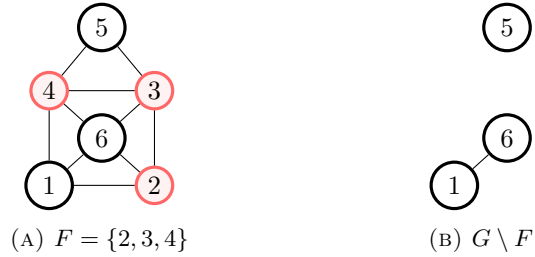
\begin{figure}[htb]
    \centering
    \begin{subfigure}{0.3\textwidth}
    \centering
    \begin{tikzpicture}[scale=0.7]
    \node[standard2] (n1) at (0,0) {1};
    \node[special] (n2) at (2,0) {2};
    \node[special] (n3) at (2,1.8) {3};
    \node[special] (n4) at (0,1.8) {4};
    \node[standard2] (n5) at (1,3) {5};
    \node[standard2] (n6) at (1,.9) {6};
    \draw (n1) -- (n2);
    \draw (n1) -- (n4);
    \draw (n3) -- (n2);
    \draw (n3) -- (n4);
    \draw (n3) -- (n5);
    \draw (n4) -- (n5);
    \draw (n1) -- (n6);
    \draw (n2) -- (n6);
    \draw (n3) -- (n6);
    \draw (n4) -- (n6);
    \end{tikzpicture}
    \caption{$F=\{2,3,4\}$}
    \end{subfigure}
    \begin{subfigure}{0.3\textwidth}
    \centering
    \begin{tikzpicture}[scale=0.7]
    \node[standard2] (n1) at (0,0) {1};
    \node[standard2] (n5) at (1,3) {5};
    \node[standard2] (n6) at (1,.9) {6};
    
    \draw (n1) -- (n6);
    \end{tikzpicture}
    \caption{$G\setminus F$}
    \end{subfigure}
    \caption{Example of a separating set}\label{fig:sepset}
\end{figure}

\begin{df}\label{def:cut-cplx}
Let $G=(V,E)$ be a graph on $|V|=n$ vertices, and let $k\ge 1$. Define the \emph{$k$-cut complex} of the graph $G$  to be the simplicial complex 
\[\Delta_k(G)\coloneqq\langle F\subseteq V\mid V\setminus F\in D_k(G)\rangle.\]
The facets of the cut complex $\Delta_k(G)$ are the separating sets of $G$ of size $(n-k)$. Thus the cut complex $\Delta_k(G)$ (if not void) has dimension $n-k-1$.  Equivalently, $\sigma$ is a face of the cut complex $\Delta_k(G)$ if and only if its complement $V\setminus \sigma$ contains a subset $S$ of size $k$ such that the induced subgraph $G[S]$ is disconnected.  Note the following inclusion, when $k\ge 2$:
\begin{equation} \label{eqn:inclusion}
\Delta_{k+1}(G)\subseteq \Delta_k(G),
\end{equation}
and the fact that the vertices of $\Delta_k(G)$ may be  a proper subset of the vertices of the graph $G$. 
See Figure~\ref{fig:Ex-Delta2} and Figure~\ref{fig:long-kayaks} for contrasting examples.
\end{df}

    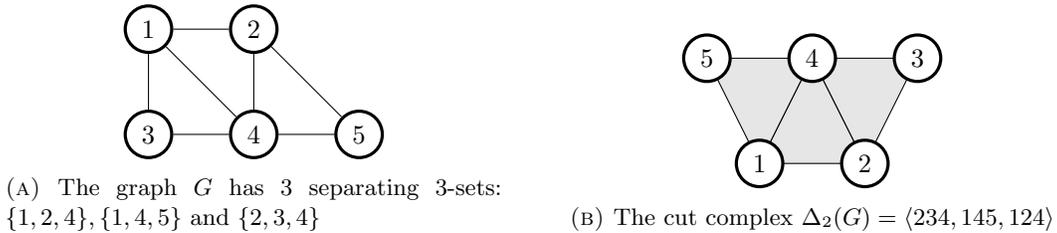
\begin{figure}[htb]
        \centering
        \begin{subfigure}{0.4\textwidth}
        \centering
        \begin{tikzpicture}[scale=0.7]
\node[standard] (node3) at (0,0) {3};
\node[standard] (node1) at (0,2) {1};
\node[standard] (node4) at (2,0) {4};
\node[standard] (node2) at (2,2) {2};
\node[standard] (node5) at (4,0) {5};

\draw (node1) -- (node3);
\draw (node1) -- (node2);
\draw (node1) -- (node4);
\draw (node2) -- (node4);
\draw (node2) -- (node5);
\draw (node3) -- (node4);
\draw (node4) -- (node5);
        \end{tikzpicture}
        \caption{The graph $G$ has 3 separating 3-sets: $\{1,2,4\}, \{1,4,5\}$ and $\{2,3,4\}$}
        \end{subfigure} \qquad
        \begin{subfigure}{0.4\textwidth}
        \centering
        \begin{tikzpicture}[scale=0.7]
\draw[fill=gray!20] (3,0) -- (4,2) -- (2,2) -- cycle;
\draw[fill=gray!20] (1,0) -- (2,2) -- (3,0) -- cycle;
\draw[fill=gray!20] (1,0) -- (2,2) -- (0,2) -- cycle;

\node[standard] (node5) at (0,2) {5};
\node[standard] (node1) at (1,0) {1};
\node[standard] (node4) at (2,2) {4};
\node[standard] (node2) at (3,0) {2};
\node[standard] (node3) at (4,2) {3};
        \end{tikzpicture}
        \caption{The cut complex $\Delta_2 (G)=\langle234,145,124\rangle$}
        \end{subfigure}
        \caption{2-cut complex of graph $G$}\label{fig:Ex-Delta2}
    \end{figure}

\begin{ex}\label{ex:Examples-cut-complex}  Let $G$ be a graph on $n$ vertices. We record some easy facts about cut complexes.  
\begin{enumerate}

\item $\Delta_k(G)$ is void if $k=1$ or $k>n$.

\item $\Delta_n(G)$ is 
$\begin{cases} \text{the void complex,} & \text{if $G$ is connected},\\
\text{the $(-1)$-dimensional complex $\{\emptyset\}$}, 
& \text{otherwise}. \end{cases}$

\item $\Delta_k(G)$ is void for $k\ge n-r+1$ if $G$ is $r$-connected: at least $r$ vertices must be removed to disconnect the graph, so $n-k\le r-1$.

\item  If $G$ is the complete graph $K_n$, then $\Delta_k(G)$ is void for all $k\ge 1$. 

\item If $G$ is the edgeless graph, then for $2\le k\le n-1$, $\Delta_k(G)$ is the $(n-k-1)$-skeleton of an $(n-1)$-dimensional simplex.

\end{enumerate}
\end{ex}

\begin{df}[{\cite{WestGraphTheory1996}}] \label{defn:chordal}
A graph is \emph{chordal} if it contains no induced cycle of size greater than 3.
\end{df}

Recall Eagon--Reiner's reformulation of Fr\"oberg's theorem from the Introduction (Theorem~\ref{thm:Froberg}) that $\Delta_2(G)$ is shellable if and only if $G$ is chordal. This is illustrated in Figures~\ref{fig:Ex-Delta2} and~\ref{fig:Delta2C5}. Figure~\ref{fig:Ex-Delta2} shows a chordal graph and its shellable 2-cut complex, whereas Figure~\ref{fig:Delta2C5} is an example of a non-chordal graph, with  nonshellable 2-cut complex.

\begin{figure}[htb]
\centering
\begin{subfigure}{0.4\textwidth}
\centering
\begin{tikzpicture}
\node[standard] (node1) at (0,3) {1};
\node[standard] (node2) at (1.5,1.7) {2};
\node[standard] (node3) at (1,0) {3};
\node[standard] (node4) at (-1,0) {4};
\node[standard] (node5) at (-1.5,1.7) {5};
%
\draw (node1) -- (node2);
\draw (node2) -- (node3);
\draw (node3) -- (node4);
\draw (node4) -- (node5);
\draw (node5) -- (node1);
\end{tikzpicture}
\caption{The cycle $C_5$}
\end{subfigure}
\begin{subfigure}{0.4\textwidth}
\centering
\begin{tikzpicture}
\draw[fill=gray!20] (0,0) -- (0.7,1.5) -- (1.4,0) -- (0,0);
\draw[fill=gray!20] (2.1,1.5) -- (0.7,1.5) -- (1.4,0) -- (2.1,1.5);
\draw[fill=gray!20] (2.1,1.5) -- (1.4,0) -- (2.8,0) -- (2.1,1.5);
\draw[fill=gray!20] (3.5,1.5) -- (2.8,0) -- (2.1,1.5) -- (3.5,1.5);
\draw[fill=gray!20] (3.5,1.5) -- (4.2,0) -- (2.8,0) -- (3.5,1.5);
\draw[very thick, color=red,->] (0,0) -- (.55*0.7,.55*1.5);
\draw[very thick, color=red] (.55*0.7,.55*1.5) -- (0.7,1.5);
\draw[very thick, color=red,->] (3.5,1.5) -- (3.5+.55*.7,1.5-.55*1.5);
\draw[very thick, color=red] (3.5+.55*.7,1.5-.55*1.5) -- (4.2,0);

\node[standard] (n5) at (0,0) {5};
\node[standard] (n2) at (0.7,1.5) {2};
\node[standard] (n4) at (1.4,0) {4};
\node[standard] (n1) at (2.1,1.5) {1};
\node[standard] (n3) at (2.8,0) {3};
\node[standard] (n5') at (3.5,1.5) {5};
\node[standard] (n2') at (4.2,0) {2};
%
\end{tikzpicture}
\caption{$\Delta_2 (C_5)=\langle{245}{,124}{,134} {,135}{,235}\rangle$}
\end{subfigure}
\caption{The 2-cut complex for $C_5$ is a M\"obius strip}
\label{fig:Delta2C5}
\end{figure}
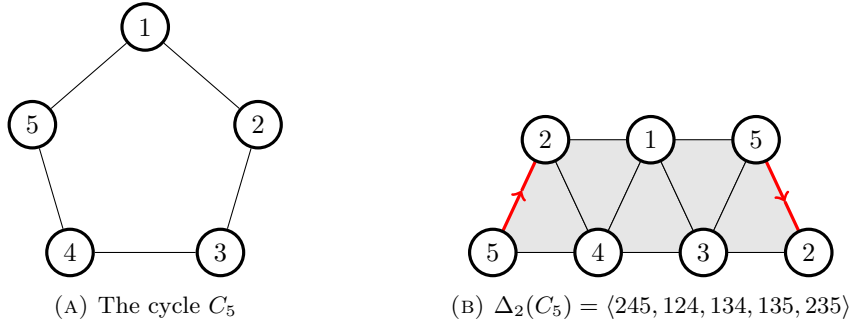

\section{Background from topology}\label{sec:Topology} 

The following facts are taken from \cite{BjTopMeth1995} and \cite{Koz2008}. Note that in this paper we assume a nonvoid simplicial complex always contains the empty set as a face. %
\begin{df}[{\cite[Section 9]{BjTopMeth1995}}] \label{def:join-cone-susp}
The \emph{join} of two simplicial complexes $\Delta_1$ and $\Delta_2$ with disjoint vertex sets is the complex
\[\Delta_1 * \Delta_2 \coloneqq \{\sigma\cup \tau : \sigma\in \Delta_1, \tau\in \Delta_2\}.\]
Thus the join $\Delta_1 * \Delta_2$ contains $\Delta_1$ and $ \Delta_2$ as subcomplexes.

The \emph{cone} over $\Delta$ and the \emph{suspension} of $\Delta$ are the complexes 
\[\cone(\Delta) \coloneqq \Delta * \Gamma_1, \quad \susp(\Delta) \coloneqq \Delta * \Gamma_2 = \Delta*\{u\}\cup\Delta*\{v\}, \]
where $\Gamma_1$ is the 0-dimensional simplicial complex with one vertex, and $\Gamma_2$ is the 0-dimensional complex with two vertices $u,v$. 

Let $\mathbb{S}^{d-1}=\{x\in \mathbb{R}^d: ||x||=1\}$ and $\mathbb{B}^{d}=\{x\in \mathbb{R}^d: ||x||\le 1\}$ denote respectively the $(d-1)$-sphere and the $d$-ball.  Then 
\[\mathbb{S}^{-1}=\emptyset, \quad \mathbb{S}^{0}=\{\text{two points}\}, \quad \mathbb{B}^{0}=\{\text{point}\}.\]
Furthermore, we have the following homeomorphism for spheres under the join operation:
\begin{equation}\label{eqn:join-of-spheres}
\mathbb{S}^{m}*\mathbb{S}^{n}\cong\mathbb{S}^{m+n+1}.
\end{equation}
\end{df}

We will consider reduced simplicial homology   \cite{Hatcher2002} over 
the integers $\mathbb{Z}$, or, for representation-theoretic purposes, over the rationals $\mathbb{Q}$, writing $\tilde{H}_i(\Delta)$ for the $i$th reduced homology of the simplicial complex $\Delta$. We record the following facts: 
$\tilde{H}_i(\emptyset)$ is nonzero if and only if $i=-1,$ in which case it is free of rank one. 
If $\Delta$ is nonvoid, 
$\tilde{H}_i(\Delta)$ is nonzero  only when $0\le i\le \dim(\Delta)$, and the reduced homology $\tilde{H}_0(\Delta)$ is free of rank one less than the number of connected components of $\Delta$. 

\begin{theorem}[{\cite{Hatcher2002}, \cite[Eqn~(9.12)]{BjTopMeth1995}}] \label{thm:Kunneth}
Let $\Delta_1$ and $\Delta_2$ be finite complexes. Assume at least one of $\tilde{H}_p(\Delta_1)$, $\tilde{H}_q(\Delta_2)$ over $\mathbb{Z}$ is torsion-free when $p+q=r-1$. Then the reduced homology of the join $\Delta_1*\Delta_2$ in degree $r$ is given by 
\[\tilde{H}_r(\Delta_1*\Delta_2) \cong \bigoplus_{p+q=r-1} \left(\tilde{H}_p(\Delta_1) \otimes \tilde{H}_q(\Delta_2)\right).\]
\end{theorem}

In particular, when the appropriate homology groups are torsion-free, the K\"unneth Theorem confirms the well-known group isomorphism 
\begin{equation}\label{hom:susp}
\tilde{H}_{r}(\susp(\Delta)) \cong  \tilde{H}_{r-1}(\Delta).
\end{equation}
Also note that if $\Delta_1$ has the homotopy type of a wedge of $\beta_p$ spheres of dimension $p$, and $\Delta_2$ has the homotopy type of a wedge of $\beta_q$ spheres of dimension $q$, then the join $\Delta_1 * \Delta_2$ has the homotopy type of a wedge of $(\beta_p\beta_q)$ spheres of dimension  $p+q+1$.

\begin{df}[{\cite{RPSCCA1996, RPSEC11997}}] \label{def:Betti-EulerChar}
Let $\Delta$ be a finite simplicial complex of dimension $d$, and let
\[\beta_i \coloneqq \rank\tilde{H}_i(\Delta, \mathbb{Z})=\dim_{\mathbb{Q}}\tilde{H}_i(\Delta, \mathbb{Q}), \ i\ge -1.\]
The $\beta_i$ are the \emph{(reduced) Betti numbers} of $\Delta$. Let $f_i$ be the number of $i$-dimensional faces of the $d$-dimensional complex $\Delta.$ Then $(f_0, f_1,\ldots f_d)$ is the \emph{$f$-vector} of $\Delta$. The \emph{Euler characteristic} of $\Delta$ is defined to be $\sum_{i\ge 0} (-1)^i f_i$.  The \emph{reduced Euler characteristic} $\mu(\Delta)$ of $\Delta$ is defined to be one less than the Euler characteristic: 
\[\mu(\Delta) = \bigg( \sum_{i\ge 0} (-1)^i f_i \bigg) - 1.\]
\end{df}

Letting $f_{-1}=1$ for the empty face,  we have the reduced Euler-Poincar\'e formula 
\[\mu(\Delta)=\sum_{i\ge -1} (-1)^i f_i =\sum_{i\ge -1} (-1)^i \beta_i.\]
\subsection{Poset Topology}\label{sec:posets}
In order to  determine the homotopy type of a simplicial complex $\Delta$, and in particular the group representation on the rational homology,  it is often helpful to work with the face lattice $\mathcal{L}(\Delta)$ of $\Delta$.  We recall some notions about posets and poset topology. See \cite{BjTopMeth1995}, \cite{WachsPosetTop2007} for more details. 

A poset $Q$ is bounded if it has a unique maximal element $\hat 1$ and a unique minimal element $\hat 0$. Let $\bar{Q}$ denote the proper part $Q\setminus\{\hat 0, \hat 1\}$ of a bounded poset $Q$. By the order complex of $Q$, we will always mean the simplicial complex, denoted $\Delta(\bar Q)$, of chains in the \emph{proper part} $\bar{Q}$ of $Q$. Let $Q_1$ and $Q_2$ be bounded posets.  It is well known (\cite{BjTopMeth1995}, \cite{WachsPosetTop2007}) that one has the homotopy equivalence 
\begin{equation}\label{eqn:prod-two-posets}
\Delta(\overline{Q_1\times Q_2}) \simeq\susp \left( \Delta(\bar{Q}_1) * \Delta(\bar{Q}_2) \right) \simeq\bbS^0*\Delta(\bar{Q}_1) * \Delta(\bar{Q}_2).
\end{equation}
In particular, if the order complex of $Q_i$ has the homotopy type of a wedge of $q_i$ spheres of dimension $d_i$, $i=1,2$, then the order complex of $Q_1\times Q_2$ is a wedge of $q_1 q_2$ spheres of dimension $d_1+d_2+2$.

Equation~\eqref{eqn:prod-two-posets} generalizes inductively to an $r$-fold product of bounded posets $Q_i$, $i=1, \ldots, r$, $r\ge 2$, and we record this homotopy equivalence for later use:
\begin{equation}\label{eqn:prod-r-posets}
\Delta(\overline{Q_1\times \dotsb \times Q_r}) \simeq\bbS^{r-2} * \left( \Delta(\bar{Q}_1) * \dotsb * \Delta(\bar{Q}_r) \right).
\end{equation}

Recall that the face lattice of a simplicial complex $\Delta$ is the poset of faces ordered by inclusion, with the empty face as the unique bottom element, and an artificially appended top element.  This makes the face lattice $\mathcal{L}(\Delta)$ of a finite simplicial complex $\Delta$ into a bounded poset. Its proper part is the poset consisting of the  nonempty faces of $\Delta$. The order complex of the proper part of the face lattice $\mathcal{L}(\Delta)$ is the barycentric subdivision of $\Delta$, and hence is homeomorphic to $\Delta$, and therefore has the same homotopy type. See, e.g.,\ \cite[Section 9.3]{BjTopMeth1995}. When $G$ is a finite group with a simplicial action on $\Delta$, the representation on the rational homology of $\Delta$ coincides with the representation on the homology of the face lattice. 

Let $B_p$ denote the Boolean lattice of subsets of a set with $p$ elements, and let $P(p,k)=B_p^{\le p-k}\cup \{\hat 1\}$ denote the truncated Boolean lattice, i.e., the 
subposet of $B_p$ consisting of subsets with at most $p-k$ elements, with an artificially appended top element $\hat 1$.  (This makes $P(p,k)$ a bounded poset with unique top and bottom elements.)  Note that $B_p$ is the face lattice of the boundary of a $(p-1)$-simplex, and $P(p,k)$, $0\le k\le p-1$, is the face lattice of the $(p-k-1)$-skeleton of a $(p-1)$-simplex. The poset $P(p,k)$ is an example of a rank-selected subposet, and there is a large literature on the topic of rank-selection in posets and group actions \cite{RPSGaP1982},  \cite{WachsPosetTop2007}.

It is a well-known fact that the poset $P(p,k)$ is lexicographically shellable (\cite{BjTAMS1980}, \cite{BjWachsTAMS1983}), with M\"obius number $\binom{n-1}{k-1}$, and hence its order complex is Cohen--Macaulay, and is homotopy equivalent to a wedge of $\binom{p-1}{k-1}$ spheres of dimension $p-k-1$. The homotopy type follows from the fact that the order complex of $P(p,k)$ is the barycentric subdivision of the $(p-k-1)$-skeleton of a $(p-1)$-simplex, and the latter is shellable by \cite[Theorem 2.9]{BjWachsI1996}.  See also \cite[Corollary 4.4, Theorem 8.1]{BjWachsTAMS1983}.

Our main tools from poset topology for determining homotopy type are as follows:

\begin{theorem}[Quillen fiber lemma {\cite{QFibLemma1978}, \cite[Theorem 15.28]{Koz2008}, \cite[Theorem 5.2.1]{WachsPosetTop2007}}] \label{thm:QuillenfiberLemma}
Let $P$ and $Q$ be bounded posets and $f:\bar{P}\rightarrow \bar{Q}$ a poset map.  If for all $q\in \bar{Q},$ the order complex of the fiber $f^{-1}_{\le}(q)\coloneqq\{p\in \bar{P}:f(p)\le q\}$ is contractible, then the map $f$ induces a homotopy equivalence of order complexes $\Delta(\bar{P})\simeq\Delta(\bar{Q})$. Furthermore, if $G$ is a finite group of automorphisms of $P$ and $Q$, and the poset map $f$ commutes with the action of $G$, then the homotopy equivalence is group equivariant and hence induces a $G$-module isomorphism in rational homology.
\end{theorem}

Recall (see \cite{RPSEC11997}) that the reduced Euler characteristic of a simplicial complex coincides with the M\"obius number of its face lattice.  The following result of Baclawski will be useful. 

\begin{theorem}[{\cite[Lemma 4.6]{BacEuJC1982}, 
\cite[Lemma 3.16.4]{RPSEC11997}}] \label{thm:Bac-mu}
If $P$ is a bounded poset and $Q$ is a subposet of $P$ containing $\hat 0, \hat 1$, then 
\[\mu(Q)-\mu(P) = \sum_{\substack{\hat 0<x_1<x_2<\dots<x_k<\hat 1 \\ k\ge 1, x_i\notin Q}} (-1)^k \mu_P(\hat 0, x_1) \mu_P(x_1,x_2) \dots \mu_P( x_k,\hat 1),\]
where the sum runs over all nonempty chains with elements not in $Q$.  Here $\mu_P$ denotes the M\"obius function of the poset $P$.

When $Q$ is a subposet obtained from $P$ by removing an antichain $\mathcal{A},$  this simplifies to 
\begin{equation}\label{eqn:mu-deleted-antichain}
\mu(Q)-\mu(P) = \sum_{\substack{\hat 0<x<\hat 1 \\ x\in \mathcal{A}}} (-1) \mu_P(\hat 0, x)  \mu_P(x,\hat 1).
\end{equation}
\end{theorem}
 
Finally, given a graph $G$ on $n$ vertices, and $k$ such that the cut complex $\Delta_k(G)$ is nonvoid, we have the inclusion of posets 
\begin{equation}\label{eqn:cut-complex-subposet-truncBoolean}
\mathcal{L}(\Delta_k(G))\subseteq P(n,k)=B_n^{\le n-k}\cup \{\hat 1\}.
\end{equation}

We now immediately obtain our first nontrivial homotopy result for a cut complex.  
\begin{prop}\label{prop:Deltak-edgeless-graphS}
Let $G$ be the edgeless graph with $n$ vertices. If $k\ge n$, the cut complex  $\Delta_k(G)$ is void. If $2\le k\le n-1$, then $\Delta_k(G)$ is shellable and 
\[ \Delta_k(G)\simeq\bigvee_{\binom{n-1}{k-1}} \mathbb{S}^{n-k-1}.\]
\end{prop}
\begin{proof} If the cut complex $\Delta_k(G)$ is nonvoid, then $2\le k\le n-1$ and  Example~\ref{ex:Examples-cut-complex}, Part (5) tells us that it is the $(n-k-1)$-skeleton of an $(n-1)$-simplex, hence shellable by \cite[Theorem~2.9]{BjWachsI1996}. The homotopy type and Betti number were given in the discussion above Theorem 3.4.
\end{proof}

In many cases, for example, when every vertex of the graph $G$ is a vertex of the cut complex $\Delta_k(G)$, the automorphism group  of $G$ induces a simplicial action (one that sends simplices to simplices) on the cut complex $\Delta_k(G)$, and hence in turn acts on the rational homology.
In subsequent sections, we will use the face lattice of the cut complex to determine this homology representation. 

Finally we will make use of the following.

\begin{prop}[{\cite[Chapter 21, (21.3)]{VINK2008}, \cite[Proposition~2.7]{BDJRSX-TOTAL2024}}] \label{prop:TopFact2ndIsoThm}
Let $X$ be a topological space with subspaces $A,B \subseteq X$ such that $X=A\cup B$, $A\cap B\neq \emptyset$, and $A, B$ are both closed subspaces or both open subspaces. Then the quotient map $A/(A\cap B) \rightarrow X/B$ of the inclusion $A\hookrightarrow X$ is a homeomorphism.
\end{prop}

\begin{prop}[{\cite[Proposition~0.17, Example~0.14]{Hatcher2002}}] \label{prop:quotient-by-contractible-homotopy}
Let $(X,A)$ be a CW pair consisting of a CW complex $X$ and a subcomplex $A$. 
\begin{enumerate}
\item If the subcomplex $A$ is contractible, then the quotient map $X\rightarrow X/A$ is a homotopy equivalence.
\item If $A$ is contractible in the complex $X$, then there is a homotopy equivalence 
\[X/A\simeq X \vee \susp(A).\]
\end{enumerate} 
\end{prop}

\section{Constructive Theorems}\label{sec:Constructions}

In this section we consider the effect of  some common graph operations on the cut complex.  We begin with the following universality property.

\subsection{Any simplicial complex is a cut complex}\label{sec:Natalie}

\begin{theorem}[Natalie Behague] \label{thm:anycomplexiscutcomplex}
Let $\Delta$ be any pure simplicial complex. There exists some $k$ and some chordal graph $G$ such that $\Delta$ is equal to the cut complex $\Delta_k(G)$. 
\end{theorem}
\begin{proof} The idea of the construction is as follows: start with a clique whose vertices correspond to the vertex set of the complex $\Delta$. For each facet of $\Delta$, add a vertex that is connected to every  vertex of that facet. If $n$ is the number of vertices of $\Delta$, $t$ is the number of facets of $\Delta$ and $d$ is the dimension of $\Delta$, the resulting graph $G$ has $(n+t)$ vertices, and our claim is that $\Delta_{n+t-(d+1)}(G)=\Delta$. Figure~\ref{fig:Mark-SimplicialComplex-to-CutComplex} illustrates the procedure.

\begin{figure}[htb]
    \centering
    \begin{subfigure}{0.4\textwidth}
    \centering
\begin{tikzpicture}[scale=1.1]
    \tikzstyle{facet}=[shape=rectangle, fill=white]
    \draw[draw=red!60, fill=blue!5] (0,0) -- (0,2) -- (1.5,1) -- cycle;
    \draw[draw=red!60, fill=blue!5] (0,0) -- (3,0) -- (1.5,1) -- cycle;
    \draw[draw=red!60, fill=blue!5] (0,2) -- (3,2) -- (1.5,1) -- cycle;
    \draw[draw=red!60, fill=blue!5] (3,0) -- (3,2) -- (1.5,1) -- cycle;
    \node (a)[facet] at (.5,1) {$F_1$};
    \node (b)[facet] at (1.5,.3) {$F_2$};
    \node (c)[facet] at (1.5,1.7) {$F_3$};
    \node (d)[facet] at (2.5,1) {$F_4$};
    \node[special] (node1) at (0,0) {1};
    \node[special] (node2) at (0,2) {2};
    \node[special] (node3) at (3,2) {3};
    \node[special] (node4) at (3,0) {4};
    \node[special] (node5) at (1.5,1) {5};
\end{tikzpicture}
\caption{A pure simplicial complex $\Delta$}
\end{subfigure}
\begin{subfigure}{0.4\textwidth}
\centering
\begin{tikzpicture}[scale=0.9]
        \node[coordinate] (n1) at (0,1.5) {1};
        \node[coordinate] (n2) at (.75,.85) {2};
        \node[coordinate] (n3) at (.5,0) {3};
        \node[coordinate] (n4) at (-.5,0) {4};
        \node[coordinate] (n5) at (-.75,.85) {5};
        \draw (n1) -- (n2);
        \draw (n1) -- (n3);
        \draw (n1) -- (n4);
        \draw (n1) -- (n5);
        \draw (n2) -- (n3);
        \draw (n2) -- (n4);
        \draw (n2) -- (n5);
        \draw (n3) -- (n4);
        \draw (n3) -- (n5);
        \draw (n4) -- (n5);
        \node[coordinate] (F1) at (0,3) {f1};
        \node[coordinate] (F2) at (-2,1.4) {f2};
        \node[coordinate] (F3) at (2,0.2) {f3};
        \node[coordinate] (F4) at (-1.3,-1.2) {f4};
        \draw[] (F1)--(n1);
        \draw[] (F1)--(n2);
        \draw[] (F1)--(n5);
        
        \draw[] (F2)--(n1);
        \draw[] (F2)--(n4);
        \draw[] (F2)--(n5);
        
        \draw[] (F3)--(n3);
        \draw[] (F3)--(n2);
        \draw[] (F3)--(n5);
        
        \draw[] (F4)--(n3);
        \draw[] (F4)--(n4);
        \draw[] (F4)--(n5);
        \node[blue2] (F1) at (0,3) {$f_1$};
        \node[blue2] (F2) at (-2,1.4) {$f_2$};
        \node[blue2] (F3) at (2,0.2) {$f_3$};
        \node[blue2] (F4) at (-1.3,-1.2) {$f_4$};
        
        \node[special] (n1) at (0,1.5) {1};
        \node[special] (n2) at (.75,.85) {2};
        \node[special] (n3) at (.5,0) {3};
        \node[special] (n4) at (-.5,0) {4};
       \node[special] (n5) at (-.75,.85) {5};
\end{tikzpicture}
\caption{A graph $G$ such that $\Delta=\Delta_6(G)$}
\end{subfigure}
\caption{The construction of Theorem~\ref{thm:anycomplexiscutcomplex}}
\label{fig:Mark-SimplicialComplex-to-CutComplex}
\end{figure}
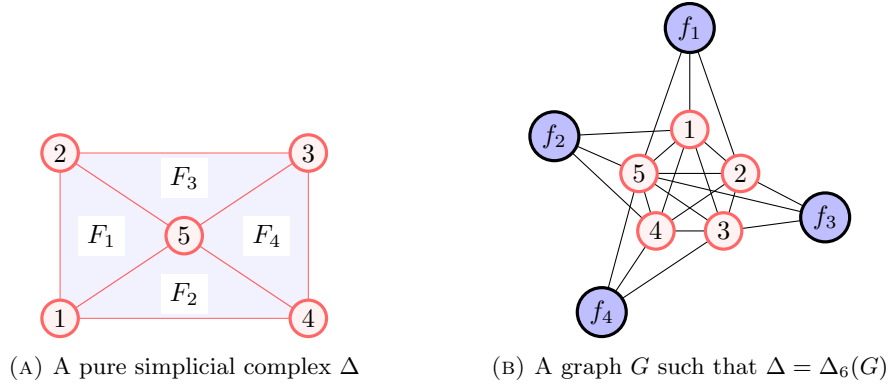

Let the vertices of $\Delta$ be labelled $v_1,v_2,\dots,v_n$ and let $\F=\{F_1,F_2,\dots,F_t\}$ be the set of facets of $\Delta$.  
First assume $t>1$.

Let the vertex set of $G$ consist of $n+t$ vertices labelled $u_1,u_2,\dots,u_n,f_1,f_2,\dots,f_t$. We define the edge set as follows: $u_iu_j$ is an edge for all $1\le i,j \le n$, and for all $1 \le i \le n$, $1 \le j \le t$, we have that $u_if_j$ is an edge if and only if vertex $v_i$ is contained in facet $F_j$ of $\Delta$. 

Let $k = n + t - (d+1)$. Consider the separating sets of $G$ of size $d+1 = n+t - k$. Since each $f_j$ has degree $d+1$, the neighbourhood  $\Gamma(f_j)$ is a separating set of size $d+1$ for each $1 \le j \le t$. 

In fact, these are the only separating sets of size $d+1$. Suppose $S$ is a set not containing $\Gamma(f_j)$ for any $j$. Given any $x,y \not\in S$, either $x = u_i$ for some $i$ or $x$ is adjacent to some $u_i \not\in S$. Similarly, either $y = u_j$ or $y$ is adjacent to some $u_j \not\in S$. Then either $u_i = u_j$ or $u_iu_j$ is an edge, and either way we have a path from $x$ to $y$ using only vertices outside the set $S$. Thus $S$ cannot be separating.

The facets of the cut complex $\Delta_k(G)$ are $\Gamma(f_j) = \{u_i : v_i \in F_j\}$ for $1 \le j \le t$. The vertices $f_j$ are not in any facets and so the complex $\Delta_k(G)$ has vertex set $\{u_1,u_2,\dots,u_n\}$. Identifying $u_i$ with $v_i$ for each $1 \le i \le n$ makes it immediately clear that $\Delta_k(G) = \Delta$.

Note that if $t=1$, then $n+t-(d+1)=1$, and we cannot identify $\Delta$ as $\Delta_1(G)$. 
In this case $n=d+1$.

If $n=1$, then $\Delta$
can be realized as the cut complex $\Delta_m(H)$ of the star graph $K_{1,m}$ 
for any $m\ge 2$. The center of the star constitutes the unique separating set of size 1, provided $m$ is at least 2.   

If  $n=d+1\ge 2$, then $\Delta=\{v_1,\dots,v_n\}$ is a $d$-simplex, and the above construction for $t>1$ is modified by taking the graph $G$ to be the complete graph on $n$ vertices, and then adding $n$ additional vertices $f_j$, each of which is adjacent to each of the vertices of the $K_n$. Then $\Delta$ will be the cut complex $\Delta_n(G)$ (note $G$ has $2n$ vertices),  since 
the $n$ vertices in the complete graph $K_n$ constitute the only separating set of size $n$.

Finally, this construction always produces a \emph{chordal} graph $G$. Since $G$ contains the clique on the vertex set $V$ of $\Delta$  as an induced subgraph, and the set of facets $\F$ of $\Delta$  gives an independent set of $G$, any cycle in $G$ not contained in the clique on $V$ must have at least one vertex $f$ corresponding to a facet $F\in \F$, and the two adjacent vertices must be vertices $u,v\in F$, by construction of $G$. Thus $\{u, f, v\}$ is a 3-cycle, and $G$ is chordal.
\end{proof}

The value of $k$ given by Theorem \ref{thm:anycomplexiscutcomplex}  may be very large. In particular the theorem gives us limited information on which simplicial complexes can arise as cut complexes for a specific fixed value of $k$. 

\begin{rem}\label{rem:torsionS}
This theorem gives us an example of a cut complex with torsion in its homology, in contrast to the other specific cut complexes studied in this paper.  Let $\Delta_{\mathbb{RP}^2}$ be the minimal triangulation of the projective plane $\mathbb{RP}^2$ with 6 vertices, 15 edges and 10 faces. The integral homology is ${H}_0=\bbZ$, ${H}_1=\bbZ_2$, ${H}_i=0$, $i\ge 2$. 
Theorem~\ref{thm:anycomplexiscutcomplex} constructs a graph $G$  on  $n+t=16$ vertices such that $\Delta_{\mathbb{RP}^2}=\Delta_{13}(G),$ since $k=(n+t)-(d+1)=13$. Hence $\Delta_{13}(G)$ is not torsion-free.
\end{rem}

\subsection{Facets of Cut Complexes}\label{sec:facets} 

\begin{df}\label{def:ridge}  A \emph{ridge} of a  pure simplicial complex is a face of dimension  one lower than the dimension of the complex. 
\end{df}

\begin{theorem}\label{theorem:facets}
    Let $k\ge 2$, $G=(V,E)$ a graph. Then the facets of $\Delta_{k+1} (G)$ are precisely  the ridges of $\Delta_k(G)$ that are contained in at least $k$ facets.
\end{theorem}
\begin{proof}
Let $K$ be a facet of $\Delta_{k+1} (G)$, so $V\setminus K$ is a disconnected $(k+1)$-set. Let $V\setminus K = \{ v_1, v_2, \dots, v_{k+1} \}$. We write 
\[ K = \bigcap_{i=1}^{k+1} \big(K \cup v_i\big). 
\]
If $G\setminus K$ has at least three components, then for all $v_i$, $G\setminus (K\cup v_i)$ has at least two components.  If $G\setminus K$ has two components, and each of those components has at least two vertices, then for all $v_i$, $G\setminus (K\cup v_i)$ has at least two components.  Finally, if $G\setminus K$ has two components and one of those components has just one vertex (say $v_1$), then the other component has $k\ge 2$ vertices, and for all $i\ge 2$, $G\setminus (K\cup v_i)$ has at least two components.  Therefore there are at least $k$ vertices $v_i$ such that $K\cup v_i$ is a facet of $\Delta_k(G)$.

On the other hand, suppose $K$ is a ridge of $\Delta_k(G)$ that is contained in at least $k$ facets. Assume $K$ is {\bf not} a facet of $\Delta_{k+1}(G)$; then $G\setminus K$ is connected, and so there exists a spanning tree $T$ of $G\setminus K$. Since $k\ge 2$, this tree has at least two leaves, say, $v$ and  $w$. Since removing a leaf of a spanning tree does not disconnect the graph, $G\setminus (K\cup v)$ and $G\setminus (K\cup w)$ are both connected. Since $|V\setminus K| = k+1$, there can only be at most $k-1$ facets of $\Delta_k(G)$ that contain $K$, which contradicts our assumption. 
\end{proof}

The above theorem implies that the faces of the $(k+1)$-cut complex of a graph $G$ are  completely determined by those of  the $k$-cut complex. 

\subsection{Links and Induced Subgraphs}\label{sec:links} 

\begin{lemma}\label{lem:Extendedlinks}
    Let $k\geq 2$, $G=(V,E)$ a graph, and $W\subseteq V$. Then 
    $\Delta_k (G\setminus W) = \lk_{\Delta_k(G)} W$ if $ W$ is a face of $\Delta_k (G)$, and is void otherwise.
\end{lemma}
\begin{proof} First note that if $W$ is not a face of $\Delta_k$, then $G\setminus W$ contains no disconnected induced subgraph with $k$ vertices, so $\Delta_k(G\setminus W)$ is void.  

Consider the case of a single vertex $v$ of $\Delta_k(G)$.  Let $|V|=n$.  A subset $F\subset V$ is a facet of $\Delta_k (G\setminus v)$ if and only if $v\notin F$, $|F| = n-k-1$, and $(V\setminus v )\setminus F$ is disconnected. Since $(V\setminus v )\setminus F  = V\setminus (v\cup F)$, this is equivalent to $v\notin F$, $|F\cup v| = n-k$ and $v\cup F\in \Delta_k(G)$, i.e., $F\cup v$ is a facet of $\Delta_k(G)$, and thus $F$ is a facet of $\lk_{\Delta_k(G)}\{v\}$. Equivalently, 
$\Delta_k (G\setminus v) = \lk_{\Delta_k(G)} \{v\}$. 

The result for an arbitrary face $W$ now follows by repeated application.
\end{proof}

\begin{prop}\label{prop:induced subgraphs}
      Let $k\geq 2$, $G=(V,E)$ a graph, and $W\subseteq V$. If the $k$-cut complex $\Delta_k(G)$ is shellable, so is $\Delta_k(G\setminus W)$.  Equivalently, 
      if $\Delta_k(G)$ is shellable for a graph $G$, then 
      $\Delta_k(H)$ is shellable for every induced subgraph $H$ of $G$. 
\end{prop}
\begin{proof} Immediate from Lemma~\ref{lem:Extendedlinks}, since the void complex is shellable, and shellability is preserved by the operation of taking links of faces; see 
 \cite[Proposition~10.14]{BjWachsII1997}, \cite[Theorem~3.1.5]{WachsPosetTop2007}. 
\end{proof}

\subsection{Disjoint Union of Graphs}\label{subsec:Marge-DisjUnion}

\begin{df}\label{defn:Disjunion}
If $G_1$, $G_2$ are graphs, their \emph{disjoint union} is the graph $G_1+G_2$ having vertex set equal to the disjoint union of the vertex sets of $G_1$ and $G_2$, and edge set equal to the disjoint union of the edge sets of $G_1$ and $G_2$.
\end{df}

\begin{theorem}\label{thm:MargeDisjunion}
Let $k\geq2$, and $G_1$, $G_2$ graphs.   Then $\Delta_k(G_1+G_2)$ is shellable if and only if $\Delta_k(G_1)$ and $\Delta_k(G_2)$ are shellable.
\end{theorem}
\begin{proof}
$\Delta_k(G_1+G_2)$ shellable implies $\Delta_k(G_1)$ and $\Delta_k(G_2)$ are shellable as $G_1$ and $G_2$ are induced subgraphs of $G_1+G_2$. 

Now suppose $\Delta_k(G_1)$ and $\Delta_k(G_2)$ are shellable. 
We construct a shelling of $\Delta_k(G_1+G_2)$.  The facets of
$\Delta_k(G_1+G_2)$ are the separating sets of size
$|V_1|+|V_2|-k$ of the graph $G_1+G_2$.  There are three types of facets.
Type 1: $(|V_1|+|V_2|-k)$-sets containing some, but not all vertices of each of
$V_1$ and $V_2$.
Type 2: sets of the form $V_1\cup A$, where $A$ is a $(|V_2|-k)$-subset of $V_2$
that disconnects $G_2$.
Type 3: sets of the form $B\cup V_2$, where $B$ is a $(|V_1|-k)$-subset of $V_1$
that disconnects $G_1$.

We give an ordering of the facets of Type 1, followed by the facets of Type 2,
followed by the facets of Type 3.
First, we know that a shelling order of the facets of Type 1 exists, because
the complex spanned by the facets of Type 1 is the
$(|V_1|+|V_2|-k-1)$-skeleton of the join of the boundary of the simplex with
vertex set $V_1$ with the boundary of the simplex with vertex set $V_2$.
See 
\cite[Theorem 2.9]{BjWachsI1996} for the skeleton result; the join result is due to  
\cite[Corollary~2.9]{Provan-Billera};  see also 
\cite[Remark 10.22]{BjWachsII1997}.
Also, we know that a shelling order of the facets of Type 2 exists, because
the complex spanned by the facets of Type 2 is the
$(|V_1|+|V_2|-k-1)$-skeleton of the join of the 1-facet complex consisting
of the $(|V_1|-1)$-simplex with vertex set $V_1$ with $\Delta_k(G_2)$.
Similarly, a shelling order of the facets of Type 3 exists, because
the complex spanned by the facets of Type 3 is the
$(|V_1|+|V_2|-k-1)$-skeleton of the join of the 1-facet complex consisting
of the $(|V_2|-1)$-simplex with vertex set $V_2$ with $\Delta_k(G_1)$.

It remains to show that the resulting
ordering of the facets of Type 1, followed by the facets of Type 2,
followed by the facets of Type 3, is a shelling order.
Clearly, the intersection of two facets of the same type satisfies the
shelling condition.  Now suppose that $F_i$ is a facet of Type 1 and $F_j$
is a facet of Type 2.  Then $(F_i\cap F_j)\cap V_1\subseteq F_i\cap V_1
\subseteq V_1\setminus \{x\}$ for some $x\in V_1$, and
$(F_i\cap F_j)\cap V_2\subseteq F_j\cap V_2
\subseteq V_2\setminus \{y,z\}$ for some $y,z\in V_2$ (since $k\ge 2$),
so $F_i\cap F_j\subseteq (V_1\setminus\{x\}) \cup (F_j\cap V_2 \cup \{y\})$,
which is a $((|V_1|-1)+(|V_2|-k+1))=(|V_1|+|V_2|-k)$-subset that is a facet
of $\Delta_k(G_1+G_2)$ of Type 1.
The same argument shows that the intersection of a facet of Type 1 with a
facet of Type 3 is contained in a facet of Type 1.
Finally, suppose $F_i$ is a fact of Type 2 and $F_j$ is a facet of Type 3.
Then for some set $A\subset V_1$ of size $|V_1|-k$ and some set
$B\subset V_2$ of size $|V_2|-k$, $F_i\cap F_j=A\cup B$ with
$|A\cup B|=|V_1|+|V_2|-2k$.  Choose, say, $k-1$ elements of $V_1\setminus A$
and 1 element of $V_2\setminus B$ to extend $A\cup B$ to a
$(|V_1|+|V_2|-k)$-element subset of $V_1\cup V_2$, not containing all of
$V_1$ or all of $V_2$.  This is then a facet of Type 1 that contains
$F_i\cap F_j$.

Finally, we can conclude that a shelling order of Type 1 facets, followed by a
shelling order of Type 2 facets, followed by a shelling order of Type 3 facets
is a shelling of $\Delta_k(G_1+G_2)$.
\end{proof}

\subsection{Join of Graphs}\label{subsec:Mark-Joins}
   \begin{df}
    Given graphs $G_1=(V_1,E_1)$ and $G_2=(V_2,E_2)$ on disjoint vertex sets, their \emph{join} $G_1*G_2$ is their disjoint union with the set of all edges between $V_1$ and $V_2$ added as well.
    \end{df}

\begin{theorem}\label{thm:Mark-Joins}
Let $k\geq2$, and $G_1$, $G_2$ graphs. Then $\Delta_k(G_1*G_2)$ is shellable if and only if between $\Delta_k(G_1)$ and $\Delta_k(G_2)$, one is shellable and the other is the void complex.   
\end{theorem}
\begin{proof}
Let $k\geq2$, $G_1=(V_1,E_1)$, $G_2=(V_2,E_2)$, $\F$ the set of facets of $\Delta_k(G_1*G_2)$, and $\F_1, \F_2$ the facets of $\Delta_k(G_1)$ and $\Delta_k(G_2)$, respectively. Furthermore let $\F_1^+=\{F\cup V_2\mid F\in \F_1\}$ and $\F_2^+=\{F\cup V_1\mid F\in \F_2\}$. We consider $D_k(G_1*G_2)$. If $S$ induces a disconnected subgraph of $G_1*G_2$, then $S\subseteq V_1$ or $S\subseteq V_2$. 
So $D_k(G_1*G_2)=D_k(G_1) \sqcup D_k(G_2)$, and $\F=\F_1^+\sqcup \F_2^+$. Let $F_1\in \F_1^+$ and $F_2\in \F_2^+$; then $F_1^c\in D_k(G_1)$ so $F_1^c\subseteq V_1$, and similarly $F_2^c\subseteq V_2$. Suppose $F_1\cap F_2\subseteq F\in\F$; then $F^c\subseteq (F_1\cap F_2)^c= F_1^c\cup F_2^c$. However $F^c$ has size $k$ and is contained entirely in $V_1$ or $V_2$. So $F^c=F_1^c$ or $F^c=F_2^c$, and thus $F=F_1$ or $F=F_2$; so the only facets to contain $F_1\cap F_2$ are $F_1$ and $F_2$. However, $|F_1\cap F_2|=|V_1\cup V_2|-2k=|F_1|-k<|F_1|-1$. Thus if both $\F_1$ and $\F_2$ are non-empty, then $\Delta_k(G_1*G_2)$ is not shellable. 
Assume without loss of generality that $\F_2=\emptyset$;
then $\F=\F_1^+$ and $\Delta_k(G_1*G_2)$ is shellable if and only if $\Delta_k(G_1)$ is shellable.
\end{proof}
In many ways the join of graphs behaves like the opposite of the disjoint union of graphs; see  Theorem~\ref{thm:MargeDisjunion}. It destroys cut complex shellability except in very specific boundary conditions. However, the following special case 
is worth highlighting.
    \begin{cor}\label{join-one}
    Given a graph $G$, let $G*1$ denote the join of the graph $G$ with the graph consisting of a single vertex. Then $\Delta_k(G*1)$ is shellable if and only if $\Delta_k(G)$ is shellable.
    \end{cor}

 \begin{df}\label{defn:dominating-vertex} A vertex $v$ of a graph $G$ is a \emph{dominating} vertex if it is connected by an edge to every other vertex of $G$.
\end{df}
Thus the new vertex in $G*1$ is a dominating vertex.
    \begin{df}\label{defn:threshold-graph}
    A \emph{threshold graph} is a graph constructed from a single vertex by adding a sequence of isolated and dominating vertices.
    \end{df}
Since the cut complex of the graph on one vertex is the void complex for all $k \geq 2$, it is trivially shellable. Hence Corollary~\ref{join-one} and Theorem~\ref{thm:MargeDisjunion}
immediately give us the following corollary. 
\begin{cor}\label{cor:Mark-threshold-graph}
If $G$ is a threshold graph, then $\Delta_k(G)$ is shellable for all $k\geq 2$.
\end{cor}

We now give a precise structural description of  the cut complex of the join of two graphs, which will allow us to determine the homotopy type from that of the cut complexes of the individual graphs.  Passing  to the face lattice gives an alternative proof of the homotopy equivalence, and also shows that it is group-equivariant, a fact we will need in Section~\ref{sec:new-multipartite-2022April23S}.    
We begin with a more general proposition.
We write $\langle V\rangle$ for the simplex on the vertex set $V$. 

\begin{prop}\label{prop:general-joinS2022Sept15}  For $i=1,2$ let  $\Delta_i$ be a simplicial complex on a finite vertex set $V_i$, such that $V_1\cap V_2=\emptyset$. %
Define a new simplicial complex $\Delta$ by 
\[\Delta\coloneqq (\Delta_1 * \langle V_2\rangle)\cup (\langle V_1\rangle*\Delta_2).\]
\begin{enumerate}
\item
There is a homotopy equivalence 
$\Delta \simeq \susp(\Delta_1*\Delta_2).$
\item 
Moreover, there is a group-equivariant poset map from the product of face lattices 
\[\mathcal{L}(\Delta_1)\times \mathcal{L}(\Delta_2)\]
to the face lattice of the simplicial complex $\Delta$, thereby 
  inducing a group-equivariant homotopy equivalence of the respective order complexes.
That is, if  $H_i$ is a group acting simplicially on $\Delta_i$, $i=1,2$,  then there is a group equivariant homotopy equivalence 
\[\susp(\Delta_1*\Delta_2)\simeq_{H_1\times H_2} \Delta .\]
\end{enumerate}
\end{prop}
\begin{proof} For Part (1): 

For simplicity write $A=\Delta_1 * \langle V_2\rangle$ and $B=\langle V_1\rangle*\Delta_2$.  Since $\Delta=A\cup B$, 
 Proposition~\ref{prop:TopFact2ndIsoThm} gives 
a homotopy equivalence 
\[\Delta/A\simeq B/(A\cap B).\]
Note that $A$ and $B$ are contractible. Hence by Proposition~\ref{prop:quotient-by-contractible-homotopy}, 
the space on the left is homotopy equivalent to $\Delta$, and the space on the right is homotopy equivalent to $\susp (A\cap B)=\susp(\Delta_1*\Delta_2 )$.

For Part (2): 

Let $H_i$ be a group acting simplicially on $\Delta_i$, $i=1,2$.  Recall that 
the face lattice $\mathcal{L}(\Delta_i) $ has an artificially appended top element $\hat 1_{\mathcal{L}(\Delta_i)}.$ 
We claim that there is a group-equivariant poset  map 
\[\phi:   \mathcal{L}(\Delta_1)\times \mathcal{L}(\Delta_2) \rightarrow 
   \mathcal{L}( \Delta)   \]
which induces a $(H_1\times H_2)$-homotopy equivalence of order complexes.  More precisely, 
 for faces $\sigma_i\in \mathcal{L}(\Delta_i)$, define $\phi$ to be the map sending 
 \begin{align*} &(\sigma_1,\sigma_2) \mapsto \sigma_1\sqcup \sigma_2,\\
&(\sigma_1,\hat 0)\mapsto \sigma_1\sqcup \emptyset =\sigma_1,\\
&(\sigma_1, \hat 1_{\mathcal{L}(\Delta_2)})\mapsto \sigma_1\sqcup V_2,\\
&(\hat 0, \sigma_2)\mapsto \emptyset \sqcup \sigma_2=\sigma_2,\\
&(\hat 1_{\mathcal{L}(\Delta_1)}, \sigma_2)\mapsto V_1\sqcup \sigma_2.
\end{align*}
  One checks that $\phi$ is a poset map, which clearly commutes with the action of $H_1\times H_2$.

The facets of the  complex $\Delta$ are of the form $F_1\sqcup V_2$,  $V_1\sqcup F_2$, where $F_i$ is a facet of $\Delta_i$, $i=1,2$. 
Hence $\sigma=\alpha_1\sqcup\alpha_2$ is a face of $\Delta$ 
if and only if $\alpha_i\subseteq V_i,\ i=1,2$ and $\alpha_i\in \Delta_i$ for at least one $i$.

We show that the fibers
\begin{equation*}
\phi^{-1}_{\ge}(\alpha_1\sqcup\alpha_2) = \left\{ (\sigma_1,\sigma_2) \ne \big( \hat 1_{\mathcal{L}(\Delta_1)}, \hat 1_{\mathcal{L}(\Delta_2)} \big) : \text{either $\sigma_i=\hat 1_{\mathcal{L}(\Delta_i)}$ or $\sigma_i\in\Delta_i$ and $\sigma_i\supseteq \alpha_i, i=1,2$} \right\}
\end{equation*}
have a unique minimal element, and hence are contractible.

Suppose  $\alpha_i\in \Delta_i$ for both $i=1,2$. Then clearly $(\alpha_1,\alpha_2)$ is in the fiber, and is its unique minimal element. 

Otherwise $\alpha_i\in \Delta_i$ for only one of $i=1,2$, say for $i=2$. 
Since $\alpha_1\notin \Delta_1$, if $(\sigma_1, \sigma_2)$ is in the fiber, this forces $\sigma_1=\hat 1_{\mathcal{L}(\Delta_1)}$. This is because $\Delta_1$ is a simplicial complex, so $\alpha_1\notin \Delta_1$ and 
$\sigma_1\supseteq \alpha_1$ implies $\sigma_1\notin \Delta_1$. 

Hence $(\hat 1_{\mathcal{L}(\Delta_k(G_1))}, \alpha_2)$ is in the fiber, and by the above argument is  its unique minimal element.

We have shown that in all cases the fiber is contractible, and hence by Theorem~\ref{thm:QuillenfiberLemma}, $\phi$ induces a group-equivariant homotopy equivalence of order complexes 
\begin{equation*}
\Delta(\overline{ \mathcal{L}(\Delta_1)\times \mathcal{L}(\Delta_2))}\simeq \Delta(\overline{\mathcal{L}( \Delta))}.
\end{equation*}
But the right-hand side is the barycentric subdivision of $\Delta$, 
and from  ~\eqref{eqn:prod-two-posets}, the left-hand side is homotopy equivalent to 
\[\susp \big(\Delta(\mathcal{L}(\Delta_1)) *\Delta( \mathcal{L}(\Delta_2))\big),\]
where again the order complex $\Delta(\mathcal{L}(\Delta_i))$ of the face lattice  is 
 homeomorphic to the barycentric subdivision of $\Delta_i$, 
 and hence to $\Delta_i$, $i=1,2$. Hence the left-hand side is  homeomorphic and thus homotopy equivalent to $\susp (\Delta_1 *\Delta_2)$, and the claim follows.
\end{proof}

\begin{theorem}\label{thm:susp-of-joins2022Sept12} Let $k\ge 1$,  let $G_i$ be a graph with vertex set $V_i$, $i=1,2$, where $V_1\cap V_2=\emptyset$.  Assume $\Delta_k(G_i)$ is nonvoid for at least one $i=1,2$.
\begin{enumerate}
\item We have the decomposition 
\begin{equation}\label{eqn:graph-join-decomp}
\Delta_k(G_1*G_2)=(\Delta_k(G_1)*\langle V_2\rangle)\cup(\langle V_1\rangle*\Delta_k(G_2)). 
\end{equation}
\item Assume only one cut complex, say $\Delta_k(G_1)$, is void. Then $\Delta_k(G_1*G_2)$ is contractible.
\item Assume both cut complexes $\Delta_k(G_i)$, $i=1,2$, are nonvoid. Then there is a homotopy equivalence 
\begin{equation}\label{eqn:join-susp-iso} \susp ( \Delta_k(G_1)*\Delta_k(G_2) )\simeq \Delta_k(G_1*G_2).
\end{equation}
Moreover, there is a group-equivariant poset map from the product of face lattices 
\[\mathcal{L}(\Delta_k(G_1))\times \mathcal{L}(\Delta_k(G_2))\]
to the face lattice of the simplicial complex $\Delta_k(G_1*G_2)$ 
 which induces a group-equivariant homotopy equivalence of the respective order complexes.
This in turn gives a group-equivariant $(H_1\times H_2)$-homotopy equivalence %
\begin{equation*}\label{eqn:join-susp-iso-group-equiv} \susp ( \Delta_k(G_1)*\Delta_k(G_2) )\simeq_{H_1\times H_2} \Delta_k(G_1*G_2),
\end{equation*}
where $H_i$ is a group acting simplicially on the cut complex $\Delta_k(G_i)$ of the graph $G_i$.
\end{enumerate}
\end{theorem}

\begin{proof} 
Recall that $\sigma$ is a face of the cut complex 
of a graph $G$ if and only if it is contained in a facet, which by definition is the complement of a disconnected 
set of size $k$, i.e., if and only if the complement of $\sigma$ contains a disconnected  
set of size $k$.

There is an edge between any vertex of $G_1$ and any vertex of $G_2$ in the join of graphs $G_1*G_2$. Hence $F$ is a face of $\Delta_k(G_1*G_2)$ 
if and only if $F=F_1\sqcup F_2$, $F_i\subset V_i$,  where 
\begin{itemize}
\item the complement of $F_1$ in $V_1$ contains a disconnected %
set of size $k$ in $G_1$, or
\item the complement of $F_2$ in $V_2$ contains a disconnected %
set of size $k$ in $G_2$.
\end{itemize}

This shows that $\Delta_k(G_1*G_2)$ satisfies the hypotheses of Proposition~\ref{prop:general-joinS2022Sept15}. Equation~\eqref{eqn:graph-join-decomp} is now immediate, as are the 
remaining parts of the theorem. Note that $\Delta_k(G_i)*\langle V_{3-i}\rangle$, $i=1,2$, is contractible since the full simplex is contractible.
\end{proof}

\begin{rem}\label{rem:joinS-total}
We note that exactly the same theorem holds for the \emph{total} cut complex $\Delta^t_k(G_1*G_2)$ of the join of two graphs, which is studied in   \cite{BDJRSX-TOTAL2024}.
\end{rem}

From the K\"unneth Theorem we now also have:
\begin{cor}\label{cor:homology-joins}
Let $G_i$, $i=1,2$ be graphs and $k$ be such that the cut complexes $\Delta_k(G_i)$, $i=1,2$ are nonvoid. Assume the homology of one of $\Delta_k(G_1), \Delta_k(G_2)$ is always free. Then we have the following isomorphism in  homology, which is group-equivariant in rational homology:
\begin{equation*}\tilde{H}_d(\Delta_k(G_1*G_2))\cong \bigoplus_{p+q=d-2} \tilde{H}_{p}(\Delta_k(G_1)) \otimes \tilde{H}_{q}(\Delta_k(G_2)).
  \end{equation*}
\end{cor}
\begin{proof} From the  isomorphism of face lattices in the preceding theorem and the K\"unneth Theorem, Theorem~\ref{thm:Kunneth},  we have 
\begin{align*}
\tilde{H}_d(\Delta_k(G_1*G_2)) & \cong \tilde{H}_d(\susp(\Delta_k(G_1)*\Delta_k(G_2)) )\\ 
& = \tilde{H}_{d-1}(\Delta_k(G_1)*\Delta_k(G_2))\\ 
& \cong \bigoplus_{p+q=d-2} \tilde{H}_{p}(\Delta_k(G_1)) \otimes \tilde{H}_{q}(\Delta_k(G_2)),
\end{align*}
as claimed. 
\end{proof}
    
\subsection{Wedge of Graphs}\label{subsec:Mark-Wedges}

\begin{df}
Given graphs $G_1$ and $G_2$, a \emph{wedge} of $G_1$ and $G_2$, denoted $G_1 \vee G_2$, is formed by taking a vertex from $G_1$ and a vertex from $G_2$ and identifying them.
\end{df}
Note that a wedge of two graphs is not unique in general. See Figure~\ref{fig:Wedge}.

\begin{figure}[htb]
   \scalebox{0.3}
        \centering
        \begin{tikzpicture}
\node[standard] (node3) at (0,0) {3};
\node[standard] (node1) at (0,2) {1};
\node[blue] (node4) at (2,0) {4};
\node[red] (node2) at (2,2) {2};

\draw (node1) -- (node3);
\draw (node1) -- (node2);
\draw (node1) -- (node4);
\draw (node2) -- (node4);
\draw (node3) -- (node4);
\end{tikzpicture}
\begin{tikzpicture}

\node[red] (node5) at (3,1) {5};
\node[standard] (node6) at (4.5,2) {6};
\node[standard] (node7) at (4.5,0) {7};

\draw (node5) -- (node6);
\draw (node5) -- (node7);
\draw (node7) -- (node6);
\end{tikzpicture}
\hspace{15mm}
\begin{tikzpicture}

\node[standard] (bnode3) at (6,1) {3};
\node[standard] (bnode1) at (7,2) {1};
\node[standard] (bnode4) at (7,0) {4};
\node[red, very thick] (bnode2) at (8,1) {2,5};
\node[standard] (bnode6) at (9,2) {6};
\node[standard] (bnode7) at (9,0) {7};

\draw (bnode1) -- (bnode3);
\draw (bnode1) -- (bnode2);
\draw (bnode1) -- (bnode4);
\draw (bnode2) -- (bnode4);
\draw (bnode3) -- (bnode4);
\draw (bnode2) -- (bnode6);
\draw (bnode2) -- (bnode7);
\draw (bnode6) -- (bnode7);

\end{tikzpicture}
\hspace{15mm}
\begin{tikzpicture}

\node[standard] (bnode1) at (10,1) {1};
\node[standard] (bnode2) at (11,2) {2};
\node[standard] (bnode3) at (11,0) {3};
\node[blue, very thick] (bnode4) at (12,1) {4,5};
\node[standard] (bnode6) at (13,2) {6};
\node[standard] (bnode7) at (13,0) {7};

\draw (bnode1) -- (bnode3);
\draw (bnode1) -- (bnode2);
\draw (bnode1) -- (bnode4);
\draw (bnode2) -- (bnode4);
\draw (bnode3) -- (bnode4);
\draw (bnode4) -- (bnode6);
\draw (bnode4) -- (bnode7);
\draw (bnode6) -- (bnode7);

\end{tikzpicture}
\caption{Example of different wedges of the same two graphs}\label{fig:Wedge}
\end{figure}
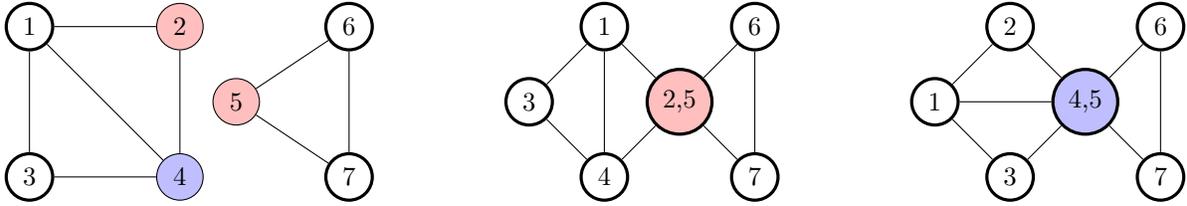

\begin{theorem}\label{thm:Mark-Wedge}
Let $k\geq 2$, and let $G_1$, $G_2$ be graphs. Then $\Delta_k(G_1\vee G_2)$ is shellable if and only if $\Delta_k(G_1)$ and $\Delta_k(G_2)$ are shellable.
\end{theorem}  
\begin{proof}
We know $\Delta_k(G_1\vee G_2)$ being shellable implies $\Delta_k(G_1)$ and $\Delta_k(G_2)$ are shellable, since $G_1$ and $G_2$ are induced subgraphs of $G_1\vee G_2$. Now suppose $\Delta_k(G_1)$ and $\Delta_k(G_2)$ are shellable. If $\Delta_k(G_1\vee G_2)$ is the void complex, then it is shellable, so assume $\Delta_k(G_1\vee G_2)$ is not the void complex. Let $V_i$ be the vertex set of $G_i$ and let $\{v_0\}=V_1\cap V_2$, so that $v_0$ is the common wedge point.
We consider the facet set $\F$ of $\Delta_k(G_1\vee G_2)$; they are separating sets of size
$|V_1|+|V_2|-k-1$ of the graph $G_1\vee G_2$. %
There are four types of facets:

Type 1: Sets of the form $v_0\cup A\cup B$, where $A\subsetneq V_1-v_0$ and $B\subsetneq V_2-v_0$.

Type 2: Sets of the form $A\cup B$, where $\emptyset\neq A\subsetneq V_1-v_0$ and $\emptyset\neq B\subsetneq V_2-v_0$,
where either $A$ is a separating set for $G_1$ or $B$ is a separating set for $G_2$, or both.

Type 3: Sets of the form $A\cup (V_2-v_0)$, where $A$ is a $(|V_1|-k)$-subset of $V_1$ whose removal  disconnects $G_1$.

Type 4: Sets of the form $(V_1-v_0)\cup B$, where $B$ is a $(|V_2|-k)$-subset of $V_2$  whose removal  disconnects $G_2$.

Note that if $k\ge |V_1|-1$, then there is no Type 3 facet, and
if $k\ge |V_2|-1$, then there is no Type 4 facet.

We start the shelling with Type 1 facets.
Order the Type 1 facets of $\Delta_k(G_1\vee G_2)$ first in order of decreasing
size of $A$, then, among facets with sets $A$ of the same size, in lexicographic
order of the sets $A$, and then, among facets with fixed $A$, in lexicographic
order of the sets $B$.  Label the facets of Type 1 in the resulting order
$F_1,F_2, \ldots, F_r$.

Consider $F_i= v_0\cup A_i\cup B_i$ and
$F_j= v_0\cup A_j\cup B_j$, with $i<j$.  Here $|A_i|\ge |A_j|$.
If $A_i=A_j$, then $|B_i|=|B_j|$ and $B_i$ precedes $B_j$ in lexicographic
order.
By the lexicographic shellability of the $(|B_j|-1)$-skeleton of the simplex
on $V_2 - v_0$, there exists a set $B\subsetneq V_2 - v_0$, with
$|B|=|B_j|$, preceding $B_j$ in lexicographic order such that
$B_i\cap B_j\subseteq B\cap B_j$ and $|B\cap B_j|=|B_j|-1$.  Then for
$F= v_0\cup A_j\cup B$, $F$ precedes $F_j$, $F_i\cap F_j\subseteq F\cap F_j$,
and $|F\cap F_j| = |F_j|-1$.

Now suppose $A_i\ne A_j$ but $|A_i|=|A_j|$.  Then $A_i$ precedes $A_j$ in
lexicographic order.
By the lexicographic shellability of the $(|A_j|-1)$-skeleton of the simplex
on $V_1 - v_0$, there exists a set $A\subsetneq V_1 - v_0$, with
$|A|=|A_j|$, preceding $A_j$ in lexicographic order such that
$A_i\cap A_j\subseteq A\cap A_j$ and $|A\cap A_j|=|A_j|-1$.  Then for
$F= v_0\cup A\cup B_j$, $F$ precedes $F_j$, $F_i\cap F_j\subseteq F\cap F_j$,
and $|F\cap F_j| = |F_j|-1$.

Finally, suppose $|A_i|>|A_j|$.  Then $|B_i|<|B_j|$.
Let $x\in A_i\setminus A_j$, $y\in B_j\setminus B_i$, and
$F=v_0\cup (A_j\cup x)\cup (B_j - y)$.
Then $F$ comes before $F_j$ and
$$F_i\cap F_j = v_0\cup (A_i\cap A_j) \cup (B_i\cap B_j)
  \subseteq v_0 \cup (A_j\cup x)\cup (B_j - y)=F,$$
so $F_i\cap F_j\subseteq F\cap F_j$ and
$$ |F\cap F_j| = | ( v_0\cup A_j\cup B_j)\cap
   (v_0 \cup (A_j\cup x)\cup (B_j - y))| = |v_0 \cup A_j\cup (B_j - y)|
   = |F_j| - 1.$$

So $F_1,F_2, \ldots, F_r$ is a shelling order of the Type 1 facets.

We now claim that the Type 2 facets can be added to the shelling order in any order. We will demonstrate this by showing that the intersection of a Type 2 facet and any other Type 2 or Type 1 facet is contained in a Type 1 facet that differs by only one vertex. Let $F$ be a Type 2 facet and $X$ be a distinct Type 1 or Type 2 facet. As $F\neq X$, there exists $v\in F\setminus X$. We claim $F'=F-v+v_0$ is a Type 1 facet. As $v_0\notin F$, $|F'|=|F|$. As $V_1-v_0 \nsubseteq F$, then $V_1-v_0 \nsubseteq F-v$, so $V_1-v_0 \nsubseteq F'$; similarly, $V_2-v_0\nsubseteq F'$. Then $F'=v_0\cup A\cup B$ where $A\subsetneq V_1-v_0$ and $B\subsetneq V_2-v_0$. So $F'$ is a Type 1 facet, and $|F\setminus F'|=1$, and thus any order of Type 2 facets placed after the Type 1 facets will result in a shelling order. 

Finally we add the Type 3 and Type 4 facets.  As $\Delta_k(G_1)$ is shellable, the join of $\Delta_k(G_1)$ with the simplex on $V_2-v_0$ is shellable, and we add the Type 3 facets in such a shelling order. We do the same for the Type 4 facets. Now we just need to verify that the intersection of a Type 3 or 4 facet with any facet of a different type is contained in a Type 1 facet. Without loss of generality, assume $F$ is a Type 4 facet, and $X$ is a facet of another type. Type 4 facets are the only facets that contain all of $V_1-v_0$, so we know there exists $v\in (F\setminus X)\cap (V_1-v_0)$. We now need to break into two small cases. In the first case, $v_0\notin F$, and so we choose $F'=F-v+v_0$. Then $|F'|=|F|-1+1$, and $V_1-v_0\nsubseteq F'$ as $v\notin F'$. Now, $V_2-v_0\nsubseteq F'$ as $F$ was missing $k$ vertices from $V_2$ and we only added one vertex to it. So $F'=v_0\cup A\cup B$ where $A\subsetneq V_1-v_0$ and $B\subsetneq V_2-v_0$. So $F'$ is a Type 1 facet, and $|F\setminus F'|=1$. The second case has $v_0\in F$; we choose any $v'\in F^c\cap V_2$, and set $F'=F-v+v'$. Then $|F'|=|F|$, and still  $V_1-v_0\nsubseteq F'$ as $v\notin F'$. Also, $V_2-v_0\nsubseteq F'$ as $F$ was missing $k$ vertices from $V_2$ and we only added one vertex to it. So $F'=v_0\cup A\cup B$ where $A\subsetneq V_1-v_0$ and $B\subsetneq V_2-v_0$. So $F'$ is a Type 1 facet, and $|F\setminus F'|=1$. This demonstrates that the intersection of any Type 4 facet with a facet of another type is contained in a Type 1 facet with only one vertex different, and very similar arguments can be used to show the same is true with Type 3 facets. This is the final step in confirming our shelling order, demonstrating that the complex $\Delta_k(G_1\vee G_2)$ is shellable.
\end{proof}

As an immediate corollary we obtain the next result, which was originally proved in \cite{DenkerMastersReport}  by exhibiting an explicit shelling order.

\begin{cor}\label{cor:Mark-trees}
    If $G$ is a tree, then $\Delta_k(G)$ is shellable for all $k\geq 2$.
 \end{cor}
    \begin{proof} A tree on $n$ vertices is the wedge of $n-1$ copies of $K_2$.
    The cut complex of any complete graph is the void complex for all $k>2$, and so is trivially shellable. Hence the wedge of these shellable graphs is shellable.  For $k=2$ the claim follows from  Theorem~\ref{thm:Froberg}, since trees are chordal. \end{proof}

\subsection{Minimal nonshellable graphs}\label{sec:min}

Recall that Proposition~\ref{prop:induced subgraphs} says that if $G$ has a shellable $k$-cut complex, then any induced subgraph of $G$ also has a shellable $k$-cut complex. In light of this result, it is natural to seek a description of the \emph{minimal} graphs (with respect to induced inclusion) whose $k$-cut complex fails to be shellable. We will call such a graph a {\em  minimal nonshellable graph} or a {\em minimal forbidden subgraph} for $k$-cut complex shellability.

Figure~\ref{fig:min-nonshellable}  shows some minimal nonshellable graphs for $k = 3$.  
The first graph in Figure~\ref{fig:min-nonshellable} is in fact the $k = 3$ case of a family of graphs considered in Lemma~\ref{lem:Prism-over-clique-min-nonshellable}. More generally, for arbitrary $k\ge 4$ we have the following proposition.    

 \begin{prop}\label{prop:existence-min-nonshellablegraph-all-kS-kayaks}  For each $k\ge4$, there is  a chordal graph $G_k$ with $k+2$ vertices such that $\Delta_k(G_k)$ is NOT shellable.  Furthermore, $G_k$ is a minimal forbidden subgraph for $k$-cut complex shellability.
 \end{prop}
 \begin{proof}
     The graph $G_k$ has $k+2$ vertices, which we partition into three disjoint subsets according to the parity of $k$: 
$\{2i-1: 1\le i\le m=\lfloor k/2\rfloor\}, \{2i:1\le i\le m=\lfloor k/2\rfloor\}$ and $\begin{cases} \{a\} &\text{ for odd } k=2m+1,\\ 
\{a,b\} &\text{ for even } k=2m. 
\end{cases} $

See Figure~\ref{fig:long-kayaks}. The edges are precisely those specified by requiring that the induced subgraph on each 4-vertex subset $\{2i-1, 2i, 2i+1, 2i+2\}$ forms a clique for all $1\le i\le m=\lfloor k/2\rfloor$, as do the induced subgraphs on the 3-vertex sets $\{a,1,2\}$ and $\{b, 2m-1,2m\}$, the latter existing only if $k=2m$ is even. 
Then $\Delta_k(G_k)$ has facets $\{\{2i-1, 2i\}: 1\le i\le m=\lfloor k/2\rfloor\}$ (highlighted in the figures), and is thus not shellable since $m=\lfloor k/2\rfloor\ge 2$.  Also $\Delta_{2m}(G_{2m})=\Delta_{2m+1}(G_{2m+1})$, $m\ge 2$.  Finally, for each $k$, $G_k$ is minimal nonshellable:  for all vertices $v$, $G_k\backslash v$ has only $k+1$ vertices, and $\Delta_k(G_k\backslash v)$ is either void or 0-dimensional.
\end{proof}

\begin{figure}[htb]
\captionsetup[subfigure]{labelformat=empty}
\centering
        \begin{subfigure}{0.4\textwidth}
        \centering
        \begin{tikzpicture}[scale=0.5]
\node[standard] (node1) at (0,2) {1};
\node[standard] (node2) at (0,0) {2};
\node[standard] (node3) at (2,2) {3};
\node[standard] (node4) at (2,0) {4};

\node[standard] (node5) at (-2,1) {a};
\node[standard] (node6) at (4,1) {b};

\draw[ultra thick, color=red] (node1) -- (node2);
\draw[ultra thick, color=red] (node3) -- (node4);
\draw (node1) -- (node3);
\draw (node1) -- (node4);  
\draw (node2) -- (node3);  

\draw (node3) -- (node6);
\draw (node2) -- (node4);
\draw (node4) -- (node6);
\draw (node1) -- (node5);
\draw (node2) -- (node5);

        \end{tikzpicture}
       
\caption{$G_4$}
        \end{subfigure} \qquad
        \begin{subfigure}{0.4\textwidth}
        \centering
        \begin{tikzpicture}[scale=0.5]
\node[standard] (node1) at (0,2) {1};
\node[standard] (node2) at (0,0) {2};
\node[standard] (node3) at (2,2) {3};
\node[standard] (node4) at (2,0) {4};

\draw[ultra thick, color=red] (node1) -- (node2);
\draw[ultra thick, color=red] (node3) -- (node4);

        \end{tikzpicture}
        
\caption{$\Delta_4(G_4)$}
\end{subfigure}
   \label{fig:kayak1}

\vspace{3ex}

        \begin{subfigure}{0.4\textwidth}
        \centering
        \begin{tikzpicture}[scale=0.5]
\node[standard] (node1) at (0,2) {1};
\node[standard] (node3) at (2,2) {3};
\node[standard] (node5) at (4,2) {5};
\node[standard] (node2) at (0,0) {2};
\node[standard] (node4) at (2,0) {4};
\node[standard] (node6) at (4,0) {6};

\node[standard] (node7) at (-2,1) {a};

\draw[ultra thick, color=red] (node1) -- (node2);
\draw[ultra thick, color=red] (node3) -- (node4);

\draw (node1) -- (node3);
\draw (node1) -- (node4);  
\draw (node2) -- (node3);  
\draw (node2) -- (node4);

\draw (node3) -- (node5);
\draw (node3) -- (node6);
\draw (node4) -- (node5);
\draw (node4) -- (node6);
\draw (node5) -- (node6);

\draw (node7) -- (node1);
\draw (node7) -- (node2);

        \end{tikzpicture}
\caption{$G_5$}
        \end{subfigure} \quad
        \begin{subfigure}{0.4\textwidth}
        \centering
        \begin{tikzpicture}[scale=0.5]
\node[standard] (node1) at (0,2) {1};
\node[standard] (node3) at (2,2) {3};
\node[standard] (node2) at (0,0) {2};
\node[standard] (node4) at (2,0) {4};

\draw[ultra thick, color=red] (node1) -- (node2);
\draw[ultra thick, color=red] (node3) -- (node4);

        \end{tikzpicture}

\caption{$\Delta_5(G_5)$}

        \end{subfigure}
       
   \label{fig:kayak2}

\vspace{3ex}

        \begin{subfigure}{0.4\textwidth}
        \centering
        \begin{tikzpicture}[scale=0.5]
\node[standard] (node1) at (0,2) {1};
\node[standard] (node3) at (2,2) {3};
\node[standard] (node5) at (4,2) {5};
\node[standard] (node2) at (0,0) {2};
\node[standard] (node4) at (2,0) {4};
\node[standard] (node6) at (4,0) {6};

\node[standard] (node7) at (-2,1) {a};
\node[standard] (node8) at (6, 1) {b};

\draw[ultra thick, color=red] (node1) -- (node2);
\draw[ultra thick, color=red] (node3) -- (node4);
\draw[ultra thick, color=red] (node5) -- (node6);

\draw (node1) -- (node3);
\draw (node1) -- (node4);  
\draw (node2) -- (node3);  
\draw (node2) -- (node4);

\draw (node3) -- (node5);
\draw (node3) -- (node6);
\draw (node4) -- (node5);
\draw (node4) -- (node6);

\draw (node7) -- (node1);
\draw (node7) -- (node2);

\draw (node8) -- (node5);
\draw (node8) -- (node6);

        \end{tikzpicture}
       
\caption{$G_6$}
        \end{subfigure} \qquad
        \begin{subfigure}{0.4\textwidth}
        \centering
        \begin{tikzpicture}[scale=0.5]
\node[standard] (node1) at (0,2) {1};
\node[standard] (node3) at (2,2) {3};
\node[standard] (node5) at (4,2) {5};
\node[standard] (node2) at (0,0) {2};
\node[standard] (node4) at (2,0) {4};
\node[standard] (node6) at (4,0) {6};

\draw[ultra thick, color=red] (node1) -- (node2);
\draw[ultra thick, color=red] (node3) -- (node4);
\draw[ultra thick, color=red] (node5) -- (node6);

        \end{tikzpicture}

\caption{$\Delta_6(G_6)$}
        \end{subfigure}
       
\vspace{3ex}

        \begin{subfigure}{0.4\textwidth}
        \centering
        \begin{tikzpicture}[scale=0.5]
\node[standard] (node1) at (0,2) {1};
\node[standard] (node3) at (2,2) {3};
\node[standard] (node5) at (4,2) {5};
\node[standard] (node7) at (6, 2) {7};
\node[standard] (node2) at (0,0) {2};
\node[standard] (node4) at (2,0) {4};
\node[standard] (node6) at (4,0) {6};
\node[standard] (node8) at (6, 0) {8};

\node[standard] (node9) at (-2,1) {a};

\draw[ultra thick, color=red] (node1) -- (node2);
\draw[ultra thick, color=red] (node3) -- (node4);
\draw[ultra thick, color=red] (node5) -- (node6);

\draw (node1) -- (node3);
\draw (node1) -- (node4);  
\draw (node2) -- (node3);  
\draw (node2) -- (node4);

\draw (node3) -- (node5);
\draw (node3) -- (node6);
\draw (node4) -- (node5);
\draw (node4) -- (node6);

\draw (node5) -- (node7);
\draw (node5) -- (node8);
\draw (node6) -- (node7);
\draw (node6) -- (node8);

\draw (node7) -- (node8);

\draw (node9) -- (node1);
\draw (node9) -- (node2);

        \end{tikzpicture}
        
\caption{$G_7$}
        \end{subfigure} \qquad
        \begin{subfigure}{0.4\textwidth}
        \centering
        \begin{tikzpicture}[scale=0.5]
\node[standard] (node1) at (0,2) {1};
\node[standard] (node3) at (2,2) {3};
\node[standard] (node5) at (4,2) {5};
\node[standard] (node2) at (0,0) {2};
\node[standard] (node4) at (2,0) {4};
\node[standard] (node6) at (4,0) {6};

\draw[ultra thick, color=red] (node1) -- (node2);
\draw[ultra thick, color=red] (node3) -- (node4);
\draw[ultra thick, color=red] (node5) -- (node6);
        \end{tikzpicture}

\caption{$\Delta_7(G_7)$}
        \end{subfigure}
        
        \caption{Graphs and cut complexes described in Proposition~\ref{prop:existence-min-nonshellablegraph-all-kS-kayaks}}
   \label{fig:long-kayaks}
    \end{figure}
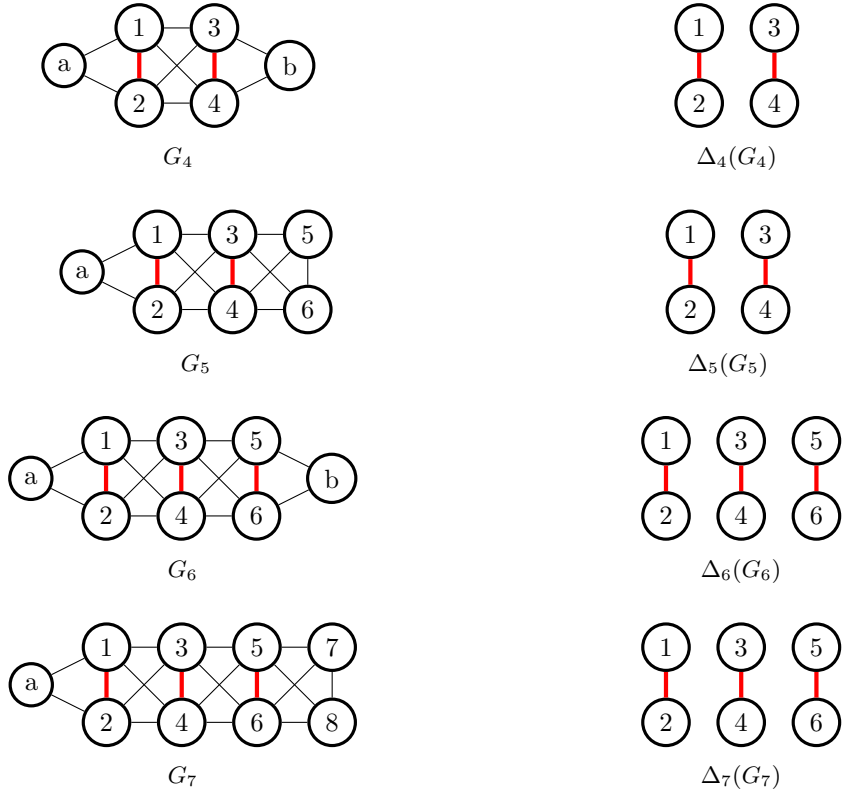

We will encounter some more families of minimal nonshellable graphs in this paper: in particular, the prism over a $k$-clique discussed in Section~\ref{sec:RowanNonshellableGraph} and the squared cycle on $k+4$ vertices in Section~\ref{sec:Wreath-MarijaDMT-Rowan-Dane-Mark} are both minimal forbidden subgraphs for $k$-cut complex shellability, by Lemma~\ref{lem:Prism-over-clique-min-nonshellable} and Proposition~\ref{prop:Dane-Wreath-k+4-min-nonshellable}, respectively.

\begin{figure}[htb]
    \centering
    \includegraphics[scale=.4]{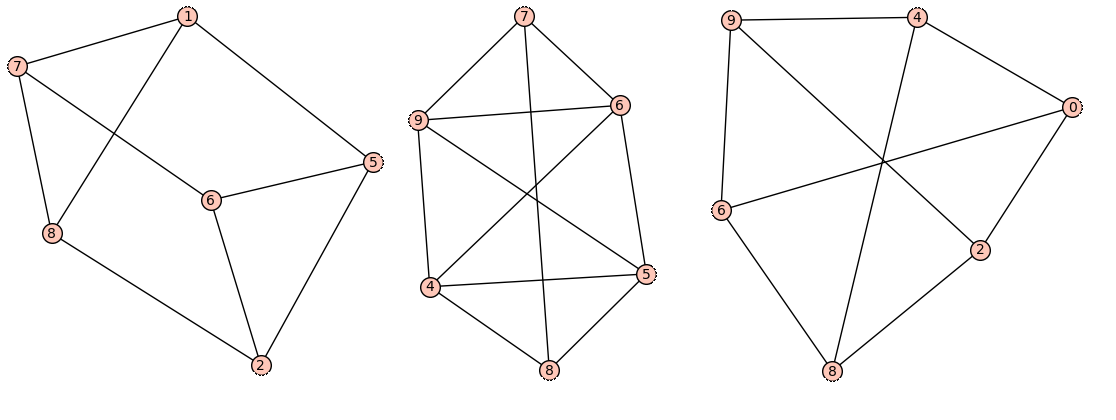}
    \caption{Some minimal nonshellable graphs for $k=3$}
    \label{fig:min-nonshellable}
\end{figure}
    
\section{Some theorems for the case $k=3$}\label{sec:Mark:k=3-cut-complex}

In this section we consider an operation which simultaneously generalizes disjoint unions and wedges of graphs.  
\begin{theorem}[Generalized Wedge]\label{thm:Gen-wedge}
Let $G=(V,E)$, $A\cup B=V$ such that if $e\in E$, e is between two vertices in $A$ or two vertices in $B$, and $\Delta_3(G[A\cap B])$ is the void complex. Then $\Delta_3(G)$ is shellable if and only if $\Delta_3(G[A])$ and $\Delta_3(G[B])$ are shellable.
\end{theorem}
\begin{proof}
We know that $G$ having a shellable cut complex implies its induced subgraphs also have shellable cut complexes. So we assume $\Delta_3(G[A])$ and $\Delta_3(G[B])$ are shellable. We will now construct a shelling order for $\Delta_3(G)$ on its set of facets $\F$. 

Let $S=A\cap B$.  Note that the cases $|S|=0, 1$ were settled in 
Theorem~\ref{thm:MargeDisjunion} and Theorem~\ref{thm:Mark-Wedge}, respectively.

Because $k=3$, it is more convenient to classify our facets in terms of their complements, the disconnected 3-sets.
Let $\F_A=\{F\in\F\mid F^c\subseteq A\}$ be the set of facets whose complement lies entirely in $A$, and $\F_B=\{F\in\F\mid F^c\subseteq B\}$ be facets whose complement lies in $B$. The intersection $\F_A\cap\F_B$ is empty because $\Delta_3(G[A\cap B])$ is the void complex, so no disconnected 3-sets lie in $S=A\cap B$. For the remaining facets, we characterize them by the intersection of their complement with $S$, $\F_i=\{F\in\F\setminus(\F_A\cup\F_B)\mid |F^c\cap S|=i\}$. We know $S$ cannot contain a disconnected set of size 3, so $\F_3 = \emptyset$, and if a set of three elements has 2 elements in $S$, then it is completely contained in either $A$ or $B$, so $\F_2=\emptyset$. So only $\F_0$ and $\F_1$ are potentially nonempty. This gives a complete classification of facets $\F=\F_0\sqcup\F_1\sqcup\F_A\sqcup\F_B$.

Our shelling order construction starts by ordering the facets in $\F_0$, the facets whose complement shares no elements with $S$, or, equivalently, the facets of $\Delta_3(G)$ that contain $S$. This subcomplex of $\Delta_3(G)$ is precisely $\Delta_3(G[V\setminus S])*\langle S\rangle$   as in Proposition~\ref{prop:induced subgraphs}, since  $\Delta_3(G[V\setminus S])$ is the link of $\langle S\rangle$, but $\Delta_3(G[V\setminus S])=\Delta_3(G[A\setminus S]+G[B\setminus S])$, and so is shellable from Theorem \ref{thm:MargeDisjunion}.   This provides us with a shelling order for $\F_0$, call it $S_0$. (We note in passing that the complex $\langle \mathcal{F}_0\rangle$ is contractible.)

Now we append an order of the facets in $\F_1$; in this case any order will do. If $F\in \F_1$ and $X\in\F_0\cup \F_1$, then either $|F\cap X|=|F|-1$ (in which case $X=F'$ suffices for the shellability criterion) or there exists $x\in F\setminus (X\cup S)$ as $X^c$ has at most one element in $S$. In that case, let $y\in F^c\cap S$, and consider $F'=F-x+y$.  Then $F\cap X\subseteq F'$, and $|F\cap F'|=|F|-1$. Additionally, $F'\in \F_0$ as $F'^c\cap S=\emptyset$, and $F'\nsubseteq A$, $F'\nsubseteq B$, so any order of $\F_1$ will do. Choose some order and call it $S_1$.

The facets in $\F_A$ are the facets of $\Delta_3(G[A])$ with the elements of $B\setminus A$ added, and so can inherit the shelling order of $\Delta_3(G[A])$: call it $S_A$.
A similar argument holds for $S_B$. We claim that the order of $S_0$ followed by $S_1$ followed by $S_A$ followed by $S_B$ is a shelling order. We have already demonstrated that $S_0$ followed by $S_1$ fulfills the shelling criterion, and $S_A$ and $S_B$ work internally. To complete the proof we will demonstrate that the intersection of facets of $\F_B$ with facets outside of $\F_B$ are always contained in facets in $\F_0$ or $\F_1$ that fulfill the shelling criterion, and a similar argument will hold for the facets of $\F_A$.

Let $F\in \F_B$, and $X\in\F\setminus\F_B$. Then there exists $x\in X^c\cap B^c$. We will choose $F'=F-x+y$ so that $F\cap X \subset F'$ and $|F\cap F'|= |F|-1$, but we need to decide which $y$ to add. If $|F^c\cap S|=0$, then any $y\in F^c$ will do, as then $F'\in \F_0$; if $|F^c\cap S|=1$, add the $y\in F^c\cap S$, and again $F'\in\F_0$. If $|F^c\cap S|=2$, then as $|F^c|=3$ and $G[F^c]$ is not connected, the element of $F^c$ not in $S$ cannot be connected to both elements in $S$. Choose the $y\in F^c\cap S$ such that the remaining two elements of $F^c$ do not share an edge. Then $F'=F-x+y\in\F_1$, as the element of $F'$ contained in $B\setminus A$ cannot share an edge with $x$ as $x\in A\setminus B$, and also does not share an edge with the element in $S$, so it is an isolated vertex in $G[F'^c]$ and so $F'\in\F_1$. In each case $|F\cap F'|=|F|-1$, $F\cap X\subseteq F'$, and $F'$ comes before $F$ in the shelling order as $F'\in\F_0\cup\F_1$. A similar argument holds for every facet in $\F_A$ and every facet not in $\F_A$. So every pair of facets in the shelling order meets the shelling criterion, and $S_0+S_1+S_A+S_B$ is a shelling order for $\Delta_3(G)$.
\end{proof}

The construction of Proposition~\ref{prop:existence-min-nonshellablegraph-all-kS-kayaks} shows that the conclusion of Theorem~\ref{thm:Gen-wedge} is false for $k\ge 4$. 

\begin{cor}\label{cor:not-min}
If $G$ contains a separating set $S$ such that $\Delta_3(S)$ is the void complex, then $G$ is not a minimal forbidden subgraph for 3-cut complex shellability.
\end{cor}
\begin{proof}
Let $G=(V,E)$, $V=A\sqcup S\sqcup B$, where $S$ separates $A$ and $B$, both non-empty.  If $\Delta_3(G)$ is not shellable, either $\Delta_3(G[A\cup S])$ or $\Delta_3(G[B\cup S])$ is not shellable, so $G$ is not a minimal forbidden subgraph. 
\end{proof}
The following is in contrast to  Proposition~\ref{prop:existence-min-nonshellablegraph-all-kS-kayaks}.
\begin{cor}\label{cor:chordal-Delta3-shell}
If $G$ is chordal, then $\Delta_3(G)$ is shellable.
\end{cor}
\begin{proof}
Suppose $G$ is chordal, but $\Delta_3(G)$ is not shellable. Let $H$ be a minimal induced subgraph of $G$ with a nonshellable 3-cut complex. $H$ is chordal as it is an induced subgraph of $G$, so it has a vertex $v$ whose neighborhood $N(v)=S$ is a clique. Either $H$ is the complete graph, in which case $\Delta_3(H)$ is the void complex and therefore shellable, or $S$ is a separating set such that $\Delta_3(S)$ is the void complex, so $H$ is shellable. In either case we have a contradiction, so $\Delta_3(G)$ is in fact shellable.
\end{proof}

\begin{cor}\label{cor:Delta3-min-forbidden-is3-conn}
Every minimal forbidden subgraph for 3-cut complex shellability is at least 3-connected.
\end{cor}
\begin{proof}
If $G$ has a separating set $S$ such that $|S|<3$, then $\Delta_3(S)$ is trivially the void complex, so $G$ is not a minimal forbidden subgraph for 3-cut complex shellability.
\end{proof}

\begin{cor}\label{cor:min-forb-subgraph-Marge}
Suppose $G$ is 3-connected, every vertex of $G$ has a neighbor of degree 3, and $\Delta_3(G)$ is not shellable.  Then $G$ is a minimal forbidden subgraph for 3-cut complex shellability.
\end{cor}
\begin{proof}
Assume the hypothesis.  Then every proper induced subgraph $H$ of $G$ has a vertex of degree at most 2.  So $H$ has a component that is either  $K_i$ for $1\le i\le 3$ or is not 3-connected.  So $H$ is not a minimal forbidden subgraph for 3-cut complex shellability. So $G$ is a minimal forbidden subgraph for 3-cut complex shellability.
\end{proof}

\section{The face lattice of the $k$-cut complex}\label{sec:Euler-char}

In this section we investigate the face lattice of the cut complex. Note that a simplicial complex $\Delta$ is shellable if and only if its face lattice admits a recursive coatom ordering \cite[Theorem~4.3, Corollary~4.4]{BjWachsTAMS1983}. Under certain favourable conditions (see Theorem~\ref{thm:truncBoolean-minus-antichain} below), the face lattice admits a particularly simple description which allows us to conclude that homology is torsion-free and occurs in at most two dimensions, and also gives us a way to compute the Betti numbers of the cut complex.

The reduced Euler characteristic of a simplicial complex $\Delta$ is the M\"obius number $\mu(\mathcal{L}(\Delta))$ of its face lattice $\mathcal{L}(\Delta)$ \cite{RPSEC11997}. We exploit the fact that the face lattice of the cut complex is a subposet of the truncated Boolean lattice, and then use poset topology techniques to determine the M\"obius number, using Theorem~\ref{thm:Bac-mu}.  Recall that the truncated Boolean lattice, by definition, has an artificially appended top element.

For a subset  $A$  of the vertex set $V(G)$ of a graph $G$, we say $A$ is a \emph{connected set}  if the induced subgraph $G[A]$ is connected. 
Recall from Section~\ref{sec:Topology} that $P(n,k)$ denotes the truncated Boolean lattice $B_n^{\le n-k}$, and let 
 \[\mathcal{Z}_k(G) \coloneqq \{A^c: |A|=k, A \text{ is a connected subset of }V(G)\}\subseteq P(n,k),\]
 where $A^c$ denotes the complement of $A$ in the vertex set $V(G)$. 
Thus $\mathcal{Z}_k(G)$  is the set of those $(n-k)$-element subsets of $[n]$ that are not facets of the cut complex,
 and 
the number of facets of $ \Delta_k(G)$ is $\binom{n}{k}-|\mathcal{Z}_k(G)|$.
Clearly the face lattice of $\Delta_k(G)$ is a subposet of $P(n,k)\setminus \mathcal{Z}_k(G)$.
\begin{theorem}\label{thm:truncBoolean-minus-antichain} Let $G$ be a graph with vertex set $V(G)$ of size $n$, and let $k\ge 2$.
 Then 
 the face lattice of $\Delta_k(G)$ coincides with  $P(n,k)\setminus \mathcal{Z}_k(G)$ if and only if the $(n-k-1)$-dimensional complex $\Delta_k(G)$ contains a complete $(n-k-2)$-skeleton, that is, if and only if either of the following equivalent conditions holds:
\begin{gather}\label{eqn:Condition-truncBoolean} \text{for every subset $X$ of a set $A^c\in \mathcal{Z}_k(G)$ with $|X|=n-k-1$, $X^c$ contains a disconnected set of size $k$;} \\
\label{eqn:Condition-truncBoolean-2} 
\text{if $A^c\in \mathcal{Z}_k(G)$ and  $x\notin A$, 
 there is a $y\in A$}  \text{ such that $(A\setminus\{y\})\cup\{x\}$ is disconnected. } 
\end{gather}
 If condition~\eqref{eqn:Condition-truncBoolean} holds, the reduced Euler characteristic of $\Delta_k(G)$ is given by 
\[(-1)^{n-k-1}\mu(\Delta_k(G))=\binom{n-1}{k-1}- |\mathcal{Z}_k(G)|
=|\{F: \text{$F$ is a facet of $\Delta_k(G)$}\}| -\binom{n-1}{k}.\]
Furthermore, in this case the nonzero homology of $\Delta_k(G)$ is torsion-free and occurs in at most two dimensions, $n-k-1$ and $n-k-2$.

If condition~\eqref{eqn:Condition-truncBoolean} holds and $\Delta_k(G)$ is shellable, then it is homotopy equivalent to 
\[\begin{cases}
\text{a point, if $\mu(\Delta_k(G))=0$, i.e., if the number of facets of $\Delta_k(G)$ is $\binom{n-1}{k}$}, &\\
\text{a wedge of $\left(\binom{n-1}{k-1}- |\mathcal{Z}_k(G)|\right)$ spheres in dimension $n-k-1$, otherwise}.
\end{cases}\]

\end{theorem}
\begin{proof} The equivalence of the two conditions, and the fact that  $\Delta_k(G)$ contains a complete $(n-k-2)$-skeleton, follow from the definitions.  Either condition  ensures that  any faces contained in the removed facets $\mathcal{Z}_k(G)$ are already  faces of $\Delta_k(G)$, since $X^c$ contains a disconnected set of size $k$ if and only if $X$ is a subset of a facet of $\Delta_k(G)$.

For the remaining statements, we apply Theorem~\ref{thm:Bac-mu} to $P=P(n,k)$ and $Q=P\setminus\mathcal{Z}_k(G)$; note that $\mathcal{A}=\mathcal{Z}_k(G)$ is an antichain.

Also note that the intervals $(\hat 0, A^c)$ are Boolean intervals of $P(n,k)$ for all $A^c\in \mathcal{Z}_k(G)$, where $A$ ranges over all connected $k$-subsets of $V(G)$. We obtain 
\[\mu(P(n,k)\setminus\mathcal{Z}_k(G))
=\mu(P(n,k)) -\sum_{A^c\in \mathcal{Z}_k(G)} \mu(\hat 0, A^c) \mu(A^c, \hat 1) 
= \mu(P(n,k)) -  (-1)^{n-k} (-1)| \mathcal{Z}_k(G)|.\]
Since $\mu(P(n,k))=(-1)^{n-k-1}\binom{n-1}{k-1},$ the result follows.

The second equality  follows from the discussion preceding the theorem.

The last statement about the homology of $\Delta_k(G)$ is a consequence of \cite[Theorem~2.1, Theorem~2.5]{NonmodJAC1999}, which applies because we are deleting an antichain from the Cohen--Macaulay poset $B_n^{\le n- k}$. 

Finally, note that a nonzero reduced Euler characteristic gives the dimension of the unique nonvanishing homology, if the nonzero homology is concentrated in a single degree, as is the case for a shellable complex.  If the Euler characteristic is zero, the shellable complex must be contractible.
\end{proof}

\begin{prop}\label{prop:class-of-graphs-2022-3-26} Let $G$ be a graph on $n$ vertices, let $k\ge 2$, and assume that $G$ contains no cycles of length less than or equal to $(k+1)$.  Then 
$\Delta_k(G)$ contains a complete $(n-k-2)$-skeleton.
\end{prop}
\begin{proof}    If $G$ has no cycles of length less than or equal to $(k+1)$, any $(k+1)$-subset of the vertex set induces    a 
 forest, and thus contains a disconnected set of size $k$.  
By Definition~\ref{def:cut-cplx}, all subsets of size $(n-k-1)$ are faces of the cut complex. 
\end{proof}

For instance, consider the disjoint union $C_n+C_m$ of two cycle graphs $C_n$, $C_m$ with $n,m \ge 4$ and the cut complex $\Delta_2(C_n+C_m)$. There are no triangles, and there are $(n+m)$ connected subsets of size 2, so the preceding results give the reduced Euler characteristic $\mu(\Delta_2(C_n+C_m))=
(-1)^{n+m}.$  Sage code (for $9\ge m\ge n\ge 4$) gives homology of rank 1 in the top dimension and rank 2 in the dimension one below the top, which is consistent with the Euler characteristic.

The converse of Proposition~\ref{prop:class-of-graphs-2022-3-26} is false, as shown by  Theorem~\ref{thm:anycomplexiscutcomplex}. 

Since trees are acyclic, Corollary~\ref{cor:Mark-trees} together with Theorem~\ref{thm:truncBoolean-minus-antichain} and Proposition~\ref{prop:class-of-graphs-2022-3-26} immediately give:
\begin{cor}\label{cor:Delta2trees-contractible} If $T$ is a tree, then $\Delta_2(T)$ is contractible.
\end{cor}

If the graph $G$ has $(k+1)$-cycles,  the conditions in Theorem~\ref{thm:truncBoolean-minus-antichain} can be modified 
 to obtain a precise description of the face lattice of $\Delta_k(G)$, from which  we again obtain a formula for the reduced Euler characteristic. This is described in a forthcoming paper.

\subsection{A theorem about $\Delta_2$ for connected triangle-free  graphs}

Theorem~\ref{thm:truncBoolean-minus-antichain} 
asserts that if $G$ has $n$ vertices and $e$ edges and $\Delta_2(G)$ contains a complete $(n-4)$-skeleton, then $\Delta_2(G)$ has torsion-free homology, with Euler characteristic $(-1)^{n-4}(e-n+1)$, and the homology is nonzero in at most the top two dimensions.  In this section we refine this result, 
noting the following special case of Proposition~\ref{prop:class-of-graphs-2022-3-26}: if $G$ is triangle-free, $\Delta_2(G)$ contains a complete codimension 1 skeleton.

\begin{theorem}\label{thm:Delta2nonchordalSMark}
Let $G$ be a graph on $n$ vertices with $e$ edges, and assume $G$ is connected and triangle-free.  If $G$ is not a tree, then $\Delta_2(G)$ is a wedge of $e-n+1$ spheres in dimension $n-4$, that is, in codimension 1. 
\end{theorem}

We will prove this below.

First, we refer the reader to  \cite{JonssonBook2008} for the necessary background on discrete Morse matchings. The particular case of element matchings is described in \cite{DeshpandeSingh2021};  a brief summary appears in \cite[Appendix]{BDJRSX-TOTAL2024}.
We begin by showing that there is a Morse matching on  $\Delta_2(T)$ for any tree $T$.  If $G$ is a triangle-free connected graph, we can then choose a spanning tree $T$ of $G$, and restrict the Morse matching of $\Delta_2(T)$ to $\Delta_2(G)$.  

\begin{lemma}\label{lem:MorseMatchingTrees}  Let $T$ be any tree on $n$ vertices. Then $\Delta_2(T)$ admits a perfect Morse matching.
\end{lemma}
\begin{proof}
  We  construct an element matching for $\Delta_2(T)$. First  choose a root for $T$, labelled 1, and then label $T$ recursively with labels $1, 2, 3, \dotsc$, so that $x<y$ if $x$ is the parent of $y$ in $T$. We use this recursive labelling as the order for our element matching. By Proposition~\ref{prop:class-of-graphs-2022-3-26}, as $T$ is triangle-free, the $(n-3)$-dimensional cut complex $\Delta_2(T)$ has a complete $(n-4)$-skeleton.

For each vertex $a$, denote by  $\mathcal{M}_a$ the set of element matchings 
\[\{(\sigma, \sigma \cup a): 
\text{$a\notin \sigma$, and $\sigma$, $\sigma \cup a$ are both faces of $\Delta_2(G)$ that have not been matched by any $\mathcal M_b$, $b < a$}\},\]
and consider the sequence $\mathcal M_1, \mathcal M_2, \mathcal M_3, \dotsc$.

The unmatched faces after $\mathcal{M}_1$ are precisely the faces $\sigma$ not containing 1 such that $\sigma\cup 1$ is not a face; since $\Delta_2(T)$ has a complete $(n-4)$-skeleton,  $\sigma \cup 1$ must have dimension greater than $n-4$, and hence $\sigma$ is either a facet or a face of dimension $(n-4)$.

Suppose $\sigma$ is a face of dimension $(n-4)$, so that its (set) complement is $\sigma^c=\{1,x,y\}$. Since $\sigma \cup 1$ is NOT a face, $xy$ must be an edge. 
Also $\sigma$ must be contained in some facet, and $G$ is triangle-free, so exactly  one of $1x$, $1y$ is not an edge, and exactly one of $\sigma \cup x$, $\sigma\cup y$ is a facet.  Assume without loss of generality that  $x$ is the parent of $y$; thus $1x$ is the edge, and 
 the matching $\mathcal{M}_x$ matches the pair 
 $(\sigma, \sigma \cup x) $.  

This shows that all  $(n-4)$-dimensional faces are matched. 

Now suppose $\tau$ is an unmatched facet after $\mathcal{M}_1$, so that its complement $\tau^c=\{x,y\}$ where $x\ne 1, y\ne 1$, and $xy$ is not an edge. But then $x,y$ have a common ancestor $a\ne 1$, and $\{xya\}^c$ will be matched to $\tau^c$ by the element matching $\mathcal{M}_a$.

Since a sequence of element matchings is acyclic (see \cite[Appendix]{BDJRSX-TOTAL2024}, \cite{DeshpandeSingh2021}), we have exhibited a Morse matching for $\Delta_2(T)$. 
In fact we have shown that all faces are matched, and hence there are no critical cells, recovering the result of Corollary~\ref{cor:Delta2trees-contractible} that $\Delta_2(T)$ is contractible. 
\end{proof}

\begin{proof}[Proof of Theorem~\ref{thm:Delta2nonchordalSMark}]
 Let $G$ be a connected,  triangle-free graph. If $G$ is chordal, it must be a tree, and hence $\Delta_2(G)$ is shellable and contractible, and the theorem is verified, since the number of edges is one less than the number of vertices. 
 
 Now assume $G$ is nonchordal, so that $G$ is not a tree. Let $T$ be a spanning tree of $G$. Then  $\Delta_2(G)\subset \Delta_2(T)$, $\Delta_2(G)\ne \Delta_2(T)$. 
 We will show that there is a Morse matching for $\Delta_2(G)$ with $e-n+1$ critical cells in codimension 1, one below the top.

By the preceding lemma, the cut complex $\Delta_2(T)$ has a Morse matching $\mathcal{M}$ with no unmatched cells.  The restriction $\mathcal{M}\vert_{\Delta_2(G)}$ to the subcomplex $\Delta_2(G)$ remains acyclic, because there are fewer cells, and hence we have a Morse matching for $\Delta_2(G)$.    Unmatched faces in $\mathcal{M}\vert_{\Delta_2(G)}$ must be of the form $\sigma\setminus\{b\}$, where $\sigma$ is 
a face in $\Delta_2(T)\setminus \Delta_2(G)$ and $(\sigma\setminus\{b\},\sigma)$ was a matched pair under $\mathcal{M}$  in $\Delta_2(T)$. 
Since $G$ is triangle-free, $\Delta_2(G)$ contains a complete skeleton of codimension 1, so $\sigma\setminus\{b\}\in \Delta_2(G)$ and  $\sigma\in\Delta_2(T)\setminus \Delta_2(G)$ must be a face in the top dimension. 

Hence in the restriction $\mathcal{M}\vert_{\Delta_2(G)}$, there is exactly one critical cell of codimension 1 corresponding to each  facet  of $\Delta_2(T)\setminus \Delta_2(G)$. Each of the $e-(n-1)$ edges not in the spanning tree yields exactly one such facet, giving $e-(n-1)$ critical cells, where $e$ is the number of edges of $G$, and $n$ is the number of vertices.  It follows that 
$\Delta_2(G)$ is homotopy equivalent to a wedge of $(e-n+1)$ spheres of dimension $(n-4)$.  
\end{proof}

Theorem~\ref{thm:Delta2nonchordalSMark} is corroborated by  the nonchordal triangle-free families studied in \cite{BDJRSX-TOTAL2024}, namely complete bipartite graphs, cycles, and grid graphs.  In addition, we have the following examples.  Recall that the Kneser graph $K(m,r)$ \cite{WestGraphTheory1996} has one vertex for each $r$-subset of $[m]$, with an edge between subsets if and only if they are disjoint.

\begin{cor}\label{cor:Ex-Delta2-triangle-freeS}
For each of the following graphs $G$, $\Delta_2(G)$ is a wedge of spheres in
codimension 1.
\begin{enumerate}
\item Let $G$ be the Kneser graph $K(m,r)$, $m< 3r$.  Then $\Delta_2(G)$ is a wedge
      of $\frac{1}{2}\binom{m}{r}\binom{m-r}{r}-\binom{m}{r}+1$ spheres in
      dimension $\binom{m}{r}-4$.
\item In particular, for  the Petersen graph $G=K(5,2)$,
      $\Delta_2(G)$ is a wedge of $6$ spheres in dimension~6.
\item Let $G$ be the Cartesian product $H_1\times H_2$ of two triangle-free
      connected graphs $H_1$ and $H_2$, with $n_i = |V(H_i)|\ge 2$ and
     $m_i = |E(H_i)|$.  Then $\Delta_2(G) $ is a wedge
      of $n_1m_2+n_2m_1 - n_1n_2+1$ spheres in dimension $n_1+n_2-4$.
\end{enumerate}
\end{cor}

\begin{proof} 
The Kneser graph is connected and, by the pigeonhole principle, it is triangle-free if and only if  $m< 3r$.
It has $\binom{m}{r}$ vertices and $\frac{1}{2}\binom{m}{r}\binom{m-r}{r}$ edges, so the result follows. 

 One can check that Cartesian products preserve connectedness and the triangle-free property. 
\end{proof}

\begin{rem}\label{rem:counterexample:Delta2NOTtriangle-free}
 The hypotheses in Theorem~\ref{thm:Delta2nonchordalSMark} are necessary. 
If $G$ is not triangle-free, the  conclusion is false, the counterexample being multipartite graphs with three or more parts. (These contain a 4-cycle without a chord as an induced subgraph.) The smallest counterexample is the multipartite graph with $r\ge 4$ parts each with 2 vertices:
 we will prove in Proposition~\ref{prop:MultipartiteCase3} that $\Delta_2(G)$ has dimension $(2r-3)$, but is homotopy equivalent to one sphere $ \bbS^{r-2}$ in dimension $(r-2)$.

Likewise, if $G$ is triangle-free but not connected, the conclusion of 
Theorem~\ref{thm:Delta2nonchordalSMark} is again false.  Sage computations show that the disjoint union $\Delta_2(C_m+C_n)$ has homology of rank 1 in the top dimension and rank 2 in codimension 1, for $m\ge n\ge 4.$  Again, we already know homology is torsion-free and concentrated in the top two dimensions.

\end{rem}

\subsection{Applications to computing Betti numbers of cut complexes}

In this subsection we show how to use Theorem~\ref{thm:truncBoolean-minus-antichain} to compute the reduced Euler characteristic for cut complexes for some families of graphs. This in turn determines the Betti number if the homotopy type is a wedge of spheres in a single dimension, as is the case for the families studied in the next section.

\begin{prop}\label{prop:disj-unions} Fix $k\ge 2$ and let $G_i$ be a graph on $n_i$ vertices, $1\le i\le r$, such that for all $i$, the $k$-cut complex $\Delta_k(G_i)$ satisfies condition~\eqref{eqn:Condition-truncBoolean}  of Theorem~\ref{thm:truncBoolean-minus-antichain}. Let $N=\sum_{i=1}^r n_i$. Then the disjoint union $G_1+\dotsb +G_r$ also satisfies the condition, and has Euler characteristic equal to 
\[(-1)^{N-k-1}\left( \binom{N-1}{k-1} - \sum_{i=1}^r |\mathcal{Z}_k(G_i)|\right).\]
If, in addition, $\Delta_k(G_1+\dotsb +G_r)$ is nonvoid and 
shellable, its  Betti number is given by $(-1)^{N-k-1}$ times the above formula. 

In the special case when  each $\Delta_k(G_i)$ is nonvoid and shellable with Betti number $\beta_i$, $\Delta_k(G_1+\dotsb +G_r)$ is shellable with Betti number given by 
\[ \binom{N-1}{k-1} -\sum_{i=1}^r \binom{n_i-1}{k-1}  + \sum_{i=1}^r \beta_i.\]
\end{prop}
\begin{proof}  This is clear from Theorem~\ref{thm:truncBoolean-minus-antichain}, since $A$ is a connected subset of the disjoint union if and only if it is a connected subset of some $G_i$.  In particular the last two statements are a consequence of the shellability condition, by the final statement of 
Theorem~\ref{thm:truncBoolean-minus-antichain}. 
 Note that $|\mathcal{Z}_k(G_i)|=0$ if $n_i<k$. 
\end{proof}

For instance, for paths $P_{n_i}$, $1\le i \le r$, where each $P_{n_i}$ has $n_i$ vertices, we can use Proposition~\ref{prop:cut-complex-trees} in the next section to compute that the rank of the unique nonvanishing homology group of $P_{n_1}+\dotsb +P_{n_r}$ is given by the explicit formula $\binom{N-1}{k-1}-\sum_{i \in \{1, \dotsc, r\}, n_i \ge k}
(n_i-k+1)$.

Next we consider the wedge of two graphs, as defined in Section~\ref{subsec:Mark-Wedges}.
Examples are the wedge of two trees (also a tree), two  paths end-to-end (also a path), two cycles (a figure 8), a cycle and a tree, a cycle and a path (balloon).
\begin{prop}\label{prop:S-wedges} Let $G_i$ be a graph on $n_i$ vertices, $i=1,2$. Let  $k\ge 2$, $k\le n_1+n_2-3$.
  Let $G_1 \vee_{w_0} G_2$ be the wedge  of the two graphs at some vertex $w_0$.  If $G_i, i=1,2$ satisfy the condition~\eqref{eqn:Condition-truncBoolean}, then so does $G_1 \vee_{w_0} G_2$.  Furthermore, its Euler characteristic $\mu(G_1 \vee_{w_0} G_2)$ satisfies 
\[(-1)^{n_1+n_2-k-2}\mu(G_1 \vee_{w_0} G_2)=\binom{n_1+n_2-2}{k-1}-|\mathcal{Z}_k(G_1)|\delta_{k\le n_1} -|\mathcal{Z}_k(G_2)| \delta_{k\le n_2}- \bar\xi_k(G_1,G_2, w_0),\]
where $\delta_{A}$ is the Kronecker delta, equal to 1 or 0 according as the statement $A$ is true or false, and 
\[ \bar\xi_k(G_1,G_2, w_0)=|\{(A_1,A_2)\in V(G_1)\times V(G_2): w_0\in A_i, \ A_i^c\in \mathcal{Z}_{r_i}(G_i), \ r_1+r_2=k+1, \ A_i\setminus\{w_0\}\ne \emptyset\}|.\] 
In particular, if  $\Delta_k(G_i)$ is shellable for $i=1,2$, then so is 
$\Delta_k(G_1 \vee_{w_0} G_2)$, with Betti number given by 
\[(-1)^{n_1+n_2-k-2}\mu(G_1 \vee_{w_0} G_2)=\binom{n_1+n_2-2}{k-1}- \xi_k(G_1,G_2, w_0),\]
where 
\[ \xi_k(G_1,G_2, w_0)=|\{(A_1,A_2)\in V(G_1)\times V(G_2): w_0\in A_i,\ A_i^c\in \mathcal{Z}_{r_i}(G_i),\ r_1+r_2=k+1\}|.\] 
\end{prop}

\begin{proof}  Denote the vertex set of $G_i$ by $V(G_i), i=1,2$. We assume $V(G_1)\cap V(G_2)=\{w_0\}$, the wedge point. The connected subsets of the wedge $G_1 \vee_{w_0} G_2$ fall into three categories (some of which may be empty):
\begin{enumerate}
\item Connected subsets $A_1$ of  $V(G_1)$ of size $k$ such that $A_1\cap V(G_2)\subseteq \{w_0\}$, provided $k\le n_1$;
thus $A_1^c\in \mathcal{Z}_k(G_1)\subset \mathcal{Z}_k(G_1 \vee_{w_0} G_2)$.  
\item Connected subsets $A_2$ of  $V(G_2)$ of size $k$ such that $A_2\cap V(G_1)\subseteq \{w_0\}$, provided $k\le n_2$; 
thus $A_2^c\in \mathcal{Z}_k(G_2)\subset \mathcal{Z}_k(G_1 \vee_{w_0} G_2)$. 
\item Connected sets in $G_1 \vee_{w_0} G_2$ of size $k$, of the form $A_1\cup A_2$ where for $ i=1,2$ we have $w_0\in A_i\subset V(G_i)$ and  $A_i\setminus\{w_0\}\ne \emptyset$;  thus  
$A_i\cap V(G_{3-i})= \{w_0\}$ and $|A_1\cup A_2|=k=|A_1|+|A_2|-1$. Note that $(A_1\cup A_2)^c \in \mathcal{Z}_k(G_1 \vee_{w_0} G_2)$.
\end{enumerate}

Clearly these account for all the connected subsets of size $k$ of $G_1 \vee_{w_0} G_2$. 
We claim that condition~\eqref{eqn:Condition-truncBoolean-2} is satisfied for all three cases. 
For Case (3) it is clear that removing $w_0$ from $A_1\cup A_2$ results in two disconnected components, and hence makes $((A_1\cup A_2)\setminus \{w_0\}) \cup\{x\}$ 
disconnected for any $x\notin A_1\cup A_2$.

For Case (1), let $X=A_1\cup\{x\}$, $x\notin A_1$.  If $x\in V(G_1)$ we are done since $G_1$ satisfies condition~\eqref{eqn:Condition-truncBoolean-2}.  

Otherwise $x\in V(G_2)\setminus V(G_1)$. If $w_0\notin A_1$, then $A_1\cup\{x\}$ is itself disconnected and we are done. 
If $w_0\in A_1$, then clearly $(A_1\setminus\{w_0\}) \cup\{x\}$ results in a disconnected set of size $k$.   Case (2) follows similarly.

Since the number of vertices of the wedge is $n_1+n_2-1$,  the dimension of the cut complex $\Delta_k(G_1 \vee_{w_0} G_2)$ is $n_1+n_2-2-k$ and the expression for the Euler characteristic  follows from Theorem~\ref{thm:truncBoolean-minus-antichain}.
\end{proof}

For instance, the wedge of two paths $P_{n_1}$ and $P_{n_2}$ wedged at a leaf is the path $P_{n_1+n_2-1}$. Assume $k\le\min(n_1-2, n_2-2)$; one can check that in this case $\xi_k(P_{n_1},P_{n_2}, w_0)=k-2$, and hence  the formula of Proposition~\ref{prop:S-wedges} 
agrees with Proposition~\ref{prop:cut-complex-trees} in the next section.

We record the following computation that will be needed to write down a formula for the Betti numbers for wedges of cycles and paths.

\begin{lemma}\label{lem:count-connected-subsets}  Assume $k\ge 1$. 
\begin{enumerate}
\item Let $n\ge 1$. Then
$|\mathcal{Z}_k(P_n)|=n-k+1$ if $k\le n$ and $|\mathcal{Z}_k(P_n)|$ is zero otherwise.
\item
Let  $n\ge 3$. Then $|\mathcal{Z}_k(C_n)|=n$ if $k<n$, $|\mathcal{Z}_n(C_n)|=1$, and $|\mathcal{Z}_k(C_n)|$ is zero otherwise.
\item
Let $\xi_n(a)$ denote the number of connected subsets of size $a$ in the cycle $C_n$ passing through a fixed vertex $w_0$. Then $\xi_n(a)=a$ if $a<n$, $\xi_n(n)=1$, and $\xi_n(a)$ is zero if $a>n$.
\item
Let  $n_1\ge 3$, $n_2\ge 1$.  Let $w_0$ be a leaf of $P_{n_2}$ and let $G$ be the wedge $G=C_{n_1}\vee_{w_0} P_{n_2}$.  Then 
\begin{equation}\label{eqn:Zk-cycle-path-wedge}
\bar\xi_k(G)=\sum_{a=\max(2,k+1-n_2)}^{\min(n_1, k-1)} \xi_n(a)
\end{equation}
\end{enumerate}
\end{lemma}
\begin{proof} Only Equation \eqref{eqn:Zk-cycle-path-wedge} requires comment.  By definition of $\bar\xi_k(G)$, we are counting connected subsets $A$ of size $k$, passing through $w_0$, consisting of $a\ge 2$ vertices in $C_{n_1}$ through $w_0$ and 
$k-a+1\ge 2$ vertices in $P_{n_2}$, also passing through $w_0$. Thus $a\le n_1$ and $k-a+1\le n_2$.  Since the set  $A$ is uniquely determined by its intersection with $C_{n_1}$,  the expression  follows.
\end{proof}

Now define $\xi_k(C_{n_1}, P_{n_2}, w_0)$ to be the sum 
\[\xi_k(C_{n_1}, P_{n_2}, w_0) \coloneqq \sum_{a=\max(1,k+1-n_2)}^{\min(n_1, k)} \xi_n(a).\]
\begin{prop}\label{prop:special-wedges-balloons} Let $G=C_{n_1} \vee_{w_0} P_{n_2}$ where $C_n$ is the cycle graph on $n$ vertices and $P_n$ is the path on $n$ vertices, and $w_0$ is a leaf of the path.   Then $G$ has $n_1+n_2-1$ vertices. If $k=n_1+n_2-2$, $\Delta_k(G)$ is a single point $\{w_0\}$.  Assume $2\le k\le n_1+n_2-3$.
If $k\ne n_1-1$, then  $G$ satisfies condition~\eqref{eqn:Condition-truncBoolean-2} and $\Delta_k(G)$ has reduced Euler characteristic $\mu(\Delta_k(G))$ where 
\begin{align*}
(-1)^{n_1+n_2-1-k} \mu(\Delta_k(G)) & =\binom{n_1+n_2-2}{k-1} -|\mathcal{Z}_k(C_{n_1})| -(n_2-k+1) \delta_{k\le n_2} -\bar\xi_k(C_{n_1}, P_{n_2}, w_0)\\
& = \binom{n_1+n_2-2}{k-1} - \xi_k(C_{n_1}, P_{n_2}, w_0).
\end{align*} 

If $k\ge 3$, $k\ne n_1-1$, , 
$\Delta_k(G)$ is shellable, and has the homotopy type of a wedge of  $(-1)^{n_1+n_2-1-k} \mu(\Delta_k(G))$ spheres in dimension $n_1+n_2-2-k$.

If $k=2$, $k\ne n_1-1$, we obtain $(-1)^{n_1+n_2-3} \mu(\Delta_k(G))=-1$.  Nonzero homology can only be in the top two dimensions.  
\end{prop}

\begin{proof} First we observe that the $k$-cut complex of the cycle graph $C_n$ satisfies the condition~\eqref{eqn:Condition-truncBoolean-2}  provided $k\le n-1$. If $k=n-1,$ the condition fails because all subsets of $C_n$ of size $n-1$ are connected. 

By definition, $\bar\xi_k(C_{n_1},P_{n_2}, w_0)$ 
counts connected $k$-subsets passing through the wedge point $w_0$, and these consist of a path with $k=a+1\ge 2$ vertices in $P_{n_2}$ starting at $w_0$, attached to a path in $C_{n_1}$ through $w_0$ with  $a\ge 2$ vertices.  The precise counts are provided in Lemma~\ref{lem:count-connected-subsets}.

When $3\le k$, we know from Proposition~\ref{prop:class-of-graphs-2022-3-26}  that $\Delta_k(C_{n_1})$ and $\Delta_k(P_{n_2})$ satisfy condition~\eqref{eqn:Condition-truncBoolean-2}, and hence so does $\Delta_k(G)$, by Proposition~\ref{prop:S-wedges}.  The Betti number follows from the general expression for the reduced Euler characteristic. \qedhere
\end{proof}

A completely analogous argument gives us the following for a wedge of two cycles.

\begin{prop}\label{prop:figure-eight} Consider the wedge  $G= C_{n_1}\vee_{w_0} C_{n_2}$ where $C_n$ is the cycle graph on $n$ vertices, and $w_0$ is a common vertex of both $C_{n_1}$ and $C_{n_2}$. Then $G$ has $n_1+n_2-1$ vertices.  If $k=n_1+n_2-1$, 
$\Delta_k(G)$ is a point $\{w_0\}$. Assume $2\le k\le n_1+n_2-3.$ 
If $k\ne n_i-1, i=1,2$, then $G$ satisfies  condition~\eqref{eqn:Condition-truncBoolean-2} and $\Delta_k(G)$ has reduced Euler characteristic $\mu(\Delta_k(G))$ where 
\begin{align*}
(-1)^{n_1+n_2-1-k} \mu(\Delta_k(G)) & = \binom{n_1+n_2-2}{k-1} -|\mathcal{Z}_k(C_{n_1})| -
|\mathcal{Z}_k(C_{n_2})| -\bar\xi_k(C_{n_1}, C_{n_2}, w_0)\\
& = \binom{n_1+n_2-2}{k-1} - \xi_k(C_{n_1}, C_{n_2}, w_0),
\end{align*} 
where 
\[\xi_k(C_{n_1}, C_{n_2}, w_0)=\sum_{a=\max(1, 1+k+1-n_2)}^{\min(n_1,k)} 
\xi_{n_1}(a)\xi_{n_2}(k+1-a).\]

If $k\ge 3$,
$\Delta_k(G)$ is shellable, and has the homotopy type of a wedge of  $(-1)^{n_1+n_2-1-k} \mu(\Delta_k(G))$ spheres in dimension $n_1+n_2-2-k$.

If $k=2$ we obtain $(-1)^{n_1+n_2-3} \mu(\Delta_k(G))=-1$.  Nonzero homology can only be in the top two dimensions. 
\end{prop}
\begin{proof}  We  invoke Proposition~\ref{prop:S-wedges} again, noting that the caveat  $k\ne n_i-1$, $i=1,2$, still applies.  It remains to observe that 
\[\xi_k(C_{n_1}, C_{n_2}, w_0)=\sum_{a=\max(1, 1+k+1-n_2)}^{\min(n_1,k)} 
\xi_{n_1}(a)\xi_{n_2}(k+1-a),\]
but this follows since we are counting connected subsets of size $k$ passing through the wedge point $w_0$; these separate into two parts, one in each cycle, intersecting in $w_0$. 
Again Lemma~\ref{lem:count-connected-subsets} provides the precise counts.
\end{proof}

\section{Families of graphs}\label{sec:Families}

\subsection{Complete Multipartite Graphs}\label{sec:new-multipartite-2022April23S}

In this section we determine the homotopy type of the cut complex $\Delta_k(G)$ for various families of graphs $G$.  With the exception  of squared cycle graphs, the families we consider have the property that their homology is concentrated in a single dimension.

Let $E_n$ denote the edgeless graph on $n$ vertices. 
The complete bipartite graph $K_{m,n}$ is the join of two edgeless graphs $E_m$ and $E_n$. 
In this section we will first determine completely the equivariant homotopy type for cut complexes of $K_{m,n}$, and then proceed to do the same for multipartite graphs.  Note that the automorphism group of a complete multipartite graph acts simplicially on the cut complexes.

 Let $G=K_{m,n}$ be the complete bipartite graph with $1\le m\le n.$ We point out for clarity that if we label the vertices with the sets $[m]$ and $m+[n]\coloneqq\{m+1, \dots, m+n\},$ so that the edges are $(i,j)$ for $1\le i \le m$ and $m+1\le j\le m+n,$  then for each $k\ge 2,$ the nonvoid cut complex $\Delta_k(G)$
 has facets $F$ of size $(m+n-k)$ where $F$ must contain the vertex subset $[m]$ or the vertex subset $m+[n].$

 Theorem~\ref{thm:bipartite} was proved in \cite[Theorem~3.3]{BDJRSX-TOTAL2024} by different methods.  Here we deduce it from the structure theorems of Section~\ref{sec:Constructions}. 
From Theorem~\ref{thm:Mark-Joins}, Theorem~\ref{thm:susp-of-joins2022Sept12} and Proposition~\ref{prop:Deltak-edgeless-graphS}, we have the following, since $K_{m,n}=E_m*E_n$:
\begin{theorem} \label{thm:bipartite}
Let $1\le m\le n$ and $2\le k$.  
\begin{enumerate}
\item
$\Delta_k(K_{m,n})$ is shellable if and only if  $m<k$.
Furthermore, if $m<k\le n$, then $\Delta_k(K_{m,n})$ is contractible, and if $k>n$, the cut complex is void and hence shellable. 
\item
If $k\le m \le n$, the $(m+n-k-1)$-dimensional complex $\Delta_k(K_{m,n})$ is homotopy equivalent to a wedge of $\binom{m-1}{k-1}\binom{n-1}{k-1}$ spheres of dimension $m+n-2k$.
\end{enumerate}
\end{theorem}

\begin{proof} From Theorem~\ref{thm:Mark-Joins}, $\Delta_k(K_{m,n})$ is shellable if and only if one of $\Delta_k(E_m)$, $\Delta_k(E_n)$ is shellable and the other is void.  Since $1\le m\le n$, the first statement follows from Proposition~\ref{prop:Deltak-edgeless-graphS}.  The contractibility when $m<k\le n$ follows from Part (3) of Theorem~\ref{thm:susp-of-joins2022Sept12}.

If $k>n$ both $\Delta_k(E_m)$, $\Delta_k(E_n)$ are void and again the result follows from Theorem~\ref{thm:Mark-Joins}.

Now let $k\le m\le  n$. From Proposition~\ref{prop:Deltak-edgeless-graphS} and Equation~\eqref{eqn:join-susp-iso} of Theorem~\ref{thm:susp-of-joins2022Sept12}, the cut complex is homotopy equivalent to 
\[\susp (\Delta_k(E_m)*\Delta_k(E_n))
\simeq\! \bigvee_{\binom{m-1}{k-1}\binom{n-1}{k-1}} \bbS^0*\, (\bbS^{m-k-1}*\bbS^{n-k-1})\simeq\! \bigvee_{\binom{m-1}{k-1}\binom{n-1}{k-1}} \bbS^{m+n-2k}. \qedhere\]
\end{proof}

Thus bipartite graphs are an obstruction to shellability: if a graph $H$ contains a bipartite graph $K_{m,n}$ such that $k=m=n$ or
$k\le m< n$  as an induced subgraph, then $\Delta_k(H)$ is not shellable.

In order to describe the homology representation for the case $k\le m\le n$, we will need a well-known result of Solomon, as well as a  result about group actions on products of posets.

\begin{theorem}[{\cite{SolJAlg1968}, \cite{RPSGaP1982}}] \label{thm:SolomonTruncatedBoolean}
The $\mathfrak{S}_p$-representation on the unique nonvanishing homology module of the order complex of $P(p,k)$ is the irreducible $V_\lambda$ indexed by the integer partition $\lambda=(k, 1^{p-k}).$
\end{theorem}

Note that a special case of this is the classical result that  $\mathfrak{S}_p$ acts on the homology of the boundary of a $(p-1)$-simplex like the sign.

\begin{theorem}[{\cite[Proposition 2.3]{Jer93}}] \label{thm:action-on-products}
Let $Q$ be any poset. Assume the nonzero reduced homology of $Q$ is  concentrated in dimension $d.$ Then the $r$-fold product of $Q$ has nonzero reduced homology concentrated in dimension ${dr+2r-2},$ and the symmetric group  $\mathfrak{S}_r$ acts on this unique nonvanishing homology module   like $\mathrm{(sgn_{\mathfrak{S}_r})\,}^d,$ the $d$th power of the sign representation of $\mathfrak{S}_r$.
\end{theorem}

By Theorem~\ref{thm:SolomonTruncatedBoolean}, the representation of $\mathfrak{S}_n$ on the unique nonvanishing reduced homology $\tilde{H}_{n-k-1}(\Delta_k(E_n))$ is the irreducible $V_{(k, 1^{n-k})}$, 
for $2\le k\le n$. 

Denote by $\mathfrak{S}_m[\mathfrak{S}_n]$ the wreath product group of $\mathfrak{S}_m$ with $\mathfrak{S}_n$, with $\mathfrak{S}_m$ acting on $m$ copies of $\mathfrak{S}_n$, so that its order is $m! (n!)^m$. 
Let $V$, $W$ be $\mathfrak{S}_m$-{} and $\mathfrak{S}_n$-modules, respectively. The wreath product module  $V[W]$ is the tensor product of vector spaces $W^{\otimes m}\otimes V$, equipped with  a canonical action of  $\mathfrak{S}_m[\mathfrak{S}_n]$. %
See \cite[Section~4.3]{JamesKerber1981}.

We compute the equivariant  homotopy type of the cut complex $\Delta_k(K_{m,n})=\Delta_k(E_m*E_n)$ of the complete bipartite graph  in the case $k\le m$ using  Theorem~\ref{thm:susp-of-joins2022Sept12}.  In fact that theorem immediately gives:

\begin{theorem}\label{thm:BipartiteEquiv:k<=m<=n} Let $G=K_{m,n}$ be the complete bipartite graph with $1\le m\le n.$ Let $2\le k\le m.$  
Recall that $P(p,k)$ is the face lattice of the $(p-k-1)$-skeleton of a $(p-1)$-simplex.
 Then there is a poset map  from   the face lattice  $\mathcal{L}(\Delta_k(G))$ of the cut complex to the  product of posets $P(m,k)\times P(n,k)$, which induces a group-equivariant homotopy equivalence between the respective order complexes:
 \[\Delta(\overline{P(m,k)\times P(n,k)}) \simeq\susp (\Delta(\overline{P(m,k)}) * \Delta(\overline{P(n,k)})) \simeq\Delta_k(G).\]
\noindent
 Hence the $(m+n-k-1)$-dimensional cut complex $\Delta_k(G)$ has the homotopy type of a wedge of $\binom{m-1} {k-1}\binom{n-1} {k-1}$ spheres in dimension $m+n-2k.$  In particular, it is not shellable. 

If $m<n,$ the automorphism group is $\mathfrak{S}_m\times \mathfrak{S}_n$, and the representation on the homology is
the irreducible $V_{(k, 1^{m-k})}\otimes V_{ (k, 1^{n-k})},$ indexed by the pair of integer partitions $((k, 1^{m-k}), (k, 1^{n-k})).$

If $m=n,$ the automorphism group is the wreath product
$\mathfrak{S}_2[\mathfrak{S}_n],$ and the representation on the homology of the cut complex is $\Phi_{\mathfrak{S}_2}[V_{(k, 1^{n-k})}],$
where the one-dimensional representation $\Phi_{\mathfrak{S}_2}$ equals $\mathrm{sgn}^{n-k-1},$ and $\mathrm{sgn}$ is the sign representation of $\mathfrak{S}_2.$
\end{theorem}
\begin{proof}  The only statement that needs to be verified is the homology representation, since the homotopy type 
was determined in Theorem~\ref{thm:bipartite}. (Alternatively, it follows immediately from  Theorem~\ref{thm:susp-of-joins2022Sept12} and Equation~\ref{eqn:prod-two-posets}.)  
The homology representation is clear for the automorphism group $\mathfrak{S}_m\times \mathfrak{S}_n$, because the poset map induces an equivariant isomorphism in homology. 

Note that when $k=m<n$, $\Delta_m(E_m)$ is the empty complex, and
 $\Delta_k(G)\sim\susp \Delta(\overline{P(n,k)})  $.

When $m=n$, the full automorphism group is the wreath product group $\mathfrak{S}_2[\mathfrak{S}_n]$ of order $2 (n!)^2$ (with $\mathfrak{S}_2$ acting on two copies of $\mathfrak{S}_n$), and the action of $\mathfrak{S}_2$ is given as stated, by Theorem~\ref{thm:action-on-products}.
\end{proof}

The equivariant homotopy and homology representation  determined in Theorem~\ref{thm:BipartiteEquiv:k<=m<=n} bear a striking resemblance to that of a complex studied by  Linusson, Shareshian and Welker  \cite{LinSharesh2003}. In that paper, the authors consider the following simplicial complex of graphs.  For positive integers $k,m,n$, they define $B_k(m,n)$ to be the simplicial complex consisting of (edge-sets of) all bipartite graphs, with bipartite vertex set $[m]\cup [n]$, that do not contain a matching of size $k$. Clearly we may assume $k\le m \le n$. Their result is the following.

\begin{theorem}[{\cite[Theorem 1.4]{LinShareshWelker2008}}] \label{thm:Lin-Sharesh-Welker2008}
The complex $B_k(m,n)$ has the homotopy type of a wedge of $\binom{m-1}{k-1}\binom{n-1}{k-1}$ spheres of dimension $2k-3$. 
\end{theorem}

They then make  the conjecture \cite[Conjecture 1.15]{LinShareshWelker2008} that (implicitly, for $k<m<n$)  the representation of the  group $\mathfrak{S}_m\times \mathfrak{S}_n$ on the unique nonvanishing homology is, in our notation, $(\mathrm{sgn}_{\mathfrak{S}_m}\otimes V_{(k, 1^{m-k})})\otimes (\mathrm{sgn}_{\mathfrak{S}_n}\otimes V_{ (k, 1^{n-k})})$.  It would be interesting to know if there is a topological connection between their complex and ours.

We turn now to  multipartite graphs, where the results are similar.  Let $G$ be the complete multipartite graph $K_{m_1,\ldots,m_r}$, $r\ge 3$. We may assume $m_1\le \dotsb \le m_{r-1} \le m_r$. Then 
$G=E_{m_1}*\dotsb * E_{m_r}$, and by associativity, $G=G_1*G_2$ where $G_1=
K_{m_1,\ldots,m_t}$ and $G_2=K_{m_{t+1},\ldots,m_r}$ for any $t$, $1\le t\le r-1$. 

\begin{theorem}\label{thm:MultipartiteShellable}
 Let $m_1\le \dotsb \le m_{r-1} \le m_r, r\ge 3.$
Let $G$ be the complete multipartite graph $K_{m_1,\ldots,m_r}$.
Then 
\begin{enumerate} 
\item If $m_r < k$, then $\Delta_k(G)$ is void and hence shellable.
\item If $m_{r-1}<k\le m_r$, then $\Delta_k(G)$ is contractible and shellable.
\item If $m_1 < k \le m_{r-1}$, then $\Delta_k(G)$ is contractible and not
      shellable.
\item If $k\le m_1$, then $\Delta_k(G)$ is not shellable.
\end{enumerate}   
\begin{proof} Item (1) is clear.

So assume $k\le m_r$, and consider Item (2). Take $G_1=K_{m_1,\ldots,m_{r-2}}$ and $G_2=K_{m_{r-1}, m_r}$, so
$G=G_1*G_2$. If $k\le m_{r-1}$, Theorem~\ref{thm:bipartite} tells us that 
the cut complex $\Delta_k(G_2)$ of the bipartite graph is not  
shellable, and hence Theorem~\ref{thm:Mark-Joins} shows that $\Delta_k(G)$ is 
not shellable.  On the other hand, if
$m_{r-1}<k\le m_r$, then $\Delta_k(G_1)$ is void and $\Delta_k(G_2)$ is shellable and 
not void, so by Theorem~\ref{thm:susp-of-joins2022Sept12} $\Delta_k(G)$
is contractible and shellable.

Next consider Item (3). If $m_t < k \le m_{t+1}$ for $1\le t\le r-2$, now write $G=G_1*G_2$ with
$G_1=K_{m_1,\ldots, m_t}$ and $G_2=K_{m_{t+1},\ldots, m_r}$.
Since $\Delta_k(G_1)$ is void, and $\Delta_k(G_2)$ is not void,
Part (2) of Theorem~\ref{thm:susp-of-joins2022Sept12} says that $\Delta_k(G)$
is contractible (but by above not shellable).

Finally, for Item (4), if $k\le m_1$, let $G_1=K_{m_1,m_2} $ and $G_2=K_{m_3,\ldots,m_r}$, so that $G=G_1 *G_2$.  Then Part (1) of Theorem~\ref{thm:bipartite} says that $\Delta_k(G_1)$ is not shellable, and hence Theorem~\ref{thm:Mark-Joins} shows that $\Delta_k(G)$ is 
not shellable.
\end{proof}
\end{theorem}

For complete bipartite graphs, Theorem~\ref{thm:BipartiteEquiv:k<=m<=n} asserts that $\Delta_n(K_{n,n}) \simeq \bbS^0.$ The next result generalizes this to complete multipartite graphs.

\begin{prop}\label{prop:MultipartiteCase3} Let $G=K_{\underbrace{\scriptstyle n,n,\ldots,n}_r}$ be the complete multipartite graph whose vertex set is a disjoint union of $r$ sets of size $n,$ $r\ge 2.$ Then the $(nr-n-1)$-dimensional cut complex $\Delta_n(G)$ is homotopy equivalent to a single sphere $\bbS^{r-2}$ in dimension $r-2.$ The  automorphism group is the wreath product $\mathfrak{S}_r[\mathfrak{S}_n],$ with $\mathfrak{S}_r$ permuting the $r$ copies of $\mathfrak{S}_n.$  The one-dimensional  representation afforded by the unique nonzero homology module in degree $r-2$ is 
\[\mathrm{sgn}_{\mathfrak{S}_r} [1_{\mathfrak{S}_n}].\]
\end{prop}
\begin{proof} We have $G=E_n*\dotsb*E_n,$ and $\Delta_n(E_n)=\{\emptyset\}\simeq\bbS^{-1}$.
Iterating Theorem~\ref{thm:susp-of-joins2022Sept12} gives the equivariant isomorphism of face lattices 
\[\mathcal{L}(\Delta_n(G))\cong\underbrace{\mathcal{L}(\Delta_n(E_n))\times\dotsb\times  \mathcal{L}(\Delta_n(E_n))}_r;\] using the fact that $\mathbb{S}^a*\mathbb{S}^b=\mathbb{S}^{a+b+1}$, this (with Equation~\ref{eqn:prod-r-posets}) gives the homotopy equivalence 
\[\Delta_n(G)\simeq \underbrace{\mathbb{S}^0*\dotsb *\mathbb{S}^0}_{r-1}*\underbrace{\Delta_n(E_n)*\dotsb*\Delta_n(E_n)}_r\simeq \mathbb{S}^{r-2}*\mathbb{S}^{-1}\simeq \mathbb{S}^{r-2}.\]

The equivariant isomorphism of face lattices makes the homology of $\Delta_n(G)$ isomorphic to 
the $r$-fold tensor product of $\tilde{H}_{-1}(\Delta_n(E_n))$.  Clearly $\mathfrak{S}_n$ acts trivially on each factor, and $\mathfrak{S}_r$ acts like $(\sgn)^{-1}=\sgn$ by Theorem~\ref{thm:action-on-products}. This finishes the proof. \end{proof}

For the remainder of this section, we record the following notational conventions, which will be helpful in describing the homology representations. Assume $M=\{m_1,\ldots, m_r\}$ is the multiset consisting of $r_i$ indices equal to $i$.  The automorphism group of the complete multipartite graph $G=K_{m_1,\ldots, m_r}$ is the group $ \mathfrak{G}_M$, where $\mathfrak{G}_M$ is the product of wreath products $\mathfrak{G}_M = \bigtimes_{i} \mathfrak{S}_{r_i}[\mathfrak{S}_i]$. 

The next case is a straightforward generalization of the bipartite case,  Theorem~\ref{thm:BipartiteEquiv:k<=m<=n}.  Recall that $P(n,m)$ is the truncated Boolean lattice $B_{n}^{\le n-m}\cup\{\hat 1\},$ and it is the face lattice of the $(n-m-1)$-skeleton of an $(n-1)$-simplex.   
 From Part (1) of Theorem~\ref{thm:susp-of-joins2022Sept12} and Theorem~\ref{thm:action-on-products}, we obtain:
\begin{theorem}\label{thm:MultipartiteCase4} Let $2\le k<m_1\le m_2\le \cdots\le m_r,$ and let $G$ be the complete multipartite graph $G=K_{ m_1,\cdots,  m_r}.$ There is a group-equivariant homotopy equivalence
\[ \Delta_k(G)\simeq \Delta (\overline{P(m_1,k)\times \dotsb \times P(m_r,k)}  ).\] 
Hence the $((\sum_{i=1}^r m_i)-k-1)$-dimensional cut complex $\Delta_k(G)$ has the homotopy type of a wedge of 
$\prod_{i=1}^r \binom{m_i-1}{k-1}$ spheres in dimension $\sum_{i=1}^r (m_i-k)+(r-2).$
 
The homology representation of the automorphism group $\mathfrak{G}_M$ is then the tensor product of wreath product representations 
\[ \bigotimes_{i\ge k+1} \mathrm{\Phi_{\mathfrak{S}_{r_i}}\,}[V_{(k, 1^{i-k})}],\]
where $\Phi_{\mathfrak{S}_{r_i}}$ is the one-dimensional representation of $\mathfrak{S}_{r_i}$ equal to $(\mathrm{sgn\,})^{i-k-1}.$  
\end{theorem}

We end this section with the last case of multipartite graphs that needs to be examined.
If $m_1=\ldots =m_t=m$ and $m<m_{t+1}\le \dotsb \le m_{t+r},$ for convenience we write 
$K_{m^t, m_{t+1}, \ldots, m_{t+r}}$ for the complete multipartite graph $K_{\underbrace{\scriptstyle m,\ldots,m}_t,m_{t+1}, \ldots , m_{t+r}}$.   Once again the theorem follows immediately from Part (1) of Theorem~\ref{thm:susp-of-joins2022Sept12} and Theorem~\ref{thm:action-on-products}, invoking in addition Proposition~\ref{prop:MultipartiteCase3} and Theorem~\ref{thm:MultipartiteCase4}.  The case $r=1$ is listed separately only for clarity.

\begin{theorem}\label{thm:MultipartiteCase6} 
 Let $m=m_1=\cdots = m_t< m_{t+1}\le \dotsb \le m_{t+r},$  where  $t\ge 1$.  
\begin{enumerate}
\item Let $ r\ge 2.$
 Then we have the equivariant  homotopy equivalence of cut complexes
\[\Delta_m(K_{m^t,m_{t+1}, \ldots, m_{t+r}})\simeq \mathrm{susp\,} (\,\Delta_m(K_{m^t}) * \Delta_m(K_{m_{t+1}, \ldots, m_{t+r}})\,)\]
which in turn is homotopy equivalent to 
$\bbS^{t-1}*\Delta_m(K_{m_{t+1}, \ldots, m_{t+r}})$. Let $N=\sum_{i=1}^r m_{t+i}$. The $\left( (t-1)m + N - 1 \right)$-dimensional cut complex $\Delta_m(K_{m_{t+1}, \ldots, m_{t+r}})$ is homotopy equivalent to a wedge of $\prod_{i=1}^r \binom{m_i-1}{m-1}$ spheres in dimension $t+N-rm+(r-2)$. The homology representation of the automorphism group $\mathfrak{S}_t[\mathfrak{S}_m] \times \mathfrak{G}_M,$ where $\mathfrak{G}_M$ is the automorphism group of $K_{m_{t+1}, \ldots, m_{t+r}},$ is 
\[  V_{(1^t)} [V_{(m)}] \otimes  \bigotimes_{i\ge m+1} \mathrm{\Phi_{\mathfrak{S}_{r_i}}\,}[V_{(m, 1^{i-m})}]  .\]
\item Let $r=1.$ Then we have the equivariant  homotopy equivalence of  complexes
\[\Delta_m(K_{m^t,m_{t+1}})\simeq \susp (\Delta_m(K_{m^t}) * \Delta(P(m_{t+1},m))),\]
which in turn is homotopy equivalent to 
$\bbS^{t-1}*\Delta(P(m_{t+1},m)\,).$
The order complex  $\Delta(P(m_{t+1},m))$ is the barycentric subdivision of the $(m_{t+1}-m-1)$-skeleton of an $(m_{t+1}-1)$-dimensional simplex.   
The $\left( (t-1)m + m_{t+1} - 1 \right)$-dimensional cut complex $\Delta_m(K_{m^t, m_{t+1}})$ is homotopy equivalent to a wedge of $ \binom{m_{t+1}-1}{m-1}$ spheres in dimension $t+m_{t+1}-m-1  .$   The homology representation of the automorphism group $\mathfrak{S}_t[\mathfrak{S}_m] \times \mathfrak{S}_{m_{t+1}}$  is 
\[  V_{(1^t)} [V_{(m)}]  \otimes V_{(m, 1^{m_{t+1}-m})}  .\]
\end{enumerate} 
\end{theorem}

\subsection{Cycle Graphs}\label{sec:Cycles-Dane-S-MarijaDMT}

For the cycle graph $C_n$,  we may assume $n\ge 4$ and $2\le k\le n-2$. If $n=4$, the cut complex $\Delta_2(C_4)$ is 1-dimensional with two facets $\{1,3\}$ and $\{2,4\}$, and is homotopy equivalent to the 0-sphere $\bbS^0$;  it is thus NOT shellable. In fact $\Delta_2(C_4)=\Delta_2(K_{2,2})$.

Recall from Figure~\ref{fig:Delta2C5} that the 2-dimensional cut complex $\Delta_2(C_5)$ is in fact a M\"obius strip, and hence it is homotopy equivalent to the one-sphere $\bbS^1$. This holds more generally  for the $(n-3)$-dimensional cut complex $\Delta_2(C_n)$: it was shown in  \cite{BDJRSX-TOTAL2024} that:
\begin{prop}[{\cite[Theorem~3.9]{BDJRSX-TOTAL2024}}] \label{prop:MarijaDMT-Delta2cycle}
Let $C_n$ be a cycle graph, $n \ge 5$. 
Then the $(n-3)$-dimensional  cut complex $\Delta_2(C_n)$ is homotopy equivalent to the sphere $\bbS^{n-4}$. 
\end{prop}
For $k \geq 3$, we have the following:

\begin{theorem}[Dane Miyata]\label{thm:DaneCycleShellability}
	For all $ n-1 \geq k \geq 3$, 
 the $k$-cut complex of the cycle graph $C_n$  on $n$ vertices,  $ \Delta \coloneqq \Delta_{k}(C_n) $, is shellable. 
\end{theorem}
\begin{proof}
    When $k = n-1$, $\Delta_k(G)$ is void and thus trivially shellable, so assume $k \leq n-2$.

	Fix $ k > 2 $ and $ n $, and let $ s = n-k $.  Order the vertices of $ C_n $ by going clockwise around
	the cycle, starting at some initial vertex and, for simplicity, identify the vertex set of $ C_n $ with
	$ [n] \coloneqq \{1, \dots, n\} $ with the usual order. The facets of $ \Delta $ are exactly the size $ s $ separating sets of $ C_n $.
	We will identify each facet of $ \Delta $ with the unique increasing sequence of its vertices. We claim that
	the lexicographic ordering on the facets of $ \Delta $ gives a shelling order for $ \Delta $.

	Let $ A $ and $ B $ be facets of $ \Delta $ such that $ A < B $ in the lexicographic order. Suppose $A
	$ is given by the sequence $ a_1, \dots , a_{s} $ of vertices of $ C_n $ and $ B $ is given by the sequence
	$ b_1, \dots, b_s $ of vertices of $ C_n $. Since $ A<B $, there is a unique index $ i $, such that $ a_i
	\lneq b_i $, and $ a_j = b_j $ for all $ j<i $. Furthermore, there is a minimum index $ \ell $ such that $b_{\ell} \notin A$.  Let $ C = (B \cup \{a_i\}) \setminus \{b_{\ell}\}$. In other words, $ C $ is given by the sequence 
	\[ b_1, \dots, b_{i-1} , a_i , b_{i} , \dots , \hat{b}_{\ell}, \dots b_{s} .\] 
	(Here the hat means we omit that element from the sequence). Intuitively, $ C $ is obtained from $ B $ by
	taking the smallest element of $ B$ that is not in $ A $ and replacing it with the smallest element of $ A $
	that is not in $ B $. Now there are two cases; either $ C $ gives a separating set or it does not. 

	In the case where $ C $ gives a separating set, $ C $ is a facet of $ \Delta $. Notice $ C<B $ in the lexicographic order, since the $ i $th term in $ C $ is $ a_i $ which is less than $ b_{i} $. Furthermore,
	$ C \cap B = B \setminus \{b_{\ell}\} $ and	$ b_\ell \notin A \cap B $, which means $ A \cap B \subseteq B
	\cap C $ and $ |B \cap C| = s-1 $. Thus, in this case, the shelling order condition is satisfied.

	In the case where $ C $ is not a separating set, it must either be an interval, or the complement of an interval. If $ C $ is an interval, we conclude that $ b_{\ell} = b_{s} $, meaning $ b_{\ell} $ is the largest element of $ B $. Otherwise $ b_{s} > b_{\ell} > a_{i} $, and because $ C $ contains $ a_{i} $ and $ b_{s} $ but not $ b_{\ell} $, it would not be an interval. Since $ b_{\ell} = b_{s} $	is the smallest element of $ B \setminus A$, we conclude $ b_{1}, \dots , b_{s-1} $ are all in $ A $. In our construction of $ C $, we took $ B $ and replaced $ b_{\ell} = b_{s} $ with $ a_{i} $ and so all elements of $ C $ are in $ A $, which means $ C = A $. This is a contradiction because $ A $ is a separating set but $ C $ is not. Therefore, $ C $ is not an interval, so it must be the complement of an interval, namely, 
	\[C = [n] \setminus [p,q] \]
	where $ [p,q] $ is some interval of size $ k $. In this case, we can immediately deduce the following from the construction of $ C $:
	\[ B \cap [p,q] = \{b_{\ell}\} \qquad \text{and} \qquad p > a_{i}.\]
		
	Since $A$ is a facet, it is distinct from $ C $ so there exists some $ a' \in A$ such that  $a' \in [p,q] $. We know that $ a' \notin B $ since it cannot be equal to $ b_{\ell} $ and $ B \cap [p,q] = \{b_{\ell}\} $. Thus, $ B $ is missing at least two elements of $ A $, namely $ a_i $ and $ a' $. We know $ |B|=|A| $ so $ B $ must contain an element $ b' $ distinct from $ b_{\ell} $ such that $ b' \notin A $. We defined $ b_{\ell} $ to be the minimum element of $ B \setminus A $ so $ b' > b_{\ell} $. This implies $ b' > q $ as $ B \cap [p,q] =\{ b_{\ell}\} $. In particular, we have shown that $ b' $ is greater than all elements of $ [n] \setminus B $ since $[n] \setminus B = {a_i} \cup [p,q] \setminus {b_\ell}$. 

	Note that $ B = [1,a_{i}-1] \cup [a_{i}+1,p-1] \cup \{b_\ell\} \cup [q+1,n]$ and so $[n]\setminus B=\{a_i\}\cup [p, b_\ell-1]\cup [b_\ell+1,q]$.
	
	Now recall that $ k \geq 3 $, which means $ [n] \setminus B $ has cardinality at least 3. Let  $ x,y,z $ be three distinct elements of $ [n] \setminus B $ and without loss of generality assume $ x < y < z$. Now we have four cases, $ b_{\ell} < x $, $ x < b_{\ell} < y $, $ y < b_{\ell} < z $, and $ b_{\ell} > z $. 
	
	\begin{enumerate}
		\item[\textbf{Case 1.}] $( b_{\ell} < x )$ In this case, let 
		$D = \big(B \cup \{y\}\big)  \setminus \{b'\}.$\\
		Notice $ D $ must necessarily be a separating set since $  b_{\ell} < x < y < b' $ and $ b_{\ell}, y \in D
		$, but $ x,b' \notin D $. 	
		\vspace{1em}
		\item[\textbf{Case 2.}] $(x< b_{\ell} < y )$ In this case, let 
		$ D = \big(B \cup \{z\}\big)  \setminus \{b'\}.$\\
		Notice $ D $ must necessarily be a separating set since $  b_{\ell} < y < z < b' $ and $ b_{\ell}, z \in D
		$, but $ y,b' \notin D $. 	
		\vspace{1em}
		\item[\textbf{Case 3.}] $( y <b_{\ell} < z )$ In this case, let 
		$ D = \big(B \cup \{x\}\big)  \setminus \{b'\}.$\\
		Notice $ D $ must necessarily be a separating set since $  x < y < b_{\ell} < b' $ and $ b_{\ell},x \in D
		$, but $ y,b' \notin D $. 	
		\vspace{1em}
	\item[\textbf{Case 4.}] $( z < b_{\ell} )$ In this case, let 
		$ D = \big(B \cup \{y\}\big)  \setminus \{b'\}.$\\
		Notice $ D $ must necessarily be a separating set since $ y < z < b_{\ell} < b' $ and $ b_{\ell}, y \in D
		$, but $ z,b' \notin D $. 	
		\vspace{1em}
	\end{enumerate}
	In each case, we see that $ D $ is a facet of $ \Delta $ and since it is constructed by taking $ B $ and
	replacing $ b' $ with a smaller element, we have that $ D < B $ in the lexicographic order. Furthermore, the only element of $ B\setminus D$ is $ b' $
	so $ |B \cap D| = s - 1 $.
	Finally, as $ b' \notin A $, we see that $ B \cap A \subseteq B \cap D $ and so the shelling order condition
	is
	satisfied. 
\end{proof}

The results of Section~\ref{sec:Euler-char} can be applied to compute Betti numbers for cut complexes of trees and cycles. These belong to  families of graphs whose cut complexes satisfy the condition~\eqref{eqn:Condition-truncBoolean} of Theorem~\ref{thm:truncBoolean-minus-antichain}.  Recall from that theorem that for a forest $\F_n$ with $n$ vertices, $|\mathcal{Z}_k(\F_n)|$ is the number of subgraphs of $\F_n$ that are trees.

In this case $|\mathcal{Z}_k(\F_n)|=n-c,$ where $c$ is the number of connected components of the forest $\F_n$.  Also if the tree is a path $P_n$, then $|\mathcal{Z}_k(P_n)|=(n-k+1)$. 

\begin{prop}\label{prop:cut-complex-trees} Let $\mathcal{F}_n$ be a forest on $n$ vertices. 
Then $\Delta_k(\mathcal{F}_n)$ is homotopy equivalent to a wedge of $\left( \binom{n-1}{k-1} -|\mathcal{Z}_k(\F_n)|\right)$ spheres in dimension $n-k-1$ if this number is nonzero, and contractible otherwise. 
 In particular, if $\mathcal{F}_n$ is a tree $T_n$ on $n$ vertices, we have 
\[ \Delta_k(T_n) \simeq \begin{cases} \{\text{a point}\}, & k=2,\\
\bigvee_{\binom{n-1}{k-1}-|\mathcal{Z}_k(T_n)|}\ \bbS^{n-k-1}, &k\ge 3.
\end{cases}   \]
\end{prop}
\begin{proof} It is immediate from Proposition~\ref{prop:class-of-graphs-2022-3-26} that condition~\eqref{eqn:Condition-truncBoolean} in Part (3) of Theorem~\ref{thm:truncBoolean-minus-antichain} is satisfied, since forests are acyclic. 
Since forests are shellable by Corollary~\ref{cor:Mark-trees} and Theorem~\ref{thm:MargeDisjunion}, the result follows.  \end{proof}

Observe that when $k=2$, $|\mathcal{Z}_k(\F_n)|$ is just the number of edges in the forest, namely, $n- c$ if the forest has $c$ connected components, and the Betti number is thus $c-1$. 
In the special case when the forest is a single tree, this  makes the reduced Euler characteristic zero, and so the $2$-cut complex of a tree is contractible, recovering the result of Corollary~\ref{cor:Delta2trees-contractible}.

For the cycle graph $C_n,$  we may assume $n\ge 4$ and $2\le k\le n-2.$    If $n=4,$ the one-dimensional cut complex  $\Delta_2(C_4)$ has two facets $\{1,3\}$ and $\{2,4\},$ and is homotopy equivalent to the 0-sphere $\bbS^0;$  it is thus not shellable. 

The cut complex $\Delta_k(C_n)$ has dimension $n-k-1.$
When $k=2$, we already know that  the cut complex $\Delta_2(C_n)$  has the homotopy type of one sphere in codimension 1, i.e., in dimension $n-4$. 
\begin{prop}\label{prop:BettiNumberCycles} Let $C_n$ be the cycle graph on $n\ge 5$ vertices.  
 For $n-2\ge k\ge 3,$ the shellable cut complex $\Delta_k(C_n)$ is homotopy equivalent to a wedge of $\binom{n-1}{k-1}-n$ spheres in  (top) dimension $n-k-1.$  
\end{prop}
\begin{proof}    Proposition~\ref{prop:class-of-graphs-2022-3-26} shows that condition~\eqref{eqn:Condition-truncBoolean} in Part (3) of Theorem~\ref{thm:truncBoolean-minus-antichain} is satisfied, since $C_n$ has no cycles of length less than $n$. The trees on $k$ vertices may be described as the intervals $[i, i+k-1], 1\le i\le n,$ but now with arithmetic modulo $n$, so that there are $n$ such trees in all, i.e., $|\mathcal{Z}_k(C_n)|=n$.

Since $\Delta_k(C_n)$ is shellable for $k\ge 3$ by Theorem~\ref{thm:DaneCycleShellability}, the statement follows.   For the cut complex $\Delta_2(C_n),$ 
we already know that 
it has the homotopy type of one sphere in dimension $n-4$, by Proposition~\ref{prop:MarijaDMT-Delta2cycle}; this is confirmed by the value $(-1)$ of the reduced Euler characteristic $\mu(\Delta_2(C_n))$.
\end{proof}
\begin{cor}\label{cor:cycleCutComplex-subposet}  The subposet $P(n,k)\backslash\mathcal{Z}_k(C_n),$ where $\mathcal{Z}_k(C_n)$ is the antichain of $n$ complements of trees on $k$ vertices $[i,i+k-1]$ (modulo $n$), $1\le i\le n$, 
as above, is homotopy equivalent to a wedge of $\binom{n-1}{k-1}-n$ spheres in the top dimension $n-k-1$  if $k\ge 3,$ and to a single sphere in dimension one less than the top if $k=2.$
\end{cor}

The Euler characteristic computation of Proposition~\ref{prop:BettiNumberCycles} can be made equivariant. The cycle graph $C_n$ and its cut complexes are invariant under the action of the cyclic subgroup $\mathfrak{C}_n$ of  the symmetric group $\mathfrak{S}_n$, generated by the $n$-cycle $(1,2,\dots,n)$.  We have the following result. 

\begin{theorem}\label{thm:Cyclicgpaction-k-cut-complex-n-cycle} The cyclic group $\mathfrak{C}_n$ of order $n$ acts on the unique nonvanishing homology of the cut complex $\Delta_k(C_n),\, 2\le k\le n-2,$ as follows:

If $k=2$, the action 
of $\mathfrak{C}_n$ is the 1-dimensional module given by 
\[\begin{cases}  \text{the trivial representation}, & \text{$n$ odd},\\
                \text{the sign representation}, & \text{$n$ even}.
\end{cases}\]

If $k\ge 3$,  the action 
of $\mathfrak{C}_n$ is given by 
$V_{(k,1^{n-k})}\big\downarrow_{\mathfrak{C}_n}^{\mathfrak{S}_n} - \mathrm{Reg}_{\mathfrak{C}_n},$
where $V_{(k,1^{n-k})}$ is the $\mathfrak{S}_n$-irreducible indexed by the partition $(k, 1^{n-k})$, the down arrow indicates restriction to the subgroup $\mathfrak{C}_n$, and $\mathrm{Reg}_{\mathfrak{C}_n}$ 
is the regular representation of $\mathfrak{C}_n$.
\end{theorem}
\begin{proof} Recall from Theorem\ref{thm:SolomonTruncatedBoolean} that  the homology module of the subposet $P(n,k)$ affords the representation $V_{(k,1^{n-k})}$ of $\mathfrak{S}_n$. In \cite{SuAIM1994}, tools were developed to compute the group action on the Lefschetz  module of the order complex of a poset, in particular when an antichain is deleted. Applying  \cite[Theorem~1.10]{SuAIM1994} to Corollary~\ref{cor:cycleCutComplex-subposet} and the antichain $\mathcal{Z}_k(C_n)$ consisting of complements of trees of size $k$, we obtain the following  $ \mathfrak{C}_n$-equivariant versions of Equation~\eqref{eqn:mu-deleted-antichain}: 
\begin{equation}\label{CycleRepk=2}
\tilde{H}_{n-4}(\Delta_2(C_n))=\bigoplus_{\stackrel{\hat 0<x<\hat 1} {x\in \mathcal{Z}_k(C_n)}} \tilde{H}(\hat 0, x)\otimes  \tilde{H}( x,\hat 1) - V_{(2,1^{n-2})}\big\downarrow_{\mathfrak{C}_n}^{\mathfrak{S}_n},
\end{equation}
and for $k\ge 3$, 
\begin{equation}\label{CycleRepk>=3}
\tilde{H}_{n-k-1}(\Delta_k(C_n))= V_{(k,1^{n-k})}\big\downarrow_{\mathfrak{C}_n}^{\mathfrak{S}_n}-\bigoplus_{\stackrel{\hat 0<x<\hat 1} {x\in \mathcal{Z}_k(C_n)}} \tilde{H}(\hat 0, x)\otimes  \tilde{H}( x,\hat 1) , \ k\ge 3.
\end{equation}
We have  omitted the homology degrees in the right-hand side since the intervals involved are Boolean lattices,
 and it is clear where the unique nonvanishing homology occurs.

To see that these decompositions are indeed group-equivariant, we observe that the $n$ elements of the antichain $\mathcal{Z}_k(C_n)$ are transitively permuted by the cyclic group $\mathfrak{C}_n$. This is also true of the one-dimensional homology modules $\tilde{H}(\hat 0, x), x\in \mathcal{Z}_k(C_n).$ Since the stabilizer of an $x\in \mathcal{Z}_k(C_n)$ is clearly the trivial group, the cyclic group transitively permutes the summands $\tilde{H}(\hat 0, x)\otimes  \tilde{H}( x,\hat 1)$, and hence acts like the regular representation on the direct sum $\oplus_{\stackrel{\hat 0<x<\hat 1} {x\in \mathcal{Z}_k(C_n)}} \tilde{H}(\hat 0, x)\otimes  \tilde{H}( x,\hat 1)$.  Note also that the homology of $(x, \hat 1)$ is the one-dimensional trivial module for all $x\in \mathcal{Z}_k(C_n).$
Hence Equation~\eqref{CycleRepk>=3} becomes 
\[\tilde{H}_{n-k-1}(\Delta_2(C_n))=V_{(k,1^{n-k})}\big\downarrow_{\mathfrak{C}_n}^{\mathfrak{S}_n} -\mathrm{Reg}_{\mathfrak{C}_n},
\]
as claimed.

In the case $k=2$,  Equation~\eqref{CycleRepk=2} becomes 
$\tilde{H}_{n-4}(\Delta_2(C_n))=\mathrm{Reg}_{\mathfrak{C}_n} - V_{(2,1^{n-2})}\big\downarrow_{\mathfrak{C}_n}^{\mathfrak{S}_n}.
$

\vspace{1em}
We can determine the restriction of the $\mathfrak{S}_n$-irreducible $V_{(2,1^{n-2})}$ to $\mathfrak{C}_n$ precisely as follows.

The permutation action of the cyclic group  $\mathfrak{C}_n$ on  the set $[n]$ is the restriction of the natural action of $\mathfrak{S}_n$ on 
$[n]$. It is well known that the latter decomposes into two irreducible components, the trivial representation and the \emph{reflection} representation, see e.g.,\ 
\cite[Example~7.18.8]{RPSEC21999}: 
$V_{(n)}\oplus V_{(n-1,1)}.$

Since $V_{(n)}$ is the one-dimensional trivial representation of $\mathfrak{S}_n$, we obtain 
\[V_{(n-1,1)}\big\downarrow_{\mathfrak{C}_n}^{\mathfrak{S}_n}=\mathrm{Reg}_{\mathfrak{C}_n}- 1_{\mathfrak{C}_n}.\]
Now observe that $V_{(2,1^{n-2})}$ is the sign representation of $\mathfrak{S}_n$ tensored with $V_{(n-1,1)}$. It is easy to see that the sign representation restricted to $\mathfrak{C}_n$ is the trivial representation $1_{\mathfrak{C}_n}$ if $n$ is odd, and the sign representation otherwise.  The regular representation is invariant with respect to tensoring with the sign, and the result follows.
\end{proof}

\subsection{Prism over a clique}\label{sec:RowanNonshellableGraph}
\begin{df}
The \emph{prism over a clique} is the graph $G_n$ with vertex set $\{1^+, \dotsc, n^+, 1^-, \dotsc, n^-\}$ for a given integer $n$, and an edge between $i^+$ and $j^+$, between $i^-$ and $j^-$, and between $i^+$ and $i^-$, for every $i, j \in \{1, \dotsc, n\}$. Note that this is the Cartesian product, $G_n=K_n\times K_2$.
\end{df}
\begin{rem}\label{rem:G_k-for-other-k}
If $n<k$, then $\Delta_k(G_n)$ is the void complex, since there are no separating sets of size $2n-k$, which is less than $n$.  So $\Delta_k(G_n)$ is shellable.   
\end{rem}

We determine the homotopy type of $\Delta_k(G_n)$ when $n\ge k$ precisely in Theorem~\ref{thm:conj-prismClique2022June21S-MarijaDMT} below. The case $k=2$ was also proved in \cite{BDJRSX-TOTAL2024}.

\begin{theorem}\label{thm:conj-prismClique2022June21S-MarijaDMT}
Let $n\ge k\ge 2$. The $(2n-k-1)$-dimensional cut complex $\Delta_k(G_n)$ has homotopy type
\[ \Delta_k(G_n)\simeq\bigvee_{\binom{n-1}{k-1}} \bbS^{2n-k-2},\] one lower than full dimension. 
  Thus for $n\ge k$, $\Delta_k(G_n)$ is not shellable.  
\end{theorem}

\begin{proof}
We prove this theorem by using discrete Morse theory, precisely by constructing a sequence of element matchings \cite[Appendix, Theorems~7.2, 7.5]{BDJRSX-TOTAL2024}. Let $V= \{1^+,\ldots, n^+, 1^-, \ldots, n^-\}$ be the vertex set, $V^+=\{1^+,\ldots, n^+\}$  and $V^-= \{1^-, \ldots, n^-\}.$ For an arbitrary $X \subset V$ let $X^+ = X \cap V^+$, and $X^- = X \cap V^-$.
We observe that if $X \subset V^+$, $X \subset V^-$, or if there is an $i$ such that $\{i^+, i^-\} \subset X$, then $X$ is connected.  
Pairs $\{i^+, i^-\} $ will be called {\em columns}.  Thus the complement of a disconnected set (in particular, a facet of $\Delta_k(G_n)$) must contain at least one element of each column. A set of vertices  $X$ contains a disconnected set of size $m$ if and only if there are $m$ columns indexed by $i_1 < i_2 < \cdots < i_m$ such that for each $j \in \{1, \ldots, m\}$, $X \cap \{i_j^+,i_j^- \} \neq \emptyset,$  $X \cap \{i_1^+, \ldots, i_m^+\} \neq \emptyset$, and $X \cap \{i_1^-, \ldots, i_m^-\} \neq \emptyset$.  We will denote by $d(X)$ the maximal cardinality of a disconnected set contained in $X$.

First, we perform an element matching $\mathcal{M}_{1^+}$ using vertex $1^+$. Faces $\sigma \in \Delta_k(G_n)$ that remain unmatched satisfy  $1^+ \notin \sigma$ and $\sigma \cup \{1^+\} \notin \Delta_k(G_n).$ Then we perform an element matching $\mathcal{M}_{1^-}$ using vertex $1^-$. 
There are two possible types of unmatched faces after the sequence $\mathcal{M}_{1^+}$ followed by $\mathcal{M}_{1^-}$:
\begin{enumerate}
\item Faces $\sigma \in \Delta_k(G_n)$ that satisfy:
$$1^+, 1^- \notin \sigma, \ \sigma \cup \{1^+\}\notin \Delta_k(G_n), \text{ and }  \sigma \cup \{1^-\} \notin \Delta_k(G_n).$$

\item Faces $\sigma \cup \{1^-\} \in \Delta_k(G_n)$ that satisfy: 
$$1^+, 1^- \notin \sigma, \ \sigma \cup \{1^+\} \in \Delta_k(G_n), \text{ and }  \sigma \cup \{1^+, 1^-\} \notin \Delta_k(G_n).$$
\end{enumerate}

To see that there are actually no faces of type (1), note that every face $\sigma$ is in a facet, and every facet has to contain either $1^+$ or $1^-$. So either $\sigma\cup\{1^+\}$ or $\sigma\cup\{1^-\}$ is in $\Delta_k(G_n)$. Therefore, there are only unmatched faces of type (2).
Denote by $K$ the set of all unmatched faces after the sequence of matchings $\mathcal{M}_{1^+}$ followed by  $\mathcal{M}_{1^-}$. 

 Let $\sigma \cup \{1^-\}$ be an arbitrary face in $K$. Conditions from (2) imply that $d(V\setminus (\sigma \cup \{1^+\})) \ge k$, $d(V\setminus (\sigma \cup \{1^-\})) \ge k$, while $d(V\setminus (\sigma \cup \{1^+, 1^-\})) \le k-1$.  Consequently,
$$d(V\setminus (\sigma \cup \{1^+\})) = k= d(V\setminus (\sigma \cup \{1^-\})), \text{ and }  d(V\setminus (\sigma \cup \{1^+, 1^-\})) = k-1.$$
Because of its repeated use, we write $$X_{\sigma}= V\setminus (\sigma \cup \{1^+, 1^-\}).$$ We conclude that the unmatched faces are all those $\sigma \cup \{1^-\}$  which satisfy the condition that  there are $k-1$ indices $i_1 < i_2 < \cdots < i_{k-1}$ ($i_1 \ge 2$) such that:
\begin{enumerate}[label=(\alph*)]
    \item $X_{\sigma} \subset \{i_1^+, i_1^-\} \cup \cdots \cup \{i_{k-1}^+, i_{k-1}^-\}$;
    \item For each $j \in \{1,\ldots, k-1\}$, $X_{\sigma} \cap \{i_{j}^+, i_{j}^-\} \neq \emptyset$;
    \item $X_{\sigma} \cap \{i_{1}^+, \ldots, i_{k-1}^+\} \neq \emptyset$;
    \item $X_{\sigma} \cap \{i_{1}^-, \ldots, i_{k-1}^-\} \neq \emptyset$.
\end{enumerate}

The set $\{i_1, i_2, \ldots, i_{k-1}\}$ of indices will be 
called the {\em support} of $X_{\sigma}$, and the individual $i_j$ the {\em supporting indices}.

Further,  we perform a sequence of element matchings $\mathcal{M}_{2^+}, \ldots, \mathcal{M}_{n^+}$, where each $\mathcal{M}_{i^+}$ denotes the element matching  using vertex $i^+$. We claim that the set of unmatched faces after all these element matchings will be 
$$\mathcal{C} = \{ V \setminus \{1^+,i_1^+, i_1^-, i_2^-, \ldots, i_{k-1}^-\} : \{i_1,\ldots,i_{k-1}\} \subset \{2,3,\ldots, n\} \text{ and  } i_1< i_2 < \ldots < i_{k-1}\},$$
i.e., the set of all $\sigma \cup \{1^-\}$ satisfying $X_{\sigma} = \{i_1^+, i_1^-, i_2^-, \ldots, i_{k-1}^-\}$ for an arbitrary subset $\{i_1,\ldots,i_{k-1}\} \subset \{2,3,\ldots, n\}$ and order $i_1< i_2 < \ldots < i_{k-1}$. See Figure~\ref{matched} for an example of such a matched face, and Figure~\ref{unmatched} for an example of an unmatched face.

Let us explain which pairs are made by these element matchings. For each face $\sigma \cup \{1^-\} \in K$, consider the  set of supporting indices  for $X_{\sigma}$: $i_1 < i_2 < \cdots < i_{k-1}$ ($i_1 \ge 2$). Consider the smallest $j\in \{1, \ldots, k-1\}$ such that $i_j^- \in X_{\sigma}$ (it exists because of condition (d)). We have two possibilities:
\begin{itemize}
    \item[(P1)] If $i_j^+ \in \sigma$, we claim that $\sigma \cup \{1^-\}$ will be matched with $(\sigma \setminus \{i_j^+\}) \cup \{1^-\}$ in the element matching $\mathcal{M}_{i_j^+}$. One can easily check that if $\sigma \cup \{1^-\} \in K$, then $(\sigma \setminus \{i_j^+\}) \cup \{1^-\} \in K$ as well. Also, $j$ is the smallest index for $(\sigma \setminus \{i_j^+\}) \cup \{1^-\} \in K$ such that $i_j^- \in X_{\sigma \setminus \{i_j^+\}}$. 
    \item[(P2)] If $i_j^+ \notin \sigma$, and if $(\sigma \cup \{i_j^+\}) \cup \{1^-\} \in K$,  we claim that $\sigma \cup \{1^-\}$ will be matched with $(\sigma \cup \{i_j^+\}) \cup \{1^-\}$ in the element matching $\mathcal{M}_{i_j^+}$. Again, the smallest negative entry $j$ is the same for $\sigma$ and $\sigma \cup \{i_j^+\}.$
\end{itemize}
    Consider an arbitrary  $\sigma$ for which $(\sigma \cup \{i_j^+\}) \cup \{1^-\} \notin K.$ Then  $X_{\sigma}$ satisfies conditions (a)--(d), while $X_{\sigma \cup \{i_j^+\}}$ does not satisfy all of them. The only condition which can be violated is condition (c), i.e., $X_{\sigma \cup \{i_j^+\}} \cap \{i_{1}^+, \cdots, i_{k-1}^+\} = \emptyset$. From this relation, and from (a)--(d) for $X_{\sigma}$, we conclude  that $X_{\sigma} = \{i_j^+\} \cup  \{i_{1}^-,  i_2^-, \cdots, i_{k-1}^-\}$. Further, since $j$ is the smallest negative entry in $X_{\sigma}$, then $j=1$, and $X_{\sigma} = \{i_1^+\} \cup  \{i_{1}^-,  i_2^-, \cdots, i_{k-1}^-\}.$ This means that $\sigma \cup \{1^-\}$ belongs precisely  to $\mathcal{C}.$ Therefore we have proved that all faces in $K \setminus \mathcal{C}$ are divided into ``pairs" of type 
    $$\sigma \cup \{1^-\} \longleftrightarrow (\sigma \cup \{i_j^+\}) \cup \{1^-\}$$
    by considering that smallest negative entry  $j$.

\begin{figure}
\begin{center}

\begin{tikzpicture}
[
        box/.style={rectangle,draw=black,thick, minimum size=1cm},
    ]

\foreach \x in {0,1,...,8}{
    \foreach \y in {0,1}
        \node[box] at (\x,\y){};
}

\draw (-0.5,-0.5) to  (0.5,0.5);
\draw (-0.5,0.5) to  (0.5,1.5);

\draw[draw=red]  (1.5,0.5) to  (2.5,1.5) (2.5,0.5) to  (1.5,1.5);
\draw[draw=red]  (3.5,-0.5) to  (4.5,0.5) (3.5,0.5) to  (4.5,-0.5);
\draw[draw=red]  (4.5,-0.5) to  (5.5,0.5) (4.5,0.5) to  (5.5,-0.5);
\draw[draw=red]  (4.5,0.5) to  (5.5,1.5) (4.5,1.5) to  (5.5,0.5);
\draw[draw=red]  (6.5,-0.5) to  (7.5,0.5) (6.5,0.5) to  (7.5,-0.5);

\node[box,fill=cyan ] at (1,1){};  
\node[box,fill=cyan ] at (1,0){};   
\node[box,fill=cyan ] at (2,0){}; 
\node[box,fill=cyan ] at (3,0){}; 
\node[box,fill=cyan ] at (3,1){}; 
\node[box,fill=cyan ] at (4,1){}; 
\node[box,fill=cyan ] at (6,1){}; 
\node[box,fill=cyan ] at (6,0){}; 
\node[box,fill=cyan ] at (7,1){}; 
\node[box,fill=cyan ] at (8,1){};
\node[box,fill=cyan ] at (8,0){}; 

\draw[draw=black, line width=3pt, mark=none]  (3.5,-0.5) rectangle (4.5,1.5); 

\draw[dotted]
    (-0.2,0.2) node {$1^-$}
    (-0.2,1.2) node {$1^+$}
    (0.8,0.2) node {$2^-$}
    (0.8,1.2) node {$2^+$}
    (1.8,0.2) node {$3^-$}
    (1.8,1.2) node {$3^+$}
    (2.8,0.2) node {$4^-$}
    (2.8,1.2) node {$4^+$}
    (3.8,0.2) node {\color{orange}{$\boldsymbol{5^-}$}}
    (3.8,1.2) node {$5^+$}
    (4.8,0.2) node {$6^-$}
    (4.8,1.2) node {$6^+$}
    (5.8,0.2) node {$7^-$}
    (5.8,1.2) node {$7^+$}
    (6.8,0.2) node {$8^-$}
    (6.8,1.2) node {$8^+$}
    (7.8,0.2) node {$9^-$}
    (7.8,1.2) node {$9^+$};
\draw[color = black]
    (4,2) node {$\sigma = \{2^+, 2^-,3^-, 4^+, 4^-,5^+, 7^+, 7^-,8^+,9^+,9^-\}\not\in \mathcal{C}.$};

\draw [color = black] (-1.8,1.3) node {$\qquad\qquad\qquad\qquad$};

\draw[color=black, fill = cyan] (-2.7,0.5) rectangle ++(0.5,0.5);
\draw[color = black] (-1.8,0.8) node { $\in \sigma$};

\draw[color = black] (-2,0) node { $\boxed{\color{red}{\times}} \in X_{\sigma}$ };   
\draw[color = black] (-2,-0.5) node { $1^+, 1^- \notin \sigma$ };

\draw[color = red] 
(2,-1) node {$i_1 = 3$}
(4,-1) node {$i_2 = 5$}
(5,-1) node {$i_3 = 6$}
(7,-1) node {$i_4 = 8$};

\draw [decorate,decoration={brace,amplitude=5pt,mirror,raise=4ex}]
  (2,-0.6) -- (7,-0.6) node[midway,yshift=-3em]{ $k-1 =4$ supporting indices $= \#$ of columns with $\boxed{\color{red}{\times}}$  boxes};

\draw (4,-2.5) node{ $X_{\sigma} = \{3^+,{\color{orange}{\boldsymbol{5^-}}},6^+,6^-,8^- \}$ };

\draw (4,-3) node{ $j=2$, $i_2^- = \color{orange}{\boldsymbol{5^-}}$ -- smallest negative in $X_{\sigma}$};

\end{tikzpicture}

\begin{tikzpicture}
[
        box/.style={rectangle,draw=black,thick, minimum size=1cm},
    ]
\foreach \x in {0,1,...,8}{
    \foreach \y in {0,1}
        \node[box] at (\x,\y){};
}

\draw (-0.5,-0.5) to  (0.5,0.5);
\draw (-0.5,0.5) to  (0.5,1.5);

\draw[draw=red]  (1.5,0.5) to  (2.5,1.5) (2.5,0.5) to  (1.5,1.5);
\draw[draw=red]  (3.5,-0.5) to  (4.5,0.5) (3.5,0.5) to  (4.5,-0.5);
\draw[draw=red]  (3.5,0.5) to  (4.5,1.5) (3.5,1.5) to  (4.5,0.5);
\draw[draw=red]  (4.5,-0.5) to  (5.5,0.5) (4.5,0.5) to  (5.5,-0.5);
\draw[draw=red]  (4.5,0.5) to  (5.5,1.5) (4.5,1.5) to  (5.5,0.5);
\draw[draw=red]  (6.5,-0.5) to  (7.5,0.5) (6.5,0.5) to  (7.5,-0.5);

\node[box,fill=cyan ] at (1,1){};  
\node[box,fill=cyan ] at (1,0){};   
\node[box,fill=cyan ] at (2,0){}; 
\node[box,fill=cyan ] at (3,0){}; 
\node[box,fill=cyan ] at (3,1){}; 
\node[box,fill=cyan ] at (6,1){}; 
\node[box,fill=cyan ] at (6,0){}; 
\node[box,fill=cyan ] at (7,1){}; 
\node[box,fill=cyan ] at (8,1){};
\node[box,fill=cyan ] at (8,0){};

\draw[draw=black, line width=3pt, mark=none]  (3.5,-0.5) rectangle (4.5,1.5); 

\draw[dotted]
    (-0.2,0.2) node {$1^-$}
    (-0.2,1.2) node {$1^+$}
    (0.8,0.2) node {$2^-$}
    (0.8,1.2) node {$2^+$}
    (1.8,0.2) node {$3^-$}
    (1.8,1.2) node {$3^+$}
    (2.8,0.2) node {$4^-$}
    (2.8,1.2) node {$4^+$}
    (3.8,0.2) node {\color{orange}{$\boldsymbol{5^-}$}}
    (3.8,1.2) node {\color{orange}{$\boldsymbol{5^+}$}}
    (4.8,0.2) node {$6^-$}
    (4.8,1.2) node {$6^+$}
    (5.8,0.2) node {$7^-$}
    (5.8,1.2) node {$7^+$}
    (6.8,0.2) node {$8^-$}
    (6.8,1.2) node {$8^+$}
    (7.8,0.2) node {$9^-$}
    (7.8,1.2) node {$9^+$};

\draw[color = black]
    (4,2) node {$5^+ \in \sigma$, hence $\sigma \cup \{ 1^-\}$ will be matched with $\big(\sigma\backslash\{ 5^+\} \big) \cup \{ 1^-\}$ in $\mathcal{M}_{5^+}$};
    
\draw[color = black] (-2,0.5) node { $\quad\qquad \sigma \backslash \{ 5^+\}:\quad$ };
\end{tikzpicture}

\end{center}
\caption{An example of a matched pair $(\sigma \cup \{1^-\}, \ \sigma \setminus \{5^+\}\cup \{1^-\} )$}
\label{matched}
\end{figure}


\begin{figure}
\begin{center}

\begin{tikzpicture}
[
        box/.style={rectangle,draw=black,thick, minimum size=1cm},
    ]
    
\foreach \x in {0,1,...,8}{
    \foreach \y in {0,1}
        \node[box] at (\x,\y){};
}

\draw (-0.5,-0.5) to  (0.5,0.5); 
\draw (-0.5,0.5) to  (0.5,1.5);

\draw[draw=red]  (1.5,0.5) to  (2.5,1.5) (2.5,0.5) to  (1.5,1.5);
\draw[draw=red]  (1.5,-0.5) to  (2.5,0.5) (1.5,0.5) to  (2.5,-0.5);
\draw[draw=red]  (3.5,-0.5) to  (4.5,0.5) (3.5,0.5) to  (4.5,-0.5);
\draw[draw=red]  (4.5,-0.5) to  (5.5,0.5) (4.5,0.5) to  (5.5,-0.5);
\draw[draw=red]  (6.5,-0.5) to  (7.5,0.5) (6.5,0.5) to  (7.5,-0.5);

\node[box,fill=cyan ] at (1,1){};
\node[box,fill=cyan ] at (1,0){};
\node[box,fill=cyan ] at (3,1){}; 
\node[box,fill=cyan ] at (3,0){}; 
\node[box,fill=cyan ] at (4,1){};
\node[box,fill=cyan ] at (5,1){};
\node[box,fill=cyan ] at (6,1){}; 
\node[box,fill=cyan ] at (6,0){}; 
\node[box,fill=cyan ] at (7,1){};
\node[box,fill=cyan ] at (8,1){};
\node[box,fill=cyan ] at (8,0){};

\draw[draw=black, line width=3pt, mark=none]  (3.5,0.5) rectangle  (4.5,1.5);

\draw[dotted]
    (-0.2,0.2) node {$1^-$}
    (-0.2,1.2) node {$1^+$}
    (0.8,0.2) node {$2^-$}
    (0.8,1.2) node {$2^+$}
    (1.8,0.2) node {$3^-$}
    (1.8,1.2) node {$3^+$}
    (2.8,0.2) node {$4^-$}
    (2.8,1.2) node {$4^+$}
    (3.8,0.2) node {$5^-$}
    (3.8,1.2) node {$5^+$}
    (4.8,0.2) node {$6^-$}
    (4.8,1.2) node {$6^+$}
    (5.8,0.2) node {$7^-$}
    (5.8,1.2) node {$7^+$}
    (6.8,0.2) node {$8^-$}
    (6.8,1.2) node {$8^+$}
    (7.8,0.2) node {$9^-$}
    (7.8,1.2) node {$9^+$};
    
\draw[color = black]
    (4,2) node {$\sigma = \{2^+, 2^-, 4^+, 4^-,5^+, 6^+, 7^+, 7^-,8^+,9^+,9^-\}\in \mathcal{C}.$};
    
\draw[color = black] (-2.2,1.5) node {$n=9$, $k=5$};
\draw[color=black, fill = cyan] (-2.7,0.5) rectangle ++(0.5,0.5);
\draw[color = black] (-1.8,0.8) node { $\in \sigma$};

\draw[color = black] (-2,0) node { $\boxed{\color{red}{\times}} \in X_{\sigma}$};   
\draw[color = black] (-2,-0.5) node { $1^+, 1^- \notin \sigma$ };



\end{tikzpicture}

\begin{tikzpicture}
\draw (0,-1) node{ $\sigma \cup \{1^-\}$ cannot be matched in $\mathcal{M}_{3^+}$, because $\sigma \cup \{ 1^+, 1^- \} \cup \{3^+\}$ is not a face by condition (c). };

\draw (0,-1.5) node{$\sigma\cup \{ 1^-\}$ is not matched in $\mathcal{M}_{5^+}$ with the following face $\tau$, because $\tau$ is matched with $\tau\cup \{ 3^+\}$. };
\end{tikzpicture}
\hspace{1cm}
\begin{tikzpicture}
[
        box/.style={rectangle,draw=black,thick, minimum size=1cm},
    ]

\foreach \x in {0,1,...,8}{
    \foreach \y in {0,1}
        \node[box] at (\x,\y){};
}

\draw (-0.5,-0.5) to  (0.5,0.5);
\draw (-0.5,0.5) to  (0.5,1.5);

\draw[draw=red]  (1.5,0.5) to  (2.5,1.5) (2.5,0.5) to  (1.5,1.5);
\draw[draw=red]  (1.5,-0.5) to  (2.5,0.5) (1.5,0.5) to  (2.5,-0.5);
\draw[draw=red]  (3.5,-0.5) to  (4.5,0.5) (3.5,0.5) to  (4.5,-0.5);
\draw[draw=red]  (3.5,0.5) to  (4.5,1.5) (3.5,1.5) to  (4.5,0.5);
\draw[draw=red]  (4.5,-0.5) to  (5.5,0.5) (4.5,0.5) to  (5.5,-0.5);
\draw[draw=red]  (6.5,-0.5) to  (7.5,0.5) (6.5,0.5) to  (7.5,-0.5);

\node[box,fill=cyan ] at (0,0){};  
\node[box,fill=cyan ] at (1,1){};  
\node[box,fill=cyan ] at (1,0){};   
\node[box,fill=cyan ] at (3,0){}; 
\node[box,fill=cyan ] at (3,1){}; 
\node[box,fill=cyan ] at (5,1){};
\node[box,fill=cyan ] at (6,1){}; 
\node[box,fill=cyan ] at (6,0){}; 
\node[box,fill=cyan ] at (7,1){}; 
\node[box,fill=cyan ] at (8,1){};
\node[box,fill=cyan ] at (8,0){}; 

\draw[draw=black, line width=3pt, mark=none]  (1.5,0.5) rectangle (2.5,1.5);

\draw[dotted]
    (-0.2,0.2) node {$1^-$}
    (-0.2,1.2) node {$1^+$}
    (0.8,0.2) node {$2^-$}
    (0.8,1.2) node {$2^+$}
    (1.8,0.2) node {$3^-$}
    (1.8,1.2) node {$3^+$}
    (2.8,0.2) node {$4^-$}
    (2.8,1.2) node {$4^+$}
    (3.8,0.2) node {$5^-$}
    (3.8,1.2) node {$5^+$}
    (4.8,0.2) node {$6^-$}
    (4.8,1.2) node {$6^+$}
    (5.8,0.2) node {$7^-$}
    (5.8,1.2) node {$7^+$}
    (6.8,0.2) node {$8^-$}
    (6.8,1.2) node {$8^+$}
    (7.8,0.2) node {$9^-$}
    (7.8,1.2) node {$9^+$};
\draw[color = black]
    (4,2) node {$\tau = \{1^-, 2^+, 2^-, 4^+, 4^-, 6^+, 7^+, 7^-,8^+,9^+,9^-\}.$};
    
\draw [color = black] (-1.8,1.3) node {$\qquad\qquad\qquad\qquad$};
\draw[color=black, fill = cyan] (-2.7,0.5) rectangle ++(0.5,0.5);
\draw[color = black] (-1.8,0.8) node {$\in \tau$};

\end{tikzpicture}
\end{center}
\caption{An unmatched face $\sigma \cup \{1^- \}$ and an example showing why it is unmatched.}
\label{unmatched}
\end{figure}
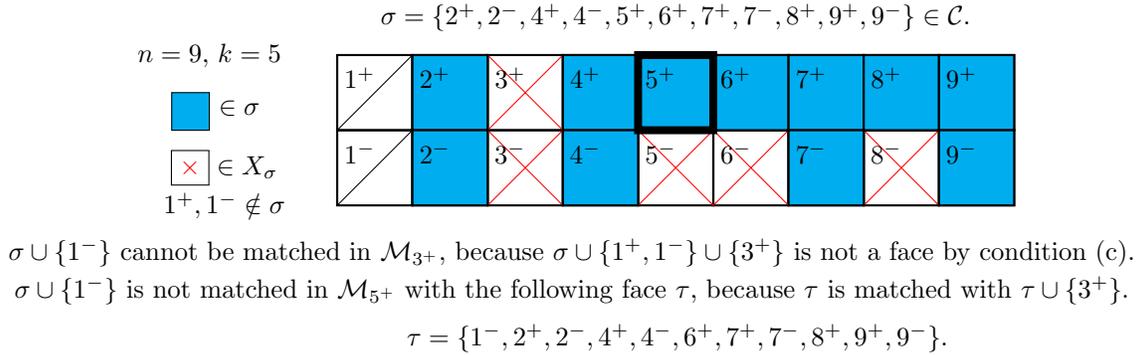


It remains to prove that for each $i \ge 2$, the element matching $\mathcal{M}_{i^+}$ makes exactly the pairs in $K\setminus \mathcal{C}$ (based on the minimal negative entry), as stated in (P1) and (P2), while it does not match any other face, and it does not match any of the faces from $\mathcal{C}$.  We prove this  by induction on $i$.  
  
First we prove the base case $i=2$ (one can observe that the base case is just a simplified version of the induction step, but we include it for completeness). If $\sigma \cup \{1^-\}$ is an arbitrary face in $K\setminus \mathcal{C}$ such that $2^-\in X_{\sigma}$, then $\sigma \cup \{1^-\}$ is paired in $\mathcal{M}_{2^+}$ as explained in (P1) and (P2). Next, let $\sigma \cup \{1^-\} \in \mathcal{C},$ and let $X_{\sigma} = \{i_1^+\} \cup  \{i_{1}^-,  i_2^-, \cdots, i_{k-1}^-\}$. If $2^+ \notin \sigma$ (i.e.,$i_1=2$), then $2^+$ cannot be added to $\sigma \cup \{1^-\}$ because $X_{\sigma \cup \{2^+\}}$ would not satisfy condition (c). If $2^+ \in \sigma$ (i.e.,$i_1 \ge 3$), then $\sigma \cup \{1^-\}$ cannot be paired by using $2^+$ because $d(X_{\sigma \setminus \{2^+\}}) =k,$ contradicting the condition of exactly $k-1$ supporting indices. So we have proved that $\mathcal{M}_{2^+}$ does not match any of the faces in $\mathcal{C}$. It remains to be proved that $\mathcal{M}_{2^+}$ does not make any other pairs. Consider an arbitrary $\sigma \cup \{1^-\} \in K \setminus \mathcal{C}$, with supporting indices $i_{1}<  i_2< \cdots <i_{k-1}$ for $X_{\sigma}$, and let $j$ be the smallest index such that $i_j^- \in X_{\sigma}.$ Now $i_j > 2$, so $2^- \in \sigma$ ($2^- \notin X_{\sigma}$). We claim that the corresponding addition/removal of $2^+$ would change the cardinality of the maximal disconnected set in the  complement (which has to be $k-1$). Indeed, if $2^+ \notin \sigma$, then $(\sigma \cup \{2^+\}) \cup \{1^-\} \notin K$ because $d(X_{\sigma \cup \{2^+\}}) < k-1$. Otherwise, if  $2^+ \in \sigma$, $(\sigma \setminus\{2^+\}) \cup \{1^-\} \notin K$ because $d(X_{\sigma \setminus \{2^+\}}) =k.$ This confirms that $\mathcal{M}_{2^+}$ does not make any other pairs, and finishes the proof for the base case.
  
Now assume that the statement holds for all $2, \ldots, i-1$ $(i \ge 3)$, and we prove it for the element matching $\mathcal{M}_{i^+}$ using $i^+.$

First, consider an arbitrary $\sigma \cup \{1^-\} \in K$, with the supporting indices  for $X_{\sigma}$: $i_1 < i_2 < \cdots < i_{k-1}$ ($i_1 \ge 2$), for which  the smallest $j\in \{1, \ldots, k-1\}$ such that $i_j^- \in X_{\sigma}$  satisfies $i_j = i.$ By the induction hypothesis,  $\sigma \cup \{1^-\}$ has  not been matched yet, and its pair defined in (P1)/(P2) has not been matched yet. Therefore this pair is  formed precisely  in $\mathcal{M}_{i^+}$. 
  
The next fact we need to prove is that $\mathcal{M}_{i^+}$ does not match any of the faces from $\mathcal{C}.$ Consider an arbitrary  $\sigma \cup \{1^-\} \in \mathcal{C},$ $X_{\sigma} = \{i_1^+, i_1^-, i_2^-, \ldots, i_{k-1}^-\}$, $2 \le i_1< i_2 < \ldots < i_{k-1}.$ Obviously, no vertex $i^+$ can be added to  $\sigma \cup \{1^-\}$, because then $X_{\sigma \cup \{i^+\}}$ would not satisfy condition (c). The only possibility is that $\sigma \cup \{1^-\}$ is matched with $(\sigma \setminus \{i^+\}) \cup \{1^-\}$, where $i_1 < i \le i_{k-1}.$ However, by induction hypothesis, the face $(\sigma \setminus \{i^+\}) \cup \{1^-\}$ was  matched with $((\sigma \setminus \{i^+\}) \cup \{i_1^+\}) \cup \{1^-\}$ when we performed the element matching $\mathcal{M}_{i_1^+}$. Therefore, $\mathcal{M}_{i^+}$ does not match any face from $\mathcal{C}.$

Finally, we claim that $\mathcal{M}_{i^+}$ does not match any other face besides  the faces described in (P1) and (P2). Again, consider a face $\sigma \cup \{1^-\} \in K \setminus \mathcal{C}$, where $X_{\sigma}$ has supporting indices $i_1 < i_2 < \cdots < i_{k-1}$, and the smallest $j\in \{1, \ldots, k-1\}$ such that $i_j^- \in X_{\sigma}$. If $i_j < i$ then this face was already matched by the induction hypothesis, so we can assume that $i_j > i$ (for $i_j =i$ we already know how this face is paired). There are two possibilities: either $i$ is a supporting index for $X_{\sigma}$,  or not.  If $i \in \{i_1, \ldots, i_{j-1}\}$, then   $i^-  \notin X_{\sigma}$, i.e., $i^-  \in \sigma$, so  $i^+  \notin \sigma$ from condition (b). The only option would be to match $\sigma \cup \{1^-\}$ with $(\sigma \cup \{i^+\}) \cup \{1^-\}$, but this is not possible because then $X_{\sigma \cup \{i^+\}}$ would not contain a disconnected $(k-1)$-set (because $X_{\sigma \cup \{i^+\}} \cap \{i^+, i^-\} = \emptyset$). The second possibility for $i$ is that $i \notin \{i_1, \ldots, i_{j-1}\}$, i.e.,  $i$ is not a supporting index for $X_{\sigma}$. Then $i^+, i^-  \in \sigma$, so the only option would be to match $\sigma \cup \{1^-\}$ with $(\sigma \setminus \{i^+\}) \cup \{1^-\}$. But the addition of $i^+$ to the complement would increase the size of a maximal disconnected set, i.e., it would imply $d(X_{\sigma \setminus \{i^+\}}) =k$, which is not possible. 

By induction, we have proved that after the sequence $\mathcal{M}_{1^+}, \mathcal{M}_{1^-}$, and then $\mathcal{M}_{2^+}, \mathcal{M}_{3^+}, \ldots \mathcal{M}_{n^+},$ the unmatched faces are exactly the faces in $\mathcal{C}.$ There are exactly $\binom{n-1}{k-1}$ faces in $\mathcal{C}$, and each of them contains exactly $2n-k-1$ vertices.  By \cite[Appendix, Theorems~7.5]{BDJRSX-TOTAL2024}, a sequence of element matchings is an acyclic matching of  the face poset,  so we conclude that the complex $\Delta_k(G_n)$ is homotopy equivalent to a CW-complex with $\binom{n-1}{k-1}$ cells of dimension $2n-k-2$ and one additional $0$-cell \cite[Appendix, Theorems~7.2]{BDJRSX-TOTAL2024}. Consequently, $\Delta_k(G_n)\simeq\bigvee_{\binom{n-1}{k-1}} \bbS^{2n-k-2}$.
\end{proof}

The description in the proof shows that $\Delta_k(G_k)$ is isomorphic to the boundary of a $k$-dimensional crosspolytope with two opposite facets removed. 

\begin{lemma}\label{lem:Prism-over-clique-min-nonshellable}
The prism over a clique, $G_k$, is a minimal forbidden subgraph for $k$-cut complex shellability.
\end{lemma}

\begin{proof}
We must examine what happens to the cut complex when a vertex of $G_k$ is deleted. By symmetry, we may assume this vertex is $k^+$.

According to  Lemma~\ref{lem:Extendedlinks}, we have $\Delta_k(G_k \setminus \{k^+\}) = \lk_{\Delta_k(G_k)}(k^+)$. The link of $k^+$ in the $k$-dimensional crosspolytope is the $(k-1)$-dimensional crosspolytope on vertex set $\{1^+, \dotsc, (k-1)^+, 1^-, \dotsc, (k-1)^-\}$. Next, consider what happens when the facets $\{1^+, \dotsc, k^+\}$ and $\{1^-, \dotsc, k^-\}$ are deleted from the $k$-dimensional crosspolytope: the vertex $k^+$ does not appear in the second of these facets, and deleting the first of them results in the facet $\{1^+, \dotsc, (k-1)^+\}$ being removed from the link of $k^+$. Therefore, $\Delta_k(G_k \setminus \{k^+\})$ is a $(k-1)$-dimensional crosspolytope with a single facet removed.

Every polytope is shellable, so in particular the $(k-1)$-dimensional crosspolytope is shellable. By symmetry, there is a shelling order in which the facet $\{1^+, \dotsc, (k-1)^+\}$ appears last, so removing this facet from the shelling order gives us a shelling order for the crosspolytope without this facet. Thus $\Delta_k(G_k \setminus \{k^+\})$ is shellable, so $G_k$ is a minimal forbidden subgraph for $k$-cut complex shellability.
\end{proof}

\subsection{Squared Cycle Graphs}\label{sec:Wreath-MarijaDMT-Rowan-Dane-Mark}

\begin{df}  The  \emph{squared cycle graph} $W_n$ is the graph with vertex set $[n],$ and edge-set $\{ (i, i+1 \mod n), (i, i+2 \mod n)\}$,  $i=1, \dots, n$.
\end{df}
Clearly, $W_n$ contains the cycle graph $C_n$. If $n\le 5$, $W_n$ is the complete graph $K_n$. For $n\ge 6$ and $n>k+3$, the cut complex $\Delta_k(W_n)$ has dimension $n-k-1$.

\begin{prop} For $n\le k+3$,  $\Delta_k(W_n)$ is void (there are no faces) and therefore shellable. 
\end{prop}

\begin{proof} 
Clearly $\Delta_k(W_n)=\emptyset$ if $n\le k+1$.  If $n=k+2$, a separating 2-set must be  of the form $\{1,j\}$, $j\notin \{2,3,k+1, k+2\}$, so $4\le j\le k$.  Since there is an edge between $j-1$ and $j+1$, and  paths from $2$ to $j-1,$ and from $j+1$ to $k$, this is impossible.  

Let $n=k+3$.  Consider the set $S=\{1<i<j\}$.  If $j<k+3$ and $i<j-1$, then $i-1, i-2, \ldots, 2, k+3, \ldots, j+1, j-1, \ldots, i+1$ is a path in $W_n\setminus S$.  If $j=k+3$ and $2<i<k+2$, then $2,3,\ldots {i-1},{i+1},\ldots, {k+2}$ is a path in $W_n\setminus S$.  So $W_n$ does not have a separating set of size 3.
\end{proof}

\begin{prop}[{\cite[Theorem~3.11]{BDJRSX-TOTAL2024}}] \label{prop:MarijaDMT-Delta2-wreath-graph}
The $(n-3)$-dimensional cut complex $\Delta_2(W_n)$ has the homotopy type of $\bbS^{n-4}$, one sphere in dimension one lower than the top, for all $n\ge 7$. If $n=6$, $\Delta_2(W_6)$ is homotopy equivalent to $\mathbb{S}^1$.  Hence for $n\ge 6$, the cut complex $\Delta_2(W_n)$ is not shellable. 
\end{prop}

Sage computations suggest the following conjectures:

\begin{conj}  For $k\ge 3$, the 4-dimensional cut complex $\Delta_k(W_{k+5})$ has the homotopy type of $\mathbb{S}^3 \vee \bigvee_{\beta(k)} \mathbb{S}^4$ (a wedge of spheres in dimensions 3 and 4) for positive integers $\beta(k)$, and is therefore not shellable.  Its nonzero homology is $\tilde{H}_3=\mathbb{Z}, \tilde{H}_4=  \mathbb{Z}^{\beta(k)}$.
\end{conj}

\begin{conj}  For $15\ge k\ge 3$, the number of 4-spheres $\beta(k)$ in $\Delta_k(W_{k+5})$ is given by the formula 
\[\beta(k)= \dfrac{(k-3)(k-2)(k+5)}{6}.\] 
These numbers match OEIS sequence A006503.
\end{conj}

\begin{conj} For $k\ge 3$  the cut complex $\Delta_k(W_n)$ is shellable for $n\ge k+6$ (supported by Sage for $n\le 13$ and $k\le 5$).
For $k=3$ and $n\ge 9,$ the Betti numbers are 
$\binom{n-4}{2}-9 = \{1, 6, 12, 19, 27, \dots\}$.
This is OEIS A051936.
\end{conj}

Sage computations also suggest that  for $k\ge 3$, the 3-dimensional cut complex $\Delta_k(W_{k+4})$ has the homotopy type of $\bbS^1$ and is therefore not shellable.
The following results enable us to prove this.  Let $n = k+4$, and label the vertices of $W_{k+4} = W_n$ with the indices $1, \dotsc, n$, with arithmetic done modulo $n$.

\begin{lemma} \label{lemma:wreath-graph-k+4-facets}
The facets of the $3$-dimensional cut complex $\Delta_k(W_{k+4})$, for $k\ge 2$, are the sets of the form
\begin{equation*}
S = \{i, i+1, j, j+1\}
\end{equation*}
with $i$ and $j$ chosen so that $i \neq j$ and so that $i+2$ and $j+2$ are not elements of this set. In other words, $S$ consists of two pairs of consecutive vertices, $\{i, i+1\}$ and $\{j, j+1\}$, with a gap of at least one vertex between the pairs in both directions.
\end{lemma}

\begin{proof} 
First, suppose $S$ is a set of this form. Then $W_n \setminus S$ has two components, specifically $\{i+2, i+3, \dotsc, j-1\}$ and $\{j+2, \dotsc, i-1\}$: there are no edges between these components, since the gaps between these two sets are size $2$. So $S$ is a facet of $\Delta_k(W_n)$.

Conversely, suppose $S$ is any facet, so $S$ has size $4$ and $W_n \setminus S$ is disconnected into two non-empty subgraphs $U$ and $V$ with no edges between them. Let $u$ and $v$ be vertices of $U$ and $V$, respectively. Consider the list of vertices $u, u+1, \dotsc, v-1, v$, read cyclically. Suppose $x$ is the first vertex in this list that is in $V$. Then both $x-2$ and $x-1$ cannot be in $V$, and they cannot be in $U$ because $(x-2, x)$ and $(x-1,x)$ are edges in $W_n$ but by construction there are no edges between $U$ and $V$; therefore, both $x-2$ and $x-1$ must be elements of $S$. Similarly, if we consider the list $v, v+1, \dotsc, u-1, u$ and take $y$ to be the first element of $U$ in this list, then $y-2$ and $y-1$ must be elements of $S$. But since $S$ has size $4$, $S$ must be exactly $\{x-2, x-1, y-2, y-1\}$. Since $(x-2)+2 = x$ and $(y-2)+2 = y$ are not elements of $S$, as they are elements of $V$ and $U$ respectively by construction, we conclude that $S$ is a set of the desired form.
\end{proof}

A routine count shows that the number of facets of $\Delta_k(W_{k+4})$ is $(k+4) (k-1) / 2$.

\begin{prop} \label{prop:wreath-graph-k+4-homotopy} 
The complex $\Delta_k(W_{k+4})$ is homotopy equivalent to the circle $\bbS^1$.  Consequently the 3-dimensional cut complex  $\Delta_k(W_{k+4})$ is not shellable.
\end{prop}

\begin{proof} 
In view of Proposition~\ref{prop:MarijaDMT-Delta2-wreath-graph},  we need only consider the case $k\ge 3$.

Given a facet $F = \{i, i+1, j, j+1\}$ in $\Delta_k(W_{k+4})$, define its \emph{smallest gap size} to be the number
\begin{align*}
\min \left( \# \{i+2, i+3, \dotsc, j-1\}, \# \{j+2, \dotsc, i-1\} \right) & = \min \left( j - i - 2 \bmod n,\, i - j - 2 \bmod n \right)
\end{align*}
with arithmetic modulo $n$. Now, for $g = 1, \dotsc, \lfloor k/2 \rfloor$, define $X_g$ to be the complex generated by all facets of $\Delta_k(W_{k+4})$ with smallest gap size at least $g$. Note that all facets of $\Delta_k(W_{k+4})$ have smallest gap size at least $1$, so $X_1$ is $\Delta_k(W_{k+4})$ itself.

Now we claim that $X_g$ is always homotopy equivalent to $X_{g+1}$, for $g < \lfloor k/2 \rfloor$. For this  we use the method of collapsible complexes, 
see \cite[Section 11]{BjTopMeth1995} and \cite{Koz2020}. The difference between the two complexes $X_g$ and $X_{g+1}$ is the set of facets with smallest gap size exactly $g$, i.e., the facets $F_i^g = \{i, i+1, i+g+2, i+g+3\}$. Recall \cite[Chapter 9]{Koz2020} that a free face in a simplicial complex is one that is contained in a unique facet. The face $\{i+1, i+g+2\}$ is a free face in $X_g$: the only facets of $\Delta_k(W_{k+4})$ that contain it are
\begin{gather*}
    \{i, i+1, i+g+2, i+g+3\} = F_i^g, \\
    \{i+1, i+2, i+g+2, i+g+3\}, \\
    \{i, i+1, i+g+1, i+g+2\}, \\
    \{i+1, i+2, i+g+1, i+g+2\}
\end{gather*}
and all of these facets but $F_i^g$ have a smallest gap size less than $g$, so are not present in $X_g$. Therefore we can remove $\{i+1, i+g+2\}$ from $X_g$ by a sequence of elementary collapses without changing the homotopy type. The remaining faces of $F_i^g$ are $\{i, i+g+1, i+g+3\}$ and $\{i, i+1, i+g+3\}$ and their subfaces, but these two faces are contained respectively in the facets $\{i-1, i, i+g+2, i+g+3\}$ and $\{i, i+1, i+g+3, i+g+4\}$, which each have gap size $g+1$. Therefore we can remove all facets $F_i^g$ with smallest gap size $g$ from $X_g$, in any order, without changing the homotopy type, producing $X_{g+1}$.

Thus $\Delta_k(W_{k+4}) = X_1$ is homotopy equivalent to $X_{\lfloor k/2 \rfloor}$. We consider the cases where $k$ is even or odd separately.

If $k = 2r$ is even, then $X_{\lfloor k/2 \rfloor} = X_r$ consists of the facets $F_i^r = \{i, i+1, i+r+2, i+r+3\}$ for $i = 1, \dotsc, r+2$, where the gaps between the pairs of vertices are size $r$ in both directions. (Note that $F_{r+2+i}^r = \{i+r+2, i+r+3, 2r+4+i, 2r+5+i\} = F_i^r$, since $2r+4 = k+4 = n$ and arithmetic takes place modulo $n$.) Two facets $F_i^r$ and $F_j^r$ intersect each other if and only if $i$ and $j$ are consecutive modulo $r+2$; therefore, the nerve of this set of facets is the cycle graph $C_{r+2}$. By the nerve theorem \cite[Corollary~4G.3]{Hatcher2002}, $X_r$ is homotopy equivalent to the nerve of its facets; thus $\Delta_k(W_{k+4})$ is homotopy equivalent to a circle when $k$ is even.

When $k = 2r+1$ is odd, then $X_{\lfloor k/2 \rfloor} = X_r$ consists of the facets $F_i^r = \{i, i+1, i+r+2, i+r+3\}$ for $i = 1, \dotsc, n$, where the two gaps between the pairs of vertices are size $r$ and $r+1$.

We claim we can continue using free faces and elementary collapses to remove the facets $F_i^r$ where $i = 1, \dotsc, r+1$. The face $\{i+1, i+r+2\}$ is a free face in $F_i^r$, since all other facets of $\Delta_k(W_{k+4})$ that contain it have a smallest gap size less than $r$. Its elementary collapse leaves behind the faces $\{i, i+1, i+r+3\}$ and $\{i, i+r+2, i+r+3\}$, which are shared respectively with the facets $\{i, i+1, i+r+3, i+r+4\} = F_{i+r+3}^r$ and $\{i-1, i, i+r+2, i+r+3\} = F_{i+r+2}^r$, which are not in the list of facets we are removing.

This leaves the facets $F_i^r$ for $i = r+2, \dotsc, n$. This time, the nerve of these facets is the complex shown in Figure~\ref{fig:wreath-graph-k+4-homotopy-odd-case}, 
which is homotopy equivalent to a circle. Thus $\Delta_k(W_{k+4})$ is homotopy equivalent to a circle for all $k$.
\end{proof}

\begin{figure}
    \centering
    \begin{tikzpicture}[scale=1.2]
            \coordinate (fr2) at (90:1);
            \coordinate (fr3) at (45:2);
            \coordinate (fr4) at (0:2);
            \coordinate (fr5) at (315:2);
            \coordinate (dots) at (270:2);
            \coordinate (fn3) at (225:2);
            \coordinate (fn2) at (180:2);
            \coordinate (fn1) at (135:2);
            \coordinate (fn) at (90:2);
            
            \draw (fr3)--(fr4)--(fr5)--(dots)--(fn3)--(fn2)--(fn1);
            \filldraw [fill=black!20] (fn1)--(fn)--(fr2)--cycle;
            \filldraw [fill=black!20] (fr3)--(fr2)--(fn)--cycle;
            
            \foreach \i/\l in {fr2/$F_{r+2}^r$, fr3/$F_{r+3}^r$, fr4/$F_{r+4}^r$, fr5/$F_{r+3}^r$, dots/$\cdots$, fn3/$F_{n-3}^r$, fn2/$F_{n-2}^r$, fn1/$F_{n-1}^r$, fn/$F_n^r$}
                \node [fill=white, fill opacity=1, inner sep=2pt] at (\i) {\l};
    \end{tikzpicture}
    \caption{The nerve of the facets $F_{r+2}^r, \dotsc, F_n^r$ in $X_r$ in the odd case of the proof of Proposition~\ref{prop:wreath-graph-k+4-homotopy}}
    \label{fig:wreath-graph-k+4-homotopy-odd-case}
\end{figure}
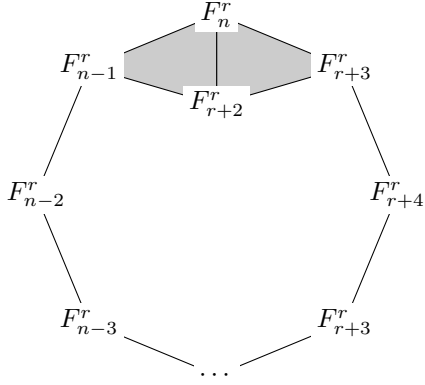

\begin{prop}\label{prop:Dane-Wreath-k+4-min-nonshellable}
Fix $ k \geq 3 $ and let $ G $ be a proper induced subgraph of $W_{k+4}$. Then $ \Delta_{k}(G) $ is
shellable.
Hence $ W_{k+4}$ is a minimal nonshellable graph for $k$-cut complex shellability. 
\end{prop}
\begin{proof}
It suffices to prove the case where $ G $ is $ W_{k+4} $ minus a single vertex.  In this case, $ G $ may be
described as the graph with vertex set $ \{a, b, 1,2, \dots, k+1\} $, and edge set \[
	\big\{\{a,b\} \{a,1\}, \{a,2\}, \{b,k+1\} , \{b,k\}\big\} \cup \big\{\{i,i+1\}\big\}_{i=1}^{k} \cup
	\big\{\{j, j+2\}\big\}_{j=1} ^{k-1}.
\] Figure \ref{fig:basket_kequals4} shows $ G $ when $ k=4 $.  Since $ G $ has $ k+3 $ vertices, the facets of
$ \Delta_{k}(G) $ are the separating sets of size 3.  Thus, from Lemma \ref{lemma:wreath-graph-k+4-facets}, we 
conclude that the facets of $ \Delta_{k}(G) $ are exactly the sets \[
		\{a,2,3\}, \dots ,\{a,k,k+1\} ,  \{b,2,3\}, \dots, \{b,k-1,k\}, \{b,1,2\} .
\]
We claim that the order of the facets as listed above gives a shelling order for $ \Delta_{k}(G) $. To see
this,
let $ F_1, \dots, F_{2k-2} $ be the facets of $ \Delta_{k+3}(G) $ listed in the order as above. Then we need
to show that for each $ i=2, \dots, 2k-2 $, the set \[
	S_i = \{F: F \subseteq F_i, \,  F \not\subseteq F_j \text{ for all } j<i\}
\] has a unique minimal element with respect to inclusion.
Indeed, for $ i \in \{2, \dots, k-1\} $ (when $ a \in F_i $), we see that the unique minimal element of $
S_i $ is $ \{i+1\}$. 
When $ i = k $, $ F_i = \{b,2,3\} $ and we see that the unique minimal element of $ S_i $ is $ \{b\} $. 
For $ i \in \{k+1, 2k-1\} $ (when $ b \in F_i $ but $ 2 \not\in F_i $), we see that the unique minimal
element of $ S_i $ is $ \{b,i+1\}$. 
Finally, when $ i = 2k-2 $, we have that the unique minimal element of $ S_i $ is $ \{1\} $.
\end{proof}
\begin{figure}[htb] 
\caption{The graph $W_8$ minus a vertex.}
		\begin{tikzpicture}[scale=0.7]%
				\node[circle,fill=black, inner sep=1.5pt] (a) at (45:2) [label={above right: $a$}] {};
				\node[circle,fill=black, inner sep=1.5pt] (b) at (135:2) [label={above left: $b$}] {};
				\node[circle,fill=black, inner sep=1.5pt] (1) at (0:2) [label={right: $1$}] {};
				\node[circle,fill=black, inner sep=1.5pt] (2) at (315:2) [label={below right: $2$}] {};
				\node[circle,fill=black, inner sep=1.5pt] (3) at (270:2) [label={below: $3$}] {};
				\node[circle,fill=black, inner sep=1.5pt] (4) at (225:2) [label={below left: $4$}] {};
				\node[circle,fill=black, inner sep=1.5pt] (5) at (180:2) [label={left: $5$}] {};
				\draw (a) -- (1) -- (2) -- (3) -- (4) -- (5) -- (b) -- (a) -- (2) -- (4) -- (b);
				\draw (1) -- (3) -- (5);
		\end{tikzpicture}
		\label{fig:basket_kequals4}
\end{figure}
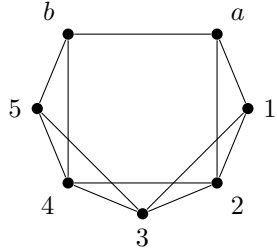
The cyclic group $\mathfrak{C}_n$  acts as a group of automorphisms of the squared cycle  $W_n$, and hence on the homology of its cut complex $\Delta_k(W_n)$.  We have the following:

\begin{prop}\label{prop:cyclic-gp-actionS-Delta_k(W_{k+4})} Let $k\ge 2$. The  one-dimensional homology module $\tilde{H}_1(\Delta_k(W_{k+4})$ affords the trivial representation of the cyclic group $\mathfrak{C}_{k+4}$.

If $n\ge 7$, the cyclic group $\mathfrak{C}_n$ acts on  the one-dimensional homology module $\tilde{H}_{n-4}(\Delta_2(W_{n}))$ like 
\[\begin{cases} \text{the  trivial representation}, &\text{if $n$ is odd},\\
\text{the  sign representation}, &\text{if $n$ is even}.
\end{cases}\]
\end{prop}

\begin{proof} We apply the Hopf trace formula \cite{Hatcher2002}   to the face lattice of the cut complex. In the present context, the precise fact that we need is as follows  \cite{Jer93}. Suppose the nonzero homology of a bounded poset $P$ is concentrated in a single degree $r-2$, and suppose $g$ is an automorphism of $P$.  Let  $P^g$ denote the subposet of $P$ consisting of all elements fixed by $g$. Then one has the formula \cite[p. 282, Eqn. (1.2)]{Jer93}
\[\mu(P^g)=(-1)^r \mathrm{tr} (g, \tilde{H}_{r-2}(P)).\]

Now take $P$ to be the face lattice of $\Delta_k({{W}_{k+4}})$. From Proposition~\ref{prop:wreath-graph-k+4-homotopy} the homology is concentrated in degree $1$, and has vector space dimension one. In order to determine this one-dimensional representation of ${\mathfrak{C}_{k+4}}$, it suffices to compute the trace of the $(k+4)$-cycle $g=(1,2,\ldots, k+4)$ which generates ${\mathfrak{C}_{k+4}}$.  But the fixed point subposet $P^g$ is clearly the trivial poset consisting of $\{\hat 0, \hat 1\}$, and thus $\mu(P^g)=-1=(-1)^3\mathrm{tr} (g, \tilde{H}_{1}(P)) $. 

Hence $\mathrm{tr} (g, \tilde{H}_{1}(P))=1$, 
confirming that the action of the cyclic group $\mathfrak{C}_{k+4}$ on the homology is trivial. 

Now let $n\ge 7.$ Proposition~\ref{prop:MarijaDMT-Delta2-wreath-graph} tells us that the nonzero homology of $\Delta_2(W_{n})$ occurs only in degree $n-4$, and has dimension 1 as a vector space. By the Hopf trace formula above applied to the face lattice $P$ of $\Delta_2(W_{n})$, for the $n$-cycle $g=(1,2,\ldots,n)\in\mathfrak{C}_n$, we have 
\[\mu(P^g)=(-1)^{n-2} \mathrm{tr} (g, \tilde{H}_{n-4}(P)).\]
Again it is easy to see that $P^g=\{\hat 0, \hat 1\}$. This time we obtain 
$ \mathrm{tr} (g, \tilde{H}_{n-4}(P))=(-1)^{n-1}$.  The claim follows.
\end{proof}

\section{Conclusion and Further Directions}

In this paper we introduced a new graph complex, the $k$-cut complex for $k > 2$, which generalizes the (2-cut) complex in the Eagon--Reiner proof of Fr\"oberg's Theorem.
We investigated how shellability and homotopy type of the $k$-cut complex are affected by the following graph operations: induced subgraphs, disjoint union, joins and wedges. We were able to extend one direction of Fr\"oberg's result for chordal graphs and the 2-cut complex  to the 3-cut complex. Our results for $k=3$ are best possible: we showed that   for any $k\ge 4$, there are examples of chordal graphs for which $\Delta_k(G)$ is not shellable. We also studied the face lattice of the cut complex, giving a formula for the reduced Euler characteristic for a broad family of graphs.  We completely determined the homotopy type of the 2-cut complex in the case of connected triangle-free graphs.

The families of graphs we considered include trees, complete multipartite graphs, cycles, prisms over cliques, and squared cycles.  In all cases except the latter, we determined completely the homotopy type of the $k$-cut complex.  Our tools encompassed  a broad range: shellability, poset topology and discrete Morse theory.

We continue the investigation of $k$-cut complexes in a subsequent paper, where we apply these methods to the families of grid graphs, and squared paths.   We also undertake a more detailed study of how   the $k$-cut complex behaves under the disjoint union operation  of graphs, including an analysis of  the face vectors and $h$-vectors.


\bibliographystyle{plain}
\bibliography{2024Jan29ArXivSIDMAResultsFINAL}

\end{document}